\documentclass[10pt]{article}
\pdfoutput=1
\usepackage{xcolor}
\usepackage{amsmath}
\usepackage{amssymb}
\usepackage{amsthm}
\usepackage{dsfont}
\usepackage{fullpage}
\usepackage{graphics}
\usepackage{graphicx}
\usepackage{subcaption}
\usepackage{caption}
\usepackage[font=footnotesize,labelfont=bf]{caption}
\usepackage{mathtools}
\usepackage{mathrsfs}
\usepackage{bbm}
\usepackage{multicol}
\usepackage[shortlabels]{enumitem}
\usepackage{pgfplots}
\pgfplotsset{compat=newest}

\usepackage{tikz}
\usetikzlibrary{patterns}
\usetikzlibrary{decorations.pathmorphing}
\usetikzlibrary{decorations.pathreplacing,decorations.markings}
\usetikzlibrary{shapes}
\makeatletter
\def\grd@save@target#1{%
	\def\grd@target{#1}}
\def\grd@save@start#1{%
	\def\grd@start{#1}}
\tikzset{
	grid with coordinates/.style={
		to path={%
			\pgfextra{%
				\edef\grd@@target{(\tikztotarget)}%
				\tikz@scan@one@point\grd@save@target\grd@@target\relax
				\edef\grd@@start{(\tikztostart)}%
				\tikz@scan@one@point\grd@save@start\grd@@start\relax
				\draw[minor help lines] (\tikztostart) grid (\tikztotarget);
				\draw[major help lines] (\tikztostart) grid (\tikztotarget);
				\grd@start
				\pgfmathsetmacro{\grd@xa}{\the\pgf@x/1cm}
				\pgfmathsetmacro{\grd@ya}{\the\pgf@y/1cm}
				\grd@target
				\pgfmathsetmacro{\grd@xb}{\the\pgf@x/1cm}
				\pgfmathsetmacro{\grd@yb}{\the\pgf@y/1cm}
				\pgfmathsetmacro{\grd@xc}{\grd@xa + \pgfkeysvalueof{/tikz/grid with coordinates/major step}}
				\pgfmathsetmacro{\grd@yc}{\grd@ya + \pgfkeysvalueof{/tikz/grid with coordinates/major step}}
				\foreach \x in {\grd@xa,\grd@xc,...,\grd@xb}
				\node[anchor=north] at (\x,\grd@ya) {\pgfmathprintnumber{\x}};
				\foreach \y in {\grd@ya,\grd@yc,...,\grd@yb}
				\node[anchor=east] at (\grd@xa,\y) {\pgfmathprintnumber{\y}};
			}
		}
	},
	minor help lines/.style={
		help lines,
		step=\pgfkeysvalueof{/tikz/grid with coordinates/minor step}
	},
	major help lines/.style={
		help lines,
		line width=\pgfkeysvalueof{/tikz/grid with coordinates/major line width},
		step=\pgfkeysvalueof{/tikz/grid with coordinates/major step}
	},
	grid with coordinates/.cd,
	minor step/.initial=.2,
	major step/.initial=1,
	major line width/.initial=0.25mm,
}

\pgfdeclarepatternformonly{north east lines wide}%
{\pgfqpoint{-1pt}{-1pt}}%
{\pgfqpoint{10pt}{10pt}}%
{\pgfqpoint{5pt}{5pt}}%
{
	\pgfsetlinewidth{.8pt}
	\pgfpathmoveto{\pgfqpoint{0pt}{0pt}}
	\pgfpathlineto{\pgfqpoint{9.1pt}{9.1pt}}
	\pgfusepath{stroke}
}

\tikzset{
	on each segment/.style={
		decorate,
		decoration={
			show path construction,
			moveto code={},
			lineto code={
				\path [#1]
				(\tikzinputsegmentfirst) -- (\tikzinputsegmentlast);
			},
			curveto code={
				\path [#1] (\tikzinputsegmentfirst)
				.. controls
				(\tikzinputsegmentsupporta) and (\tikzinputsegmentsupportb)
				..
				(\tikzinputsegmentlast);
			},
			closepath code={
				\path [#1]
				(\tikzinputsegmentfirst) -- (\tikzinputsegmentlast);
			},
		},
	},
	mid arrow/.style={postaction={decorate,decoration={
				markings,
				mark=at position .5 with {\arrow[#1]{stealth}}
	}}},
	rmid arrow/.style={postaction={decorate,decoration={
				markings,
				mark=at position .5 with {\arrowreversed[#1]{stealth}}
	}}},
	end arrow/.style={postaction={decorate,decoration={
				markings,
				mark=at position 1 with {\arrow[#1]{stealth}}
	}}},
	start arrow/.style={postaction={decorate,decoration={
				markings,
				mark=at position 0 with {\arrow[#1]{stealth}}
	}}},
	mid3 arrow/.style={postaction={decorate,decoration={
				markings,
				mark=at position .3 with {\arrow[#1]{stealth}}
	}}},
	rmid3 arrow/.style={postaction={decorate,decoration={
				markings,
				mark=at position .7 with {\arrowreversed[#1]{stealth}}
	}}},
	mid4 arrow/.style={postaction={decorate,decoration={
				markings,
				mark=at position .4 with {\arrow[#1]{stealth}}
	}}},
	rmid4 arrow/.style={postaction={decorate,decoration={
				markings,
				mark=at position .4 with {\arrowreversed[#1]{stealth}}
	}}},
}

\makeatother
\tikzset{every state/.style={minimum size=0pt}}

\tikzset{
	mark position/.style args={#1(#2)}{
		postaction={
			decorate,
			decoration={
				markings,
				mark=at position #1 with \coordinate (#2);
			}
		}
	}
}

\usetikzlibrary{3d}
\usepackage{xcolor}
\usepackage{xifthen}

\usetikzlibrary{decorations}

\pgfdeclaredecoration{ignore}{final}
{
	\state{final}{}
}

\pgfdeclaremetadecoration{middle}{initial}{
	\state{initial}[
	width={0pt},
	next state=middle
	]
	{\decoration{moveto}}
	
	\state{middle}[
	width={\pgfdecorationsegmentlength*\pgfmetadecoratedpathlength},
	next state=final
	]
	{\decoration{curveto}}
	
	\state{final}
	{\decoration{ignore}}
}

\tikzset{middle segment/.style={decoration={middle},decorate, segment length=#1}}

\makeatletter
\renewcommand\paragraph{\@startsection{paragraph}{4}{\z@}%
	{-2.5ex\@plus -1ex \@minus -.25ex}%
	{1.25ex \@plus .25ex}%
	{\normalfont\normalsize\bfseries}}
\makeatother
\newtheorem{thm}{Theorem}
\newtheorem{lm}[thm]{Lemma}
\newtheorem{defn}[thm]{Definition}
\newtheorem{prop}[thm]{Proposition}
\newtheorem{rmk}[thm]{Remark}

\newcommand{\Ch}{\mathrm{C}}

\newcommand{\complexC}{\mathbb{C}}
\newcommand{\conf}{\mathcal{X}}

\newcommand{\dd}{{\mathrm d}}
\newcommand{\ddbar}[1]{\frac{{\mathrm d}#1}{2\pi {\mathrm i}#1}}
\newcommand{\ddbarr}[1]{\frac{{\mathrm d}#1}{2\pi {\mathrm i}}}
\newcommand{\disk}{\mathbb{D}}

\newcommand{\GG}{G}
\newcommand{\hh}{h}
\newcommand{\ich}{\chi}
\newcommand{\ii}{\mathrm{i}}
\newcommand{\iich}{\mathrm{ch}}

\newcommand{\inn}{\mathrm{in}}
\newcommand{\intZ}{\mathbb{Z}}
\newcommand{\Kess}{\mathcal{K}^{(\mathrm{ess})}}
\newcommand{\Knull}{\mathcal{K}^{(\mathrm{null})}}
\newcommand{\LL}{\mathrm{L}}
\newcommand{\mm}{{\max{}}}

\newcommand{\out}{\mathrm{out}}
\newcommand{\period}{{(L)}}

\newcommand{\prob}{\mathbb{P}}
\newcommand{\realR}{\mathbb{R}}
\newcommand{\roots}{\mathcal{R}}
\newcommand{\RR}{\mathrm{R}}

\renewcommand{\Re}{\mathrm{Re}}

\numberwithin{equation}{section} 
\numberwithin{thm}{section}
	
\author{
		Zhipeng Liu\footnote{Department of Mathematics, University of Kansas, Lawrence, KS 66045. Email: \texttt{zhipeng@ku.edu}}}
\date{\today}
	
\begin{document}
		\title{Multi-point distribution of TASEP}
		
		\date{\today}
		\maketitle
		
		\begin{abstract}
			Recently Johansson and Rahman obtained the limiting multi-time distribution for the discrete polynuclear growth model \cite{Johansson-Rahman19}, which is equivalent to a discrete TASEP model with step initial condition. In this paper, we obtain a finite time multi-point distribution formula of continuous TASEP with general initial conditions in the space-time plane. We evaluate the limit of this distribution function when the times go to infinity at the same speed for both step and flat initial conditions. These limiting distributions are expected to be universal for all the models in the Kardar-Parisi-Zhang universality class.
		\end{abstract}

\section{Introduction}

In the recent twenty years, there has been an explosive development in understanding the universal law behind a family of $2$d random growth models \cite{Baik-Deift-Johansson99,Johansson00,Johansson03,Borodin-Ferrari-Prahofer-Sasamoto07,Tracy-Widom08,Tracy-Widom09,Borodin-Corwin14,Matetski-Quastel-Remenik17,Johansson17,Dauvergne-Ortmannn-Virag18,Johansson-Rahman19}. There is a growing number of models which are either proved or believed to be in the so-called Kardar-Parisi-Zhang (KPZ) universality class. All of these models share the scaling limits $t:t^{2/3}:t^{1/3}$ for the time, spatial correlation length and fluctuation order. Moreover, the scaled limiting space-time field is believed to be universal and independent of the models, but only depends on the initial condition:
\begin{equation}
\label{eq:KPZ_limit}
\lim_{T\to\infty}\frac{H(c_1xT^{2/3},c_2\tau T)-c_3 \tau T}{ c_4T^{1/3}}=\mathrm{H}(x,\tau).
\end{equation}
Here $c_1,c_2,c_3$ and $c_4$ are model-dependent constants, $H(y,t)$ is the height function of the growth model at location $y$ and time $t$, and $\mathrm{H}(x,\tau)$ is  the limiting space-time field depending only on the initial condition. This limiting field $\mathrm{H}(x,\tau)$ is believed to be universal. It was first characterized by Matetski, Quastel and Remenik \cite{Matetski-Quastel-Remenik17} as a Markov process with explicit transition probabilities and variational formulas by analyzing the totally asymmetric simple exclusion process (TASEP). It could also be characterized by the so-called directed landscape which was constructed by Dauvergne, Ortmann and Vir\'ag \cite{Dauvergne-Ortmannn-Virag18}  more recently in the context of Brownian last
passage percolation. We remark that the characterization in \cite{Dauvergne-Ortmannn-Virag18} does not imply explicit formulas (like the ones we consider here) for the distribution functions of $\mathrm{H}(x,\tau)$. Understanding the limiting field $\mathrm{H}(x,\tau)$ is a fundamental problem in the community.

It has been shown that, for a number of models in the KPZ universality class, the one point distributions of $\mathrm{H}(x,\tau)$ are given by the Tracy-Widom distributions and their analogs. See \cite{Baik-Deift-Johansson99,Johansson00,Tracy-Widom09,Amir-Corwin-Quastel11,Borodin-Corwin-Ferrari14,Aggarwal18} for the standard initial conditions and \cite{Corwin-Liu-Wang16,Chhita-Ferrari-Spohn18,Quastel-Remenik19} for general initial conditions. We refer the readers to a review paper \cite{Corwin11}.

The spatial process $\mathrm{H}(x,\tau)$ when $\tau$ is fixed, is only obtained for TASEP and its equivalent models. See \cite{Prahofer-Spohn02,Johansson03,Imamura-Sasamoto04,Borodin-Ferrari-Prahofer-Sasamoto07,Borodin-Ferrari-Prahofer07,Borodin-Ferrari-Sasamoto08a,Baik-Ferrari-Peche10} for the standard initial conditions and \cite{Matetski-Quastel-Remenik17} for general initial condition. We also refer the readers to a review paper \cite{Quastel-Remenik13} for the limiting processes.

Along the time direction, or more generally in the space-time field  $\mathrm{H}(x,\tau)$, much less was known until recently. For a standard initial condition, the so-called step initial condition, the two-point distribution along the time direction was obtained by \cite{Johansson17,Johansson19} for Brownian directed percolation and geometric last-passage percolation, and very recently, the multi-point distribution along the time direction was also found by \cite{Johansson-Rahman19} for the same geometric last-passage percolation model. We remark that  the geometric last-passage percolation model is equivalent to a discrete TASEP. Besides these distribution formulas, there are also some results on the properties of $\mathrm{H}(x,\tau)$ at two different times, see \cite{Ferrari-Spohn16, DeNardis-LeDoussal17,LeDoussal17,DeNardis-LeDoussal18,Ferrari-Occelli19,Johansson20,Corwin-Ghosal-Hammond21}.

\bigskip

 Parallelly, in the line of research \cite{Baik-Liu16, Liu16, Baik-Liu19, Baik-Liu21}, the authors studied the continuous TASEP on a periodic domain (periodic TASEP). They obtained the finite time multi-point distributions of the height function in the space-time plane, and their limits in the so-called relaxation time scale. 
 Since the periodic TASEP becomes the usual TASEP on $\intZ$ when the period tends to infinity, it is expected that their results, after taking the large period limit or equivalently the small time limit, should give the limiting multi-point distributions of $\mathrm{H}(x,\tau)$ for TASEP.  
 However, it seems quite complicated to obtain  the TASEP limits using asymptotic analysis directly from the formulas in \cite{Baik-Liu19,Baik-Liu21}. The multi-point distribution formulas involve contour integrals of a complicated Fredholm determinant which is defined on a discrete space (in terms of the so-called Bethe roots). The classic steepest descent method seems not working well: there are terms of large contributions in the integrand, which combined together are expected to cancel out when evaluated via the outside contour integrals. It is still unclear how to manage these cancellations.

 \bigskip
 
 This paper can be viewed as an extension of the work \cite{Baik-Liu19,Baik-Liu21}. Instead of performing asymptotic analysis, we rewrite the algebraic structure of their finite time multi-point distribution formulas when the period is finite but larger than a fixed number. This rewriting is constructive: We construct a new formula for contour integrals whose integrand is a type of summation over nested roots of functions satisfying certain conditions, and prove the new formula by induction.
 
 \bigskip

 The main results of this paper are as follows. 
 \begin{enumerate}[1)]
 	\item We obtain the finite time multi-point distribution of TASEP in the space-time plane. See Theorem~\ref{thm:main1}. This result generalizes the well-known multi-point distribution of TASEP along the space direction \cite{Borodin-Ferrari-Prahofer-Sasamoto07}.
 	\item For two specific initial conditions, the step and flat initial conditions, we evaluate the limit of the above multi-point distributions when the times go to infinity proportionally. See Theorem~\ref{thm:limit_step} and~\ref{thm:limit_flat}. These formulas are expected to be the multi-point distributions of the universal field $\mathrm{H}(x,\tau)$ in~\eqref{eq:KPZ_limit} for the step and flat initial conditions.
 \end{enumerate}

We remark that our formula of the multi-point distribution of TASEP for the step initial condition is different from that in~\cite{Johansson-Rahman19} of geometric last-passage percolation. We expect that, when the times are different, our formula matches theirs. But we do not have a rigorous proof at the moment due to the complexity of both formulas. 
\bigskip

Below is the organization of this paper.
 
  In Section~\ref{sec:main_results}, we present the multi-point distribution formula of TASEP in Theorem~\ref{thm:main1}, and the limiting multi-point distributions for step and flat initial conditions in Theorems~\ref{thm:limit_step} and~\ref{thm:limit_flat}. We also discuss some properties of the finite time distribution formula in Section~\ref{sec:properties_D}.
 
 In Section~\ref{sec:pTASEP}, we introduce the periodic TASEP model. We claim that the multi-point distributions for periodic TASEP, when the period is larger than a finite number, can be expressed as the same formula for TASEP in Theorem~\ref{thm:main1}. See Theorem~\ref{thm:periodic_TASEP_large_period}. Therefore Theorem~\ref{thm:main1} follows. 
 
 In Section~\ref{sec:Cauchy_sum}, we extract the key part in the proof of Theorem~\ref{thm:periodic_TASEP_large_period}. We investigate a type of summation, which we call Cauchy-type summation, over a set of nested roots of certain functions. The main result of this section is given in Proposition~\ref{prop:Cauchy_sum_over_roots}, which is also the main technical part of the paper.
 
 The remaining sections are the proofs. Section~\ref{sec:proof_Theorem} is the proof of Theorem~\ref{thm:periodic_TASEP_large_period} by using the results of Section~\ref{sec:Cauchy_sum}. Section~\ref{sec:proof_Cauchy_sum} is the proof of Proposition~\ref{prop:Cauchy_sum_over_roots}. Section~\ref{sec:asymptotics} is the only section involving the asymptotic analysis. It includes the proof of Theorems~\ref{thm:limit_step} and~\ref{thm:limit_flat}. Finally in Section~\ref{sec:proofs} we prove some properties of the finite time multi-point distributions discussed in Section~\ref{sec:main_results}.

\section*{Acknowledgement}

The author would like to thank Jinho Baik for the communications, and hosting the author's recent visit to the University of Michigan where they had many useful discussions on this paper. The author also would like to thank the Integrable
Probability Focused Research Group, funded by NSF grants DMS-1664531, 1664617, 1664619,
1664650, for their organizations on various research activities in integral probability, and their support for the author's participation in these activities. The author was supported by the University of Kansas Start Up Grant, the University of Kansas New Faculty General Research Fund, Simons Collaboration Grant No. 637861, and NSF grant DMS-1953687. 
\section{Main results}
\label{sec:main_results}

We consider the totally asymmetric simple exclusion process (TASEP) on the infinite lattice $\intZ$. Each site on $\intZ$ allows at most one particle. The evolution of the system is as follows. Each particle is assigned an independent clock which rings after an exponential waiting time with parameter $1$. Once its assigned clock rings, the particle either moves to its  right neighboring site if that site is unoccupied, or stays on its current site if its right neighboring site is occupied. Meanwhile the clock is reset.

We assume that initially there are $N$ particles and they are labeled from right to left. The location of the $i$-th particle at time $t$ is denoted by $x_i(t)$. We denote $X(t):=(x_1(t),\cdots,x_N(t))$ the configuration of particle locations at time $t$ for any $t\ge 0$. We also denote $\conf_N$ the set of all possible configurations
\begin{equation*}
\conf_N := \left\{ (x_1,\cdots,x_N) \in \intZ^N:  x_1>\cdots>x_N \right\}.
\end{equation*}
Then $X(t) \in \conf_N$ for all $t\ge 0$. We also denote $Y=(y_1,\cdots,y_N)$ the initial configuration
\begin{equation*}
y_i = x_i(0),\qquad i=1,\cdots,N.
\end{equation*}

\subsection{Multi-point distribution of TASEP with general initial configuration}
The main result in this paper is about the multi-point distribution of TASEP.
\begin{thm}
	\label{thm:main1}
	Assume $Y=(y_1,\cdots,y_N)\in\conf_N$. Consider TASEP with initial particle locations $x_i(0)=y_i$ for $1\le i\le N$. Let $m\ge 1$ be a positive integer and $(k_1,t_1),\cdots, (k_m,t_m)$ be $m$ distinct points in $\{1,\cdots,N\}\times [0,\infty)$. Assume that $0\le t_1\le\cdots\le t_m$. Then, for any integers $a_1,\cdots,a_m$,
	\begin{equation}
	\label{eq:thm1}
	\prob_Y \left( 
	\bigcap_{\ell=1}^m \left\{ x_{k_\ell} (t_\ell) \ge a_\ell 
	\right\}
	\right) = \oint\cdots\oint \left[\prod_{\ell=1}^{m-1}\frac{1}{1-z_\ell}\right] \mathcal{D}_Y(z_1,\cdots,z_{m-1}) \ddbar{z_1}\cdots\ddbar{z_{m-1}},
	\end{equation}
	where $\prob_Y$ denotes the probability given $X(0)=Y$, the contours of integration are counterclockwise circles centered at the origin and of radii less than $1$. The function $\mathcal{D}_Y(z_1,\cdots,z_{m-1})$ is defined in terms of a Fredholm determinant in Definition~\ref{def:operators_K1Y}, or equivalently in terms of series expansion in Definition~\ref{def:D_Y_series}.
\end{thm}

\begin{rmk}
	\label{rmk:01}
	We expect that when all $t_\ell$'s are equal, the above formula  
	matches the joint distribution formula in \cite{Borodin-Ferrari-Prahofer-Sasamoto07}. For $m=1$, we are able to confirm it in a formal way. See the discussions in Section~\ref{sec:one_point}. We leave the general case for a possible future project.
\end{rmk}

The proof follows directly from Theorem~\ref{thm:comparison} and Theorem~\ref{thm:periodic_TASEP_large_period}.

\vspace{0.3cm}

It turns out that the right hand side of~\eqref{eq:thm1} is still a probability distribution function up to a sign, if we assume some $z_\ell$ circles are of radii greater than $1$. More precisely, we have

\begin{prop}
	\label{prop:consistency}
	Assume the same setting with Theorem~\ref{thm:main1}. Let  $I$ be any subset of $\{1,\cdots,m-1\}$, and $J=\{1,\cdots,m\}\setminus I$. Then, for any integers $a_1,\cdots,a_m$,
	\begin{equation}
	\label{eq:thm2}
	\begin{split}
	&\prob_Y \left( 
	\left(\bigcap_{j\in J} \left\{ x_{k_j} (t_j) \ge a_j 
	\right\}
	\right)\bigcap
	\left(\bigcap_{i\in I} \left\{ x_{k_i} (t_i) < a_i 
	\right\}\right)
	\right) \\
	&= (-1)^{|I|}\oint\cdots\oint \left[\prod_{\ell=1}^{m-1}\frac{1}{1-z_\ell}\right] \mathcal{D}_Y(z_1,\cdots,z_{m-1}) \ddbar{z_1}\cdots\ddbar{z_{m-1}},
	\end{split}
	\end{equation}
	where the contours of integration are counterclockwise circles centered at the origin. The radius of $z_\ell$ contour is smaller than $1$ if $\ell\in J$, and greater than $1$ if $\ell\in I$. The function $\mathcal{D}_Y(z_1,\cdots,z_{m-1})$ is defined in terms of a Fredholm determinant in Definition~\ref{def:operators_K1Y}, or equivalently in terms of series expansion in Definition~\ref{def:D_Y_series}.
\end{prop}

The proof of Proposition~\ref{prop:consistency} is given in Section~\ref{sec:proof_consistency}.

\vspace{0.3cm}

Below we first introduce the Fredholm determinant representation of $\mathcal{D}_Y$ in Section~\ref{sec:Fredholm_representation}. In Section~\ref{sec:series_expansion} we will give an alternate formula of $\mathcal{D}_Y$ in terms of a series expansion. Finally, in Section~\ref{sec:properties_D} we will discuss some further properties of the function $\mathcal{D}_Y$.

\subsubsection{Fredholm determinant representation of $\mathcal{D}_Y(z_1,\cdots,z_{m-1})$}
\label{sec:Fredholm_representation}

We will define  $\mathcal{D}_Y(z_1,\cdots,z_{m-1})$ as a Fredholm determinant $\det(I-\mathcal{K}_1\mathcal{K}_Y)$. Such Fredholm determinant representation is not unique. There are different choices of the spaces, measures, and kernels. We will see this fact later in Section~\ref{sec:properties_D}. At this moment, we choose a specific choice of spaces, measures and kernels for the Fredholm determinant representation.

\paragraph{Spaces of the operators}
\label{sec:spaces_operators}

We will define the operators on two specific spaces of nested contours with complex measures depending on $\boldsymbol{z}=(z_1,\cdots,z_{m-1})$, where $z_\ell\ne 1$ for each $1\le \ell\le m-1$.

Suppose $\Omega_\LL$ and $\Omega_\RR$ are two simply connected regions on the complex plane such that (1) $\Omega_\LL$ contains the point $-1$, (2) $\Omega_\RR$ contains the point $0$, and (3) $\Omega_\LL$ and $\Omega_\RR$ do not intersect.

Suppose $\Sigma_{m,\LL}^{\out},\cdots,\Sigma_{2,\LL}^{\out}$, $\Sigma_{1,\LL}$, $\Sigma_{2,\LL}^{\inn},\cdots,\Sigma_{m,\LL}^\inn$ are $2m-1$ nested simple closed contours, from outside to inside, in $\Omega_\LL$ enclosing the point $-1$. Similarly,
$\Sigma_{m,\RR}^{\out},\cdots,\Sigma_{2,\RR}^{\out}$, $\Sigma_{1,\RR}$, $\Sigma_{2,\RR}^{\inn},\cdots,\Sigma_{m,\RR}^\inn$ are $2m-1$ nested simple closed contours, from outside to inside, in $\Omega_\RR$ enclosing the point $0$.  See Figure~\ref{fig:contours_finite_time} for an illustration of the contours. These contours are all counterclockwise oriented. In fact, throughout this paper, all closed contours will be counterclockwise oriented and we will not emphasize the orientations later. However, we will clearly state the orientations for the infinite contours.

	\begin{figure}[t]
	\centering
	\begin{tikzpicture}[scale=1]
	\draw [line width=0.4mm,lightgray] (-4.5,0)--(1.5,0) node [pos=1,right,black] {$\realR$};
	\draw [line width=0.4mm,lightgray] (0,-1.5)--(0,1.5) node [pos=1,above,black] {$i\realR$};
	\fill (0,0) circle[radius=2.5pt] node [below,shift={(0pt,-3pt)}] {$0$};
	\fill (-3,0) circle[radius=2.5pt] node [below,shift={(0pt,-3pt)}] {$-1$};
	
	\draw [draw=black] (-1.3,-1.3) rectangle (1.4,1.3);
		\draw [draw=black] (-4.4,-1.5) rectangle (-1.5,1.5);
    \draw[thick, dashed] (0,0) circle (18pt);
    \draw[thick] (0,0) circle (25pt);
    \draw[thick, dashed] (0,0) circle (30pt);

	\draw[thick] (-3,0) circle (20pt);
	\draw[thick, dashed] (-3,0) circle (27pt);
	\draw[thick] (-3,0) circle (34pt);
	\end{tikzpicture}
	\caption{Illustration of the contours for $m=2$: The two rectangles from left to right are $\Omega_\LL$ and $\Omega_\RR$, the three contours around $-1$ from outside to inside are $\Sigma_{2,\LL}^{\out},\Sigma_{1,\LL},\Sigma_{2,\LL}^{\inn}$ respectively, the three contours around $0$ from outside to inside are $\Sigma_{2,\RR}^{\out},\Sigma_{1,\RR},\Sigma_{2,\RR}^{\inn}$ respectively. $\mathcal{S}_1$ is the union of the three dashed contours, and $\mathcal{S}_2$ is the union of the three solid contours.}\label{fig:contours_finite_time}
\end{figure}

We define
\begin{equation}
	\label{eq:Sigma_contours}
\Sigma_{\ell,\LL}:=\Sigma_{\ell,\LL}^\out\cup \Sigma_{\ell,\LL}^\inn, \qquad \Sigma_{\ell,\RR}:=\Sigma_{\ell,\RR}^\out\cup \Sigma_{\ell,\RR}^\inn,\qquad \ell=2,\cdots,m,
\end{equation}
and
\begin{equation*}
\mathcal{S}_1:= \Sigma_{1,\LL} \cup \Sigma_{2,\RR} \cup \cdots \cup \begin{dcases}
\Sigma_{m,\LL}, & \text{ if $m$ is odd},\\
\Sigma_{m,\RR}, & \text{ if $m$ is even},
\end{dcases}
\end{equation*}
and 
\begin{equation*}
\mathcal{S}_2:= \Sigma_{1,\RR} \cup \Sigma_{2,\LL} \cup \cdots \cup \begin{dcases}
\Sigma_{m,\RR}, & \text{ if $m$ is odd},\\
\Sigma_{m,\LL}, & \text{ if $m$ is even}.
\end{dcases}
\end{equation*}

We also introduce a measure on these contours. Let
\begin{equation}
\label{eq:def_dmu}
\dd \mu(w) = \dd \mu_{\boldsymbol{z}} (w) :=
 \begin{dcases}
\frac{-z_{\ell-1}}{1-z_{\ell-1}} \ddbarr{w}, & w\in \Sigma_{\ell,\LL}^\out \cup \Sigma_{\ell,\RR}^\out, \quad \ell=2,\cdots,m,\\
\frac{1}{1-z_{\ell-1}} \ddbarr{w}, & w\in \Sigma_{\ell,\LL}^\inn \cup \Sigma_{\ell,\RR}^\inn, \quad \ell=2,\cdots,m,\\
\ddbarr{w}, & w\in \Sigma_{1,\LL} \cup \Sigma_{1,\RR}.
\end{dcases}
\end{equation}

\paragraph{Operators $\mathcal{K}_1$ and $\mathcal{K}_Y$}
\label{sec:def_operators}

Now we introduce the operators $\mathcal{K}_1$ and $\mathcal{K}_Y$ to define $\mathcal{D}_Y(z_1,\cdots,z_{m-1})$ in Theorem~\ref{thm:main1}. We assume that $Y=(y_1,\cdots,y_N)\in\conf_N$ and  $\boldsymbol{z}=(z_1,\cdots,z_{m-1})$ is the same as in Section~\ref{sec:spaces_operators}. Let
\begin{equation}
\label{eq:Q}
Q_1(j) :=\begin{dcases}
1-z_{j}, & \text{ if $j$ is odd and $j<m$},\\
1-\frac{1}{z_{j-1}}, & \text{ if $j$ is even},\\
1,& \text{if $j=m$ is odd},
\end{dcases}
\qquad 
Q_2(j) :=\begin{dcases}
1-z_{j}, & \text{ if $j$ is even and $j<m$},\\
1-\frac{1}{z_{j-1}}, & \text{ if $j$ is odd and $j>1$},\\
1,& \text{if $j=m$ is even, or $j=1$}.
\end{dcases}
\end{equation}

\begin{defn}
	\label{def:operators_K1Y}
	We define
	\begin{equation*}
	\mathcal{D}_Y(z_1,\cdots,z_{m-1})=\det\left( I - \mathcal{K}_1  \mathcal{K}_Y \right),
	\end{equation*}
	where two operators
	\begin{equation*}
	\mathcal{K}_1: L^2(\mathcal{S}_2,\dd\mu) \to L^2(\mathcal{S}_1,\dd\mu),\qquad \mathcal{K}_Y: L^2(\mathcal{S}_1,\dd\mu)\to L^2(\mathcal{S}_2,\dd\mu)
	\end{equation*}
	are defined by their kernels
	\begin{equation}
	\label{eq:K1}
	\mathcal{K}_1(w,w'):= \left(\delta_i(j) + \delta_i( j+ (-1)^i)\right) \frac{ f_i(w) }{w-w'} Q_1(j),
	\end{equation}
	and
	\begin{equation}
	\label{eq:KY}
	\mathcal{K}_Y(w',w):= \begin{dcases}
	\left(\delta_j (i) + \delta_j(i - (-1)^j)\right) \frac{ f_j(w') }{w'-w} Q_2(i), & i\ge 2,\\
	\delta_j(1)f_j(w')\Kess_Y(w';w), & i=1,
	\end{dcases}
	\end{equation}
	for any $w\in (\Sigma_{i,\LL}\cup \Sigma_{i,\RR}) \cap \mathcal{S}_1$ and $w'\in (\Sigma_{j,\LL}\cup \Sigma_{j,\RR}) \cap \mathcal{S}_2$ with $1\le i,j\le m$. Here $\Kess_Y$ is a kernel defined in Definition~\ref{def:KY_ess}. The function
	\begin{equation}
	\label{eq:fi}
	f_i(w):=\begin{dcases}
	\frac{F_i(w)}{ F_{i-1}(w)}, & w\in \Omega_\LL\setminus\{-1\},\\
	\frac{F_{i-1}(w)}{ F_i(w)}, & w\in \Omega_\RR\setminus\{0\},
	\end{dcases}
	\end{equation}
	with
	\begin{equation*}
	F_i(w) := \begin{dcases}
	 w^{k_i} (w+1)^{-a_i-k_i} e^{t_i w}, & i=1,\cdots,m,\\
	 1, & i=0,
	\end{dcases}
	\end{equation*}
	for all $w\in(\Omega_\LL\setminus\{-1\})\cup (\Omega_\RR\setminus\{0\})$.
\end{defn}

\paragraph{Kernel $\Kess_Y$}

For any fixed ${\boldsymbol\lambda}=(\lambda_1,\cdots,\lambda_N)\in\intZ^N$ with $\lambda_1\ge \cdots\ge \lambda_N\ge 0$, we define
\begin{equation}
\label{eq:def_G_ftn}
\mathcal{G}_{\boldsymbol\lambda}(W)
:=\frac 			
{ \det \left[ w_i^{-j} (w_i+1)^{\lambda_j} \right]_{i,j=1}^M}
{ \det \left[ w_i^{-j} \right]_{i,j=1}^M},
\end{equation}
where $W=\{w_1,\cdots,w_M\}$ is a set of size $M\ge N$. We also set $\lambda_i=0$ if $i> N$. It is easy to see that $\mathcal{G}_{\boldsymbol\lambda}(W)$ is a symmetric polynomial of $w_1,\cdots,w_M$. In fact, this symmetric function is closely related to the Grothendieck polynomial \cite{Motegi-Sakai13} and inhomogeneous Schur polynomials \cite{Borodin17}. It also appears naturally in the periodic TASEP \cite{Baik-Liu21}. See \cite{Motegi-Sakai13,Borodin17,Baik-Liu21} for more discussions on this symmetric function.

Suppose the number of variables $M$ is greater than  the degree of the polynomial $|{\boldsymbol\lambda}|:=\sum_{j} \lambda_j$, then $\mathcal{G}_{\boldsymbol{\lambda}} (W)$ can be uniquely expressed in terms of power sum symmetric polynomials
\begin{equation}
\label{eq:expan_G}
\mathcal{G}_{\boldsymbol\lambda} (W) =1+  \sum_{\boldsymbol\mu=(\mu_1,\cdots)} c_{\boldsymbol\lambda,\boldsymbol\mu} p_{\boldsymbol \mu}(W),
\end{equation}
where the $\boldsymbol\mu$ sum is over all possible vector $\boldsymbol\mu=(\mu_1,\cdots)$ with positive and weakly decreasing coordinates $\mu_k$ such that $|\boldsymbol\mu|\le |{\boldsymbol\lambda}|$, and the polynomial $p_{\boldsymbol\mu}(W):= \prod_{k}  \left( \sum_{i=1}^M w_i^{\mu_k} \right)$. The constant $1$ comes from evaluating $\mathcal{G}_{\boldsymbol{\lambda}}(W)$ at $w_1=\cdots=w_M=0$. It is also easy to see that the coefficients $c_{{\boldsymbol\lambda},\boldsymbol\mu}$ only depend on ${\boldsymbol\lambda}$ and $\boldsymbol\mu$ but not $M$.

\begin{defn} 
	\label{def:ich}
	We define $\ich_{\boldsymbol\lambda}(v,u)$ by the following explicit formula
	\begin{equation}
	\label{eq:def_ich}
	\ich_{\boldsymbol\lambda}(v,u) = 1 +  \sum_{\boldsymbol\mu=(\mu_1,\cdots)} c_{{\boldsymbol\lambda},\boldsymbol\mu} \hat p_{\boldsymbol \mu}(v,u),
	\end{equation}
	where
	\begin{equation*}
	\hat p_{\boldsymbol\mu}(v,u):= \prod_k \left(u^{\mu_k} -v^{\mu_k}\right).
	\end{equation*}
\end{defn}

An alternate definition of $\ich_{\boldsymbol\lambda}(v,u)$ is as follows, with $\xi=e^{\frac{2\pi\ii} M}$ defined as the $M$-th root of unity, 
\begin{equation}
\label{eq:alter_ich}
\ich_{\boldsymbol\lambda}(v,u) = \mathcal{G}_{\boldsymbol{\lambda}} (u, v\xi, v\xi^2, \cdots, v\xi^{M-1})
\end{equation}
provided $M> |\boldsymbol{\lambda}|$. The equivalence of~\eqref{eq:alter_ich} and~\eqref{eq:def_ich} follows from a direct evaluation of~\eqref{eq:expan_G} when $W=\{u, v\xi, v\xi^2, \cdots, v\xi^{M-1}\}$, by using the simple fact that $u^{\mu_k}+\sum_{j=1}^{M-1} (v\xi^{j})^{\mu_k} = u^{\mu_k} - v^{\mu_k}$ since $\mu_k \le |\boldsymbol{\mu}|\le |\boldsymbol{\lambda}| <M$.

A similar calculation when $M \le |\boldsymbol{\lambda}|$ gives
\begin{equation}
\label{eq:alter_ich_ext}
\ich_{\boldsymbol{\lambda}}(v,u) = \mathcal{G}_{\boldsymbol{\lambda}} (u, v\xi, v\xi^2, \cdots, v\xi^{M-1}) + v^M\cdot r(v,u),
\end{equation}
where $r(v,u)$ is some polynomial of $v$ and $u$ with degree no more than $|\boldsymbol{\lambda}|-M$. This formula will be used later in  Lemma~\ref{lm:iich_reformulation} in Section~\ref{sec:preliminaries} to analytically extend an analogous function for periodic TASEP, and in Section~\ref{sec:special_IC} to evaluate the kernels for flat initial condition.

\begin{defn}
	\label{def:KY_ess}
	We define
	\begin{equation*}
	\Kess_Y(v,u) = \frac{1}{v-u} \cdot \left(\frac{u+1}{v+1}\right)^{y_N+N}\cdot \ich_{\boldsymbol{\lambda}(Y)}(v,u),
	\end{equation*}
	where $\boldsymbol{\lambda}(Y) =(\lambda_1,\cdots,\lambda_N)$ with $\lambda_i= (y_i+i) - (y_N+N)$.
\end{defn}

It is obvious that $\Kess_Y(v,u)$ is a kernel analytic for $v\in\Omega_\RR$ and for $u\in\Omega_\LL\setminus\{-1\}$. It is possible that $\Kess_Y(v,u)$ has a pole at $u=-1$ if $y_N+N<0$. We use the superscript  to emphasize that $\Kess_Y$ is the essential part containing the information of the initial condition $Y$ in the bigger kernel $\mathcal{K}_Y$. See the equation~\eqref{eq:KY}.

\subsubsection{Series expansion formula for $\mathcal{D}_Y(z_1,\cdots,z_{m-1})$}
\label{sec:series_expansion}

We introduce an alternate definition of $\mathcal{D}_Y(z_1,\cdots,z_{m-1})$ in terms of series expansion. 

We assume the contours $\Sigma_{\ell,\LL}^\out,\Sigma_{\ell,\LL}^\inn,\Sigma_{\ell,\RR}^\out,\Sigma_{\ell,\RR}^\inn,$ for $2\le \ell\le m$ and $\Sigma_{1,\LL}$, $\Sigma_{1,\RR}$ are the same as in Section~\ref{sec:spaces_operators}, $\Kess_Y(v,u)$ is the same as in Definition~\ref{def:KY_ess}. We also introduce some notations:
\begin{equation*}
\Delta(W):=\prod_{i<j}(w_j-w_i)
\end{equation*}
for any vector $W=(w_1,w_2,\cdots,w_n)$. For two vectors $W=(w_1,\cdots,w_n)$ and $W'=(w'_1,\cdots,w'_{n'})$, or sets $W=\{w_1,\cdots,w_n\}$ and $W'=\{w'_1,\cdots,w'_{n'}\}$, we define
\begin{equation*}
\Delta(W;W')=\prod_{i=1}^n\prod_{i'=1}^{n'}(w_i-w'_{i'}).
\end{equation*}
Moreover, if a function $f$ is well defined on each component of a vector $W=(w_1,\cdots,w_n)$, or each element of a set $W=\{w_1,\cdots,w_n\}$, we define
\begin{equation*}
f(W)=\prod_{i=1}^n f(w_i).
\end{equation*}
We comment that in the above notations, we allow the empty product and set an empty product to be $1$.

Finally, we recall that $\Sigma_{\ell,\LL}=\Sigma_{\ell,\LL}^\out\cup\Sigma_{\ell,\LL}^\inn$, $\Sigma_{\ell,\RR}=\Sigma_{\ell,\RR}^\out\cup\Sigma_{\ell,\RR}^\inn$ for $2\le \ell\le m$ as in~\eqref{eq:Sigma_contours}, and the measure $\mathrm{d}\mu(w)=\mathrm{d}\mu_{\boldsymbol{z}}(w)$ as in~\eqref{eq:def_dmu}.

\begin{defn}[Alternate definition of $\mathcal{D}_Y$]
	\label{def:D_Y_series}
	We have an alternate definition of $\mathcal{D}_Y$ below
	\begin{equation}
	\label{eq:series_expansion}
	\mathcal{D}_Y(z_1,\cdots,z_{m-1}):=\sum_{\boldsymbol{n}\in(\intZ_{\ge 0})^m}\frac{1}{(\boldsymbol{n}!)^2}\mathcal{D}_{\boldsymbol{n},Y}(z_1,\cdots,z_{m-1})
	\end{equation}
	with $\boldsymbol{n}!=n_1!\cdots n_m!$ for $\boldsymbol{n}=(n_1,\cdots,n_m)$. Here	
	\begin{equation}
	\label{eq:D_nY}
	\begin{split}
	&\mathcal{D}_{\boldsymbol{n},Y}(z_1,\cdots,z_{m-1})\\
	&=\prod_{\ell=1}^m\prod_{i_\ell=1}^{n_\ell} \int_{\Sigma_{\ell,\LL}}\mathrm{d}\mu_{\boldsymbol{z}}(u_{i_\ell}^{(\ell)})\int_{\Sigma_{\ell,\RR}}\mathrm{d}\mu_{\boldsymbol{z}}(v_{i_\ell}^{(\ell)}) \left[(-1)^{n_1(n_1+1)/2}\frac{\Delta(U^{(1)};V^{(1)})}{\Delta(U^{(1)})\Delta(V^{(1)})} \det\left[\Kess_Y(v_i^{(1)},u_j^{(1)})\right]_{i,j=1}^{n_1}\right]\\
	&\quad\cdot \left[\prod_{\ell=1}^m\frac{(\Delta(U^{(\ell)}))^2(\Delta(V^{(\ell)}))^2}{(\Delta(U^{(\ell)};V^{(\ell)}))^2}f_\ell(U^{(\ell)}) f_\ell(V^{(\ell)})\right]\\
	&\quad\cdot \left[\prod_{\ell=1}^{m-1} \frac{\Delta(U^{(\ell)};V^{(\ell+1)})\Delta(V^{(\ell)};U^{(\ell+1)})}{\Delta(U^{(\ell)};U^{(\ell+1)})\Delta(V^{(\ell)};V^{(\ell+1)})}\left(1-z_\ell\right)^{n_\ell}\left(1-\frac{1}{z_\ell}\right)^{n_{\ell+1}}\right]
	\end{split}
	\end{equation}
	and the functions $f_\ell$ are defined in~\eqref{eq:fi}. The vectors $U^{(\ell)}$ and $V^{(\ell)}$ are given by $U^{(\ell)}=(u_1^{(\ell)},\cdots,u_{n_\ell}^{(\ell)})$, $V^{(\ell)}=(v_1^{(\ell)},\cdots,v_{n_\ell}^{(\ell)})$ for $\ell=1,\cdots,m$.
\end{defn}

\begin{rmk}
	The above formula of $\mathcal{D}_Y(z_1,\cdots,z_{m-1})$ is in terms of an infinite sum. However, it is not hard to prove that when any $n_\ell>N$, the integral on the right hand side of~\eqref{eq:D_nY} is zero. Thus the summation actually only runs for finitely many terms. Here is the reason in brief: Any term in the expansion of $\Delta(V^{(\ell)})=\det\left[(v_i^{(\ell)})^{j-1}	\right]_{i,j=1}^{n_\ell}$ will give some $(v_i^{(\ell)})^{n_{\ell}-1}$ factor. The order $n_\ell-1\ge N$ is greater than or equal to the order of poles from any consecutive $f_i$ factors at $0$ (there might be poles from $v_i^{(\ell)}=v_{i'}^{(\ell+1)}=v_{i''}^{(\ell+2)}=\cdots$ or $v_i^{(\ell)}=v_{i'}^{(\ell-1)}=v_{i''}^{(\ell-2)}=\cdots$). Thus the multiple integral around $0$ will be zero. This proof is similar to that of Proposition~\ref{prop:invariance_distribution} so we omit the details.
\end{rmk}

The equivalence of the two definitions of $\mathcal{D}_Y(z_1,\cdots,z_{m-1})$ in Definition~\ref{def:operators_K1Y} and Definition~\ref{def:D_Y_series} follows from a general statement below.

\begin{prop}
	\label{prop:equivalence_Fredholm_series}
	Let $\Sigma_1,\cdots,\Sigma_m$ be disjoint sets in $\complexC$ and let $\mathcal{H}=L^2(\Sigma_1\cup\cdots\cup\Sigma_m,\mu)$ for some measure $\mu$. Let $\widehat\Sigma_1,\cdots,\widehat\Sigma_m$ be disjoint sets in $\complexC$ and let $\widehat{\mathcal{H}}=L^2(\widehat\Sigma_1\cup\cdots\cup\widehat\Sigma_m,\widehat\mu)$ for some measure $\widehat\mu$. Let $A$ be an operator from $\widehat{\mathcal{H}}$ to $\mathcal{H}$ and $B$ an operator from $\mathcal{H}$ to $\widehat{\mathcal{H}}$, both of which are defined by kernels. Suppose $A$ and $B$ have the following block structures:
	\begin{itemize}
		\item For any $(w,\widehat w)\in \Sigma_i\times\widehat\Sigma_j$
		\begin{equation*}
		A(w,\widehat w)=\begin{dcases}
		\frac{f_i(w)\widehat f_j(\widehat w)}{w-\widehat w},& \text{if } 2s-1\le i,j\le 2s \text{ for some integer } s\ge 1,\\
		0,&\text{otherwise}.
		\end{dcases}
		\end{equation*} 
		\item For any $(\widehat w,w)\in \widehat\Sigma_j\times\Sigma_i$
		\begin{equation*}
		B(\widehat w,w)=\begin{dcases}
		\frac{\widehat g_j(\widehat w)g_i(w)}{\widehat w-w},& \text{if } 2s\le i,j\le 2s+1 \text{ for some integer } s\ge 1,\\
		\widehat g_1(\widehat w)g_1(w) H(\widehat w,w),&\text{if } i=j=1,\\
		0,&\text{otherwise}.
		\end{dcases} 
		\end{equation*}
	\end{itemize}
	Assume that the Fredholm determinant $\det(I-AB)$ is well-defined and is equal to the usual Fredholm determinant series expansion. Then
	\begin{equation*}
	\begin{split}
	\det(I-AB) & =\sum_{\boldsymbol{n}\in (\intZ_{\ge 0})^m} \frac{1}{(\boldsymbol{n}!)^2}	
	\prod_{\ell=1}^{m}  \prod_{{i_\ell}=1}^{n_{\ell}} \int_{\Sigma_\ell}\dd\mu(w_{i_\ell}^{(\ell)})
	\prod_{\ell=1}^{m}\prod_{{i_\ell}=1}^{n_{\ell}} \int_{\widehat\Sigma_\ell}\dd\widehat \mu(\widehat w_{i_\ell}^{(\ell)}) \\
	&\quad\,\,\, \left[(-1)^{n_1(n_1+1)/2}\frac{\Delta(W^{(1)};\widehat  W^{(1)})}{\Delta(W^{(1)})\Delta(\widehat W^{(1)})} \det\left[H(\widehat w_i^{(1)},w_j^{(1)})\right]_{i,j=1}^{n_1}\right]\\
	&\quad\cdot \left[\prod_{\ell=1}^m\frac{(\Delta(W^{(\ell)}))^2(\Delta(\widehat W^{(\ell)}))^2}{(\Delta(W^{(\ell)};\widehat W^{(\ell)}))^2}f_\ell(W^{(\ell)}) g_\ell(W^{(\ell)})\widehat f_\ell(\widehat W^{(\ell)}) \widehat g_\ell(\widehat W^{(\ell)})\right]\\
	&\quad\cdot \left[\prod_{\ell=1}^{m-1} \frac{\Delta(W^{(\ell)}; W^{(\ell+1)})\Delta(\widehat W^{(\ell)};\widehat W^{(\ell+1)})}{\Delta(W^{(\ell)};\widehat W^{(\ell+1)})\Delta(\widehat W^{(\ell)};W^{(\ell+1)})}\right],
	\end{split}
	\end{equation*}
	where $\boldsymbol{n}=(n_1,\cdots,n_m)$. The notations $|\boldsymbol{n}|:=n_1+\cdots+n_m$ and $\boldsymbol{n}!:=n_1!\cdots n_m!$.  The vectors $W^{(\ell)}=(w_1^{(\ell)},\cdots,w_{n_\ell}^{(\ell)})$, $\widehat W^{(\ell)}=(\widehat w_1^{(\ell)},\cdots,\widehat w_{n_\ell}^{(\ell)})$ for $\ell=1,\cdots,m$.
\end{prop}
\begin{proof}
	The proof when $H(\widehat w,w)=\frac{1}{\widehat w-w}$ was proved in \cite{Baik-Liu19}, and the general $H$ case was proved in \cite{Baik-Liu21}. See Section 4.3 of \cite{Baik-Liu19} for the proof with this special $H$. Although their proof was presented for specific choices of contours $\widehat\Sigma_i$, $\Sigma_i$, measures $\dd\mu, \dd\widehat\mu$ and functions $f_i,g_i,\widehat f_i,\widehat g_i$, it holds for this proposition by replacing their specific choices to the general settings. Hence we do not provide details here.
\end{proof}

\subsubsection{Further discussion on $\mathcal{D}_Y$}
\label{sec:properties_D}

In this section, we mainly discuss the function $\mathcal{D}_Y$. 
We will show that there are various formulas for $\mathcal{D}_Y$. In the definition~\ref{def:D_Y_series} of $\mathcal{D}_Y$, we may use a different nesting order of the contours, or modify the kernels in the Fredholm determinant representation. Especially we could replace $\Kess_Y$, which contains the information of the initial condition, by a more general form $\Kess_Y(v,u) + \Knull(v,u)$ as long as $\Knull(v,u)$ satisfies certain conditions.  
These are discussed in Propositions~\ref{prop:invariance_spaces},~\ref{prop:invariance1} and~\ref{prop:invariance_distribution}. We will also discuss one identity which $\Kess_Y$ satisfies, see Proposition~\ref{prop:orthogonality}. 

In Section~\ref{sec:special_IC}, we  write down the explicit formulas of $\mathcal{D}_Y$ when $Y$ is either the step or the flat initial condition. These formulas will be used later to evaluate the limiting multi-time distributions for these two initial conditions.

Then we  verify, in a formal way, that the function $\mathcal{D}_Y$ for $m=1$ matches the known result of the one point distribution formula. This will be given in Section~\ref{sec:one_point}.

Finally we prove two identities about $\mathcal{D}_Y$ which will be used in our proofs later.

We remark that throughout this section, the propositions are proved by only using the definition of $\mathcal{D}_Y$. We will use these propositions in the proof of other statements in the paper.

\paragraph{About the formula of $\mathcal{D}_Y$}
\label{sec:def_mathcalD_Y}

As we mentioned before (see the first paragraph of Section~\ref{sec:Fredholm_representation}), there are different Fredholm determinant representations (and the corresponding series expansions) for  $\mathcal{D}_Y$.

We first show that the spaces of the Fredholm operators could be different. More explicitly, the nesting order of the contours, if we adjust the measure appropriately, does not affect  $\mathcal{D}_Y$ in the definition.

\begin{prop}
	\label{prop:invariance_spaces}
	Let $\tilde \Sigma_{1,\LL}^\out,\cdots,\tilde \Sigma_{m-1,\LL}^\out,\tilde\Sigma_{m,\LL},\tilde\Sigma_{m-1,\LL}^\inn,\cdots,\tilde\Sigma_{1,\LL}^\inn$ be $2m-1$ nested simple closed contours, from outside to inside, in $\Omega_\LL$ enclosing the point $-1$. Let $\tilde\Sigma_{\ell,\LL}:=\tilde\Sigma_{\ell,\LL}^\out\cup\tilde\Sigma_{\ell,\LL}^\inn$ for $1\le\ell\le m-1$. We define the measure $\dd\tilde\mu(w)$ on $\tilde\Sigma_{\ell,\LL}$ in the following way
	\begin{equation*}
	\dd\tilde\mu(w)=\dd\tilde\mu_{\boldsymbol{z}}(w) :=\begin{dcases}
	\frac{1}{1-z_\ell}\ddbarr{w},& w\in \tilde\Sigma_{\ell,\LL}^\out,\quad \ell=1,\cdots,m-1,\\
	\frac{-z_\ell}{1-z_\ell}\ddbarr{w},& w\in \tilde\Sigma_{\ell,\LL}^\inn,\quad \ell=1,\cdots,m-1,\\
	\ddbarr{w},& w\in\tilde\Sigma_{m,\LL}.
	\end{dcases}
	\end{equation*}
	Then $\mathcal{D}_Y(z_1,\cdots,z_{m-1})$ is invariant if we replace all the $\Sigma_{\ell,\LL}$ contours and the associated measure $\dd\mu(w)$ to $\tilde\Sigma_{\ell,\LL}$ and $\dd\tilde\mu(w)$.  We define the $\tilde\Sigma_{\ell,\RR}$ contours in $\Omega_\RR$ enclosing $0$ and $\dd\tilde\mu(w)$ on $\tilde\Sigma_{\ell,\RR}$ in a similar way. Then $\mathcal{D}_Y(z_1,\cdots,z_{m-1})$ is also invariant if we replace all the $\Sigma_{\ell,\RR}$ contours and the associated measure $\dd\mu(w)$ to $\tilde\Sigma_{\ell,\RR}$ and $\dd\tilde\mu(w)$. 
\end{prop}

The above proposition indicates that we could flip the order of the nested contours and the associated measure accordingly without changing the value of $\mathcal{D}_Y(z_1,\cdots,z_{m-1})$. We remark that we only considered the case when the contour with the smallest or largest label lies in the middle of the contours and the remaining contours are nested in the order of their labels, but it is possible to put any contour $\Sigma_{\ell,\LL}$ or $\Sigma_{\ell,\RR}$ at the center or consider nested contours in arbitrary order. But the associated measures are not as neat as $\dd\mu$ or $\dd\tilde\mu$. It is not clear how these other different orders benefit the evaluation of $\mathcal{D}_Y(z_1,\cdots,z_{m-1})$ either. Hence we do not discuss it in details.

The proof of Proposition~\ref{prop:invariance_spaces} is provided in Section~\ref{sec:proof_invariance_spaces}.

Now we consider the Fredholm determinant kernels in $\mathcal{D}_Y$. Obviously the Fredholm determinant is invariant if we apply a conjugation to the kernels. Furthermore, we can modify the functions $F_i$'s (hence the functions $f_i$'s accordingly) as well.

\begin{prop}
	\label{prop:invariance1}
	$\mathcal{D}_Y(z_1,\cdots,z_{m-1})$ is invariant if we replace the function $F_i(w)$ by $\tilde F_i(w)=c_i F_i(w) $ for any nonzero numbers $c_1,\cdots,c_{m}$. It is also invariant if we shift all the $y_i$'s and $a_i$'s by the same integer constant $c$.
\end{prop}
\begin{proof}
	We first consider the change $F_i(w)\to c_i F_i(w)$. This will change $f_i(u)\to \frac{c_i}{c_{i-1}}f_i(u)$ for $u\in\Omega_\LL\setminus\{-1\}$ and $f_i(v)\to \frac{c_{i-1}}{c_i}f_i(v)$ for $v\in\Omega_\RR\setminus\{0\}$ by the definition of $f_i$ in~\eqref{eq:fi}. Here we set $c_0=1$. Now we consider the series expansion formula~\eqref{eq:series_expansion} of $\mathcal{D}_Y$. The $\boldsymbol{n}$-th term $\mathcal{D}_{\boldsymbol{n},Y}$ is invariant under the above changes since $f_\ell(U^{(\ell)})f_\ell(V^{(\ell)})$ has the same number of factors $\frac{c_\ell}{c_{\ell-1}}$ and $\frac{c_{\ell-1}}{c_{\ell}}$ whose product is $1$.
	
	Now we consider the case when we shift all $y_i$ and $a_i$ by the same constant $c$. This change does not affect the functions $f_\ell$ for $\ell>1$, and $f_1(u)\to f_1(u)\cdot (u+1)^{-c}$, $f_1(v)\to f_1(v)\cdot (v+1)^c$ for $u\in\Omega_\LL\setminus\{-1\}$ and $v\in\Omega_\RR\setminus\{0\}$. On the other hand, by the definition of $\Kess_Y(v,u)$ in~\eqref{def:KY_ess} we know that $\Kess_Y(v,u)\to \Kess_Y(v,u)\left(\frac{1+u}{1+v}\right)^{c}$. Thus $f_1(U^{(1)})f_1(V^{(1)})\det\left[\Kess_Y(v_i^{(1)},u_j^{(1)})\right]_{i,j=1}^{n_1}$ is unchanged. 
\end{proof}

It is more challenging to understand $\Kess_Y(v,u)$, which encodes the initial condition $Y$ in $\mathcal{D}_Y$. 
It is possible to show that $\mathcal{D}_Y$ does not depend on the  explicit formula of $\Kess_Y(v,u)$, but only depends on the value of
\begin{equation*}
\langle f,g\rangle_{Y}:=\oint_0\ddbarr{v}\oint_{-1}\ddbarr{u} f(v)\Kess_Y(v,u) g(u)
\end{equation*}
for functions $f$ and $g$ satisfying $v^{\max\{k_\ell:\ell=1,\cdots,m\}}f(v)$ and $(u+1)^{\max\{k_\ell+a_\ell:\ell=1,\cdots,m\}}g(u)$ are analytic at $0$ and $-1$ respectively, where the above contours of integration are sufficiently small. 
In other words, $f$ ($g$, respectively) is meromorphic in a neighborhood of $0$ ($-1$, respectively) with a possible pole at $0$ ($-1$, respectively) and its order is at most $\max\{k_\ell:\ell=1,\cdots,m\}$ ($\max\{k_\ell+a_\ell:\ell=1,\cdots,m\}$, respectively).
Hence the true role of $\Kess_Y(v,u)$ is to determine the above bi-linear form. We do not want to fully explain it here in details since it involves the orthogonalization of eigenfunctions and convergence of formal expansions in terms of orthogonal basis. Instead, we provide a lighter version below.

\begin{prop}
	\label{prop:invariance_distribution}
	$\mathcal{D}_Y(z_1,\cdots,z_{m-1})$ is invariant if we replace the kernel $\Kess_Y(v,u)$ by $\Kess_Y(v,u) + \Knull(v,u)$ provided $\Knull$ satisfies either conditions (1) or (2).
	\begin{enumerate}
		\item[(1)] For each fixed $u\in\cup_{\ell=1}^m\Sigma_{\ell,\LL}$, $\Knull(v,u)$ is analytic for $v\in \Omega_\RR\setminus\{0\}$. Moreover, for all $i\le \max\{k_\ell:\ell=1,\cdots,m\}$ and all $j$,
		\begin{equation*}
		\oint_0\ddbarr{v}\int_{\Sigma_{\ell,\LL}^\star}\ddbarr{u} v^{-i}\Knull(v,u)(u+1)^{-j} =0
		\end{equation*}
		for each $1\le \ell\le m$, and $\star$ is any of $\{\out,\inn\}$ if $\ell\ge 2$, or  empty if $\ell=1$.
		\item[(2)]  
		For each fixed $v\in\cup_{\ell=1}^m\Sigma_{\ell,\RR}$, $\Knull(v,u)$ is analytic for $u\in \Omega_\LL\setminus\{-1\}$. Moreover, for all all $j\le \max\{a_\ell+k_\ell:\ell=1,\cdots,m\}$ and all $i$,
		\begin{equation*}
		\int_{\Sigma_{\ell,\RR}^\star}\ddbarr{v}\oint_{-1}\ddbarr{u} v^{-i}\Knull(v,u)(u+1)^{-j} =0
		\end{equation*}
		for each $1\le \ell\le m$, and $\star$ is any of $\{\out,\inn\}$ if $\ell\ge 2$, or  empty if $\ell=1$.
	\end{enumerate}
\end{prop}

The proof of Proposition~\ref{prop:invariance_distribution} is given in Section~\ref{sec:proof_invariance_distribution}.

We could understand Proposition~\ref{prop:consistency} in the following probabilistic way. Note the fact that the distribution function itself only depends on part of the initial data. More explicitly, this distribution function is independent of $y_i$'s with $i> \max\{k_\ell:\ell=1,\cdots,m\}$ since these particles do not affect the particles ahead of them. Similarly the distribution function is independent of  $y_i$'s with $y_i+i> \max\{a_\ell+k_\ell:\ell=1,\cdots,m\}$ by using the duality of particles and empty sites. The conditions (1) and (2) above precisely indicate these independence.

By the proposition above, we know that there are many choices of choosing a kernel to replace $\Kess_Y(v,u)$ in the definition of $\mathcal{D}_Y$. It may happen that one needs to pick the appropriate kernel to obtain the asymptotics of $\mathcal{D}_Y$. We will see this fact for the flat initial condition. Nevertheless, the kernel $\Kess_Y(v,u)$ defined in Definition~\ref{def:KY_ess} has the following property. 

\begin{prop}
	\label{prop:orthogonality}
$\Kess_Y(v,u)$ is a kernel satisfying
\begin{equation}
\label{eq:orthogonality}
\oint_0 v^{-i}(v+1)^{y_i+i}\cdot \Kess_Y (v,u) \ddbarr{v} = -u^{-i}(u+1)^{y_i+i}
\end{equation}
for all $i=1,\cdots,N$.
\end{prop}

The proof of Proposition~\ref{prop:orthogonality} is given in Section~\ref{sec:proof_orthogonality}.

Note that~\eqref{eq:orthogonality} has infinitely many solutions. Formally, for each fixed $u$,~\eqref{eq:orthogonality} is a system of $N$ linear equations of infinitely many variables $v$. However, each solution $\mathcal{K}_Y(v,u)$, if it is analytic in $\Omega_\RR \times (\Omega_\LL\setminus\{-1\})$, can be expressed as
\begin{equation*}
\mathcal{K}_Y(v,u)=\Kess_Y(v,u) + \Knull(v,u)
\end{equation*}
where $\Knull(v,u):=\mathcal{K}_Y(v,u)-\Kess_Y(v,u)$ satisfies
\begin{equation}
	\label{eq:aux_2021_9_25_01}
\oint_0 v^{-i} \cdot \Knull(v,u) \ddbarr{v}=0
\end{equation}
for all integers $i$ satisfying $i\le N$. The reason of~\eqref{eq:aux_2021_9_25_01} is as follows. We first write $v^{-i}=\sum_{j=1}^i c_{i,j}v^{-j}(v+1)^{y_j+j}+P_j(v)$, where $c_{i,j}$ are constants determined by comparing the coefficients of $v^{-j}$ in both sides, and $P_j$ is a polynomial. Then~\eqref{eq:aux_2021_9_25_01} follows from the following facts
\begin{equation*}
	\oint_0v^{-j}(v+1)^{y_j+j} \cdot \mathcal{K}_Y(v,u)\ddbarr{v}=\oint_0v^{-j}(v+1)^{y_j+j} \cdot\Kess_Y(v,u)\ddbarr{v} =-u^{-j}(u+1)^{y_j+j},\quad 1\le j\le i\le N,
\end{equation*} 
and
\begin{equation*}
	\oint_0P_j(v) \cdot \mathcal{K}_Y(v,u)\ddbarr{v}=\oint_0P_j(v) \cdot\Kess_Y(v,u)\ddbarr{v}=0
\end{equation*}
due to the analyticity of $\mathcal{K}_Y(v,u)$ and $\Kess_Y(v,u)$ at $v=0$. 

Now by applying Proposition~\ref{prop:invariance_distribution} and the equation~\eqref{eq:aux_2021_9_25_01}, we know that $\mathcal{D}_Y$ is invariant if we replace $\Kess_Y(v,u)$ by any kernel which is analytic in $\Omega_\RR\times(\Omega_\LL\setminus\{-1\})$ and satisfies~\eqref{eq:orthogonality}.

\paragraph{$\mathcal{D}_Y$ for step and flat initial conditions}
\label{sec:special_IC}
We consider two special initial conditions and write down their formulas of $\mathcal{D}_Y$ explicitly. These formulas  are suitable for asymptotic analysis and will be used in Section~\ref{sec:asymptotics}.
\vspace{0.3cm}

The first initial condition we consider is the so-called \emph{step} initial condition. It is defined to be
\begin{equation*}
Y_{\mathrm{step}}=(y_1,\cdots,y_N)=(-1,\cdots,-N).
\end{equation*}
In this case $\boldsymbol{\lambda}(Y_{\mathrm{step}})=(0,\cdots,0)$ since $\lambda_i=(y_i+i)-(y_N+N)=0$. By~\eqref{eq:def_G_ftn} we have $\mathcal{G}_{\boldsymbol{\lambda}(Y_{\mathrm{step}})}(W)=1$. Now using Definitions~\ref{def:ich} and~\ref{def:KY_ess}, we know $\ich_{\boldsymbol{\lambda}(Y_{\mathrm{step}})}(v,u)=1$ and $\Kess_{Y_{\mathrm{step}}}(v,u)=\frac{1}{v-u}$. Therefore
\begin{equation*}
\mathcal{D}_{Y_{\mathrm{step}}}(z_1,\cdots,z_{m-1})=\det(I-\mathcal{K}_1\mathcal{K}_{{Y_{\mathrm{step}}}})
\end{equation*}
with	
\begin{equation*}
\mathcal{K}_1(w,w'):= \left(\delta_i(j) + \delta_i( j+ (-1)^i)\right) \frac{ f_i(w) }{w-w'} Q_1(j)
\end{equation*}
and
\begin{equation*}
\mathcal{K}_{Y_{\mathrm{step}}}(w',w):= 
\left(\delta_j (i) + \delta_j(i - (-1)^j)\right) \frac{ f_j(w') }{w'-w} Q_2(i)
\end{equation*}
for any $w\in (\Sigma_{i,\LL}\cup \Sigma_{i,\RR}) \cap \mathcal{S}_1$ and $w'\in (\Sigma_{j,\LL}\cup \Sigma_{j,\RR}) \cap \mathcal{S}_2$ with $1\le i,j\le m$. Here the spaces $\Sigma_{i,\LL},\Sigma_{i,\RR},\mathcal{S}_1,\mathcal{S}_2$ and functions $f_i,Q_1,Q_2$ are the same as in Definition~\ref{def:operators_K1Y}. One could similarly write down the series expansion of $\mathcal{D}_{Y_{\mathrm{step}}}(z_1,\cdots,z_{m-1})$. It is given by
\begin{equation*}
\mathcal{D}_{Y_{\mathrm{step}}}(z_1,\cdots,z_{m-1}):=\sum_{\boldsymbol{n}\in(\intZ_{\ge 0})^m}\frac{1}{(\boldsymbol{n}!)^2}\mathcal{D}_{\boldsymbol{n},{Y_{\mathrm{step}}}}(z_1,\cdots,z_{m-1})
\end{equation*}
with 
\begin{equation*}
\begin{split}
\mathcal{D}_{\boldsymbol{n},{Y_{\mathrm{step}}}}(z_1,\cdots,z_{m-1})
&=\prod_{\ell=1}^m\prod_{i_\ell=1}^{n_\ell} \int_{\Sigma_{\ell,\LL}}\mathrm{d}\mu_{\boldsymbol{z}}(u_{i_\ell}^{(\ell)})\int_{\Sigma_{\ell,\RR}}\mathrm{d}\mu_{\boldsymbol{z}}(v_{i_\ell}^{(\ell)})\\
&\quad\cdot \left[\prod_{\ell=1}^m\frac{(\Delta(U^{(\ell)}))^2(\Delta(V^{(\ell)}))^2}{(\Delta(U^{(\ell)};V^{(\ell)}))^2}f_\ell(U^{(\ell)}) f_\ell(V^{(\ell)})\right]\\
&\quad\cdot \left[\prod_{\ell=1}^{m-1} \frac{\Delta(U^{(\ell)};V^{(\ell+1)})\Delta(V^{(\ell)};U^{(\ell+1)})}{\Delta(U^{(\ell)};U^{(\ell+1)})\Delta(V^{(\ell)};V^{(\ell+1)})}\left(1-z_\ell\right)^{n_\ell}\left(1-\frac{1}{z_\ell}\right)^{n_{\ell+1}}\right].
\end{split}
\end{equation*}

\vspace{0.3cm}
The second initial condition we consider here is the so-called \emph{pseudo-flat} initial condition. It is defined to be
\begin{equation*}
{Y_{\mathrm{pf}}}=(y_1,\cdots,y_N)=(-2,\cdots,-2N).
\end{equation*}
In other words, $y_i=-2i$ for all $1\le i\le N$. For this pseudo-flat initial condition, we have the following result for $\Kess_{Y_{\mathrm{pf}}}$.
\begin{prop}
	\label{prop:Kess_flat}
	If $|v|<1/2$ and $|v|<|u+1|$, we have
	\begin{equation*}
	\Kess_{Y_{\mathrm{pf}}}(v,u) = \frac{2v+1}{(v-u)(u+v+1)}+ v^N p(v,u)
	\end{equation*}
	for some function $p(v,u)$ which is analytic for $(v,u)$ when $|v|<\min\{1/2, |u+1|\}$.
\end{prop}

The proof of Proposition~\ref{prop:Kess_flat} is given in Section~\ref{sec:proof_Kess_flat}.

By applying Propositions~\ref{prop:invariance_distribution} and~\ref{prop:Kess_flat}, we could replace $\Kess_{Y_{\mathrm{pf}}}$ by the kernel $\frac{2v+1}{(v-u)(u+v+1)}$ if we choose the contours appropriately such that $\Sigma_{1,\RR}$ is within the disk $\disk(1/2)=\{v:|v|<1/2\}$ and $\Sigma_{1,\LL}$ is outside of $-1-\Sigma_{1,\RR}:=\{-1-v: v\in\Sigma_{1,\RR}\}$. However, we could further reduce it to a delta kernel which makes the formula of $\mathcal{D}_{Y_{\mathrm{pf}}}(z_1,\cdots,z_{m-1})$ even simpler.

In order to introduce the new formula, we need to slightly modify the contours. Let  $\Sigma_{m,\LL}^{\out},\cdots,\Sigma_{2,\LL}^{\out}$, $\Sigma_{1,\LL}$, $\Sigma_{2,\LL}^{\inn},\cdots,\Sigma_{m,\LL}^\inn$ are $2m-1$ nested simple closed contours, from outside to inside, in $\Omega_\LL=\{w\in\complexC: \Re(w)<-1/2\}$ enclosing the point $-1$, and
$\Sigma_{m,\RR}^{\out},\cdots,\Sigma_{2,\RR}^{\out}$, $\Sigma_{1,\RR}$, $\Sigma_{2,\RR}^{\inn},\cdots,\Sigma_{m,\RR}^\inn$ are $2m-1$ nested simple closed contours, from outside to inside, in $\Omega_\RR=\{w\in\complexC: \Re(w)>-1/2\}$ enclosing the point $0$. We further assume that $\Sigma_{1,\LL}=-1-\Sigma_{1,\RR}$. 

\begin{prop}
	\label{prop:D_flat}
	Suppose the parameters satisfy $\max\{a_\ell +k_\ell: \ell=1,\cdots,m\} \le0$.	Then
	\begin{equation*}
	\mathcal{D}_{Y_{\mathrm{pf}}}(z_1,\cdots,z_{m-1})=\det\left(I -\mathcal{K}_1\mathcal{K}_{Y_\mathrm{pf}}\right),
	\end{equation*}
		where two operators
	\begin{equation*}
	\mathcal{K}_1: L^2(\mathcal{S}_2,\dd\mu) \to L^2(\mathcal{S}_1,\dd\mu),\qquad \mathcal{K}_{Y_\mathrm{flat}}^{(1)}: L^2(\mathcal{S}_1,\dd\mu)\to L^2(\mathcal{S}_2,\dd\mu)
	\end{equation*}
	are defined by their kernels
	\begin{equation*}
	\mathcal{K}_1(w,w'):= \left(\delta_i(j) + \delta_i( j+ (-1)^i)\right) \frac{ f_i(w) }{w-w'} Q_1(j)
	\end{equation*}
	and
	\begin{equation*}
	\mathcal{K}_{Y_{\mathrm{pf}}}(w',w):= \begin{dcases}
	\left(\delta_j (i) + \delta_j(i - (-1)^j)\right) \frac{ f_j(w') }{w'-w} Q_2(i), & i\ge 2,\\
	\delta_j(1)f_j(w')\delta(-w'-1,w), & i=1,
	\end{dcases}
	\end{equation*}
	for any $w\in (\Sigma_{i,\LL}\cup \Sigma_{i,\RR}) \cap \mathcal{S}_1$ and $w'\in (\Sigma_{j,\LL}\cup \Sigma_{j,\RR}) \cap \mathcal{S}_2$ with $1\le i,j\le m$. The definitions of $\mathcal{S}_1,\mathcal{S}_2$, $f_i$, $Q_1,Q_2$ are the same as in Definition~\ref{def:operators_K1Y}, with the further assumption $\Sigma_{1,\LL}=-1-\Sigma_{1,\RR}$ as described before, and the $\delta(-w'-1,w)$ is a delta kernel defined by
	\begin{equation*}
	\int_{\Sigma_{1,\LL}}\delta(-v-1,u)g(u)\ddbarr{u} = g(-v-1)
	\end{equation*}
	for any function $g\in L^2(\Sigma_{1,\LL},\ddbarr{u})$ and any $v\in\Sigma_{1,\RR}$.
\end{prop}

The proof of Proposition~\ref{prop:D_flat} is given in Section~\ref{sec:proof_D_flat}. We remark that the assumption $\max\{a_\ell +k_\ell: \ell=1,\cdots,m\} \le0$ is reasonable. In terms of TASEP, if we view empty sites as ``white particles'' and original particles as ``black particles'', then the dynamics of TASEP becomes exchanging two neighboring particles with ``black'' and ``white'' colors (``black'',``white'' change to ``white'', ``black''). $x_{k_\ell}(t_\ell)+k_\ell\ge a_\ell+k_\ell>0$ means that the $k_\ell$-th ``black particle'' has already met some ``white particles'' initially located at $\intZ_{\ge 0}$. In other words, the location of this $k_\ell$-th particle is affected by some initial condition which is outside of the ``flat'' region. In this case we do not expect a same formula as that for $\max\{a_\ell +k_\ell: \ell=1,\cdots,m\} \le0$.

\vspace{0.2cm}
It turns out that we could drop the assumption $\max\{a_\ell +k_\ell: \ell=1,\cdots,m\} \le0$ if we consider the \emph{flat} initial condition
\begin{equation*}
Y_\mathrm{flat}=(\cdots,y_{-2},y_{-1},y_0,y_1,y_2,\cdots)\quad \text{with}\quad y_i=-2i,\qquad i\in\intZ.
\end{equation*}
Here we allow the labels of particles to be negative. This follows from a translation on the labels and locations of particles in Proposition~\ref{prop:D_flat} and then let $N$ be sufficiently large. More explicitly, we have
\begin{prop}
	\label{prop:infinite_flat}
	Suppose we consider TASEP with the flat initial condition $Y_\mathrm{flat}$. Assume $m\ge 1$ is an integer. Suppose $a_\ell,k_\ell$ are integers for each $\ell=1,\cdots,m$, and $t_1,\cdots,t_m$ are real numbers satisfying $0\le t_1\le\cdots\le t_m$. Then
	\begin{equation}
	\label{eq:infinite_flat}
	\begin{split}
	&\prob_{Y_{\mathrm{flat}}}\left(\bigcap_{\ell=1}^m\left\{
	x_{k_\ell}(t_\ell)\ge a_\ell \right\}\right)\\
	&=\oint\cdots\oint \left[\prod_{\ell=1}^{m-1}\frac{1}{1-z_{\ell}}\right]\cdot \mathcal{D}_{Y_{\mathrm{flat}}}(z_1,\cdots,z_{m-1})\ddbar{z_1}\cdots\ddbar{z_{m-1}}
	\end{split}
	\end{equation}
	where the contours of integration are circles centered at the origin and of radii less than $1$. The function $\mathcal{D}_{Y_{\mathrm{flat}}}(z_1,\cdots,z_{m-1})$ has the same formula as $\mathcal{D}_{Y_{\mathrm{pf}}}(z_1,\cdots,z_{m-1})$ defined in Proposition~\ref{prop:D_flat}, without the restriction $\max\{a_\ell+k_\ell: \ell=1,\cdots,m\}\le 0$. 
\end{prop}
\begin{proof}
   When $k_\ell\ge 1$ and $a_\ell+k_\ell\le 0$ for all $\ell$, it is easy to see, similarly to the argument after Proposition~\ref{prop:D_flat}, that the event $\{x_{k_\ell}(t_\ell)\ge a_\ell\}$ only depends on part of the initial condition $y_{i}=-2i$ satisfying $i\le k_\ell$ and $i\ge -a_\ell-k_\ell+1$. Note that $-a_\ell-k_\ell+1\ge 1$. Both $Y_{\mathrm{pf}}$ and $Y_{\mathrm{flat}}$ contain this part of the initial condition if we choose $N\ge \max\{k_\ell:\ell=1,\cdots,m\}$. Thus we know
   \begin{equation*}
   \prob_{Y_{\mathrm{flat}}}\left(\bigcap_{\ell=1}^m\left\{
   x_{k_\ell}(t_\ell)\ge a_\ell \right\}\right)=\prob_{Y_{\mathrm{pf}}}\left(\bigcap_{\ell=1}^m\left\{
   x_{k_\ell}(t_\ell)\ge a_\ell \right\}\right).
   \end{equation*}
   Then~\eqref{eq:infinite_flat} follows from Proposition~\ref{prop:D_flat}. Note that although $Y_{\mathrm{pf}}$ depends on $N$, the formula of $\mathcal{D}_{Y_{\mathrm{pf}}}$ is independent of $N$.
   
   More generally, we know that the left hand side is invariant under the translation $(a_\ell,k_\ell)\to (a_\ell-2c,k_\ell+c)$ for all $\ell$. Here $c$ is any fixed integer. By choosing sufficiently large $c$, we have $a_\ell-2c +k_\ell +c\le 0$ and $k_\ell+c\ge 1$ for all $\ell$. Thus it is sufficient to show that the right hand side of~\eqref{eq:infinite_flat} is also invariant under such a translation. Below we show this by using series expansion of $\mathcal{D}_{Y_\mathrm{flat}}$.
   
   Similarly to the general initial condition case, we could write down the series expansion of $\mathcal{D}_{Y_{\mathrm{flat}}}$ .
   It is given by
   \begin{equation*}
   \mathcal{D}_{Y_{\mathrm{flat}}}(z_1,\cdots,z_{m-1}):=\sum_{\boldsymbol{n}\in(\intZ_{\ge 0})^m}\frac{1}{(\boldsymbol{n}!)^2}\mathcal{D}_{\boldsymbol{n},{Y_{\mathrm{flat}}}}(z_1,\cdots,z_{m-1})
   \end{equation*}
   with 
   \begin{equation}
   \label{eq:aux_100}
   \begin{split}
   \mathcal{D}_{\boldsymbol{n},{Y_{\mathrm{flat}}}}(z_1,\cdots,z_{m-1})
    &=\prod_{\ell=1}^m\prod_{i_\ell=1}^{n_\ell} \int_{\Sigma_{\ell,\LL}}\mathrm{d}\mu_{\boldsymbol{z}}(u_{i_\ell}^{(\ell)})\int_{\Sigma_{\ell,\RR}}\mathrm{d}\mu_{\boldsymbol{z}}(v_{i_\ell}^{(\ell)})\\
   &\quad \left[(-1)^{n_1(n_1+1)/2}\frac{\Delta(U^{(1)};V^{(1)})}{\Delta(U^{(1)})\Delta(V^{(1)})} \det\left[\delta(-v_i^{(1)}-1,u_j^{(1)})\right]_{i,j=1}^{n_1}\right]\\
   &\quad\cdot \left[\prod_{\ell=1}^m\frac{(\Delta(U^{(\ell)}))^2(\Delta(V^{(\ell)}))^2}{(\Delta(U^{(\ell)};V^{(\ell)}))^2}f_\ell(U^{(\ell)}) f_\ell(V^{(\ell)})\right]\\
   &\quad\cdot \left[\prod_{\ell=1}^{m-1} \frac{\Delta(U^{(\ell)};V^{(\ell+1)})\Delta(V^{(\ell)};U^{(\ell+1)})}{\Delta(U^{(\ell)};U^{(\ell+1)})\Delta(V^{(\ell)};V^{(\ell+1)})}\left(1-z_\ell\right)^{n_\ell}\left(1-\frac{1}{z_\ell}\right)^{n_{\ell+1}}\right].
   \end{split}
   \end{equation}
   Note that $f_\ell(u)=u^{k_\ell-k_{\ell-1}}(u+1)^{-(a_\ell+k_\ell)+(a_{\ell-1}+k_{\ell-1})}e^{(t_\ell-t_{\ell-1})u}$ for $u\in\Omega_\LL$ and $f_\ell(v)=u^{-k_\ell+k_{\ell-1}}(v+1)^{(a_\ell+k_\ell)-(a_{\ell-1}+k_{\ell-1})}e^{-(t_\ell-t_{\ell-1})u}$ for $v\in\Omega_\RR$ are both invariant under the translation described above if $\ell\ge 1$. When $\ell=1$, we have $f_1(u)=u^{k_1}(u+1)^{-(a_1+k_1)}e^{t_1u}\to u^{c}(u+1)^{c}f_1(u)$ and $f_1(v)=v^{-k_1}(v+1)^{ a_1+k_1}e^{-t_1v}\to v^{-c}(v+1)^{-c}f_1(v)$. Due to the delta kernel $\delta(-v-1,u)$, we know that the expansion of $\mathcal{D}_{\boldsymbol{n},{Y_{\mathrm{flat}}}}(z_1,\cdots,z_{m-1})$ contains paired factor $\prod_{i_1=1}^{n_1}f_1(u_{i_1}^{(1)})f_{1}(v_{\sigma(i_1)}^{(1)})$ with $u_{i_1}=-1-v_{\sigma(i_1)}^{(1)}$ for some $\sigma\in S_{n_1}$. This factor is invariant since $(u_{i_1}^{(1)})^c(u_{i_1}^{(1)}+1)^{c}(v_{\sigma(i_1)}^{(1)})^{-c}(v_{\sigma(i_1)}^{(1)}+1)^{-c}=1$. These discussions imply that $\mathcal{D}_{\boldsymbol{n},{Y_{\mathrm{flat}}}}(z_1,\cdots,z_{m-1})$, hence $\mathcal{D}_{{Y_{\mathrm{flat}}}}(z_1,\cdots,z_{m-1})$ as well, are invariant under the translation. This finishes the proof.
   
\end{proof}

\paragraph{$\mathcal{D}_Y$ when $m=1$}
\label{sec:one_point}

As we mentioned in Remark~\ref{rmk:01}, we expect that the multi-point distribution formula~\eqref{eq:thm1} at equal times matches the known result of \cite{Borodin-Ferrari-Prahofer-Sasamoto07}. We are not able to verify it at this moment, but we can formally obtain their formula when $m=1$.

Consider  $\mathcal{D}_Y$ when $m=1$. In this case, $\mathcal{D}_Y$ does not have any $z_\ell$ variables and itself gives the one point distribution $\prob_Y\left(x_k(t)\ge a\right)$ (by setting $a_1=a,k_1=k$ and $t_1=t$). By using a conjugation, we could write
\begin{equation*}
\prob_Y\left(x_k(t)\ge a\right)=\mathcal{D}_Y= \det\left.\left(I-K\right)\right|_{\ell^2(\intZ_{\le a-1})}
\end{equation*}
with
\begin{equation}
\label{eq:aux11}
K(x,y)=-\oint_0\ddbarr{v}\oint_{-1}\ddbarr{u} v^{-k}(v+1)^{y+k}e^{-tv} \cdot \Kess_Y(v,u) \cdot u^{k}(u+1)^{-x-k-1}e^{tu}.
\end{equation}
It is not hard (by using Gram-Schmidt process) to prove that there exists a system of ``orthogonal functions'' $e_i(v)$, $i=k,k-1,\cdots,-\infty$, such that
\begin{equation*}
\oint_{0}\ddbarr{v} e_i(v)\cdot v^{-j}(v+1)^{y_j+j} =\delta_i(j),\qquad \text{for all }i,j\le k.
\end{equation*}
Thus formally we could write
\begin{equation*}
v^{-k}(v+1)^{y+k}e^{-tv} =\sum_{j\le k}\left(\oint_{0}\ddbarr{v'} e_j(v') (v')^{-k}(v'+1)^{y+k}e^{-tv'}\right)\cdot v^{-j}(v+1)^{y_j+j}.
\end{equation*}
By plugging it in~\eqref{eq:aux11} and then applying Proposition~\ref{prop:orthogonality}, also noting $\oint_0v^{-j}(v+1)^{y_j+j}\Kess_Y(v,u)\ddbarr{v}=0$ if $j\le 0$ due to the analyticity of $\Kess_Y(v,u)$ on $v$, we obtain
\begin{equation*}
K(x,y)=\sum_{j=1}^k \Psi_j(x)\Phi_j(y)
\end{equation*}
with
\begin{equation*}
\Psi_j(x)=\oint_{-1}\ddbarr{u} u^{k-j}(u+1)^{-x-k-1+y_j+j}e^{tu},\quad \Phi_j(x)=\oint_0\ddbarr{v} e_j(v)v^{-k}(v+1)^{x+k}e^{-tv}.
\end{equation*}
Formally we could verify the following orthogonality by using the above integral representation and the definition of $e_j$
\begin{equation*}
\sum_{x\in\intZ} \Psi_j(x)\Phi_i(x)=\delta_i(j),\qquad \text{for all }i,j\le k.
\end{equation*}
This formulation is consistent with the one point case of the joint distribution formula obtained in \cite{Borodin-Ferrari-Prahofer-Sasamoto07}. We remark that the above calculations are formal since we did not consider the convergence issue.

\paragraph{Two identities about $\mathcal{D}_Y$}

We end this section with two identities about $\mathcal{D}_Y$, which will be used to prove Proposition~\ref{prop:consistency} and Theorem~\ref{thm:periodic_TASEP_large_period} respectively. These identities involve the $\mathcal{D}_Y$ function with different number of variables and parameters. Hence we write 
\begin{equation*}
\mathcal{D}_Y(z_1,\cdots,z_{m-1})=\mathcal{D}_Y(z_1,\cdots,z_{m-1};(a_1,k_1,t_1),\cdots,(a_m,k_m,t_m))
\end{equation*}
to emphasize the parameters if needed.

\begin{prop}
	\label{lm:lemma4} For any fixed $s$ satisfying $1\le s\le m-1$,
	\begin{equation*}
	\begin{split}
	&\oint_{|z_s|<1} \frac{1}{1-z_s}\mathcal{D}_Y(z_1,\cdots,z_{m-1})\ddbar{z_s}-\oint_{|z_s|>1}\frac{1}{1-z_s}\mathcal{D}_Y(z_1,\cdots,z_{m-1})\ddbar{z_s}\\ &=\mathcal{D}_Y(z_1,\cdots,z_{s-1},z_{s+1},\cdots,z_{m-1};(a_1,k_1,t_1),\cdots,(a_{s-1},k_{s-1},t_{s-1}),(a_{s+1},k_{s+1},t_{s+1}),\cdots,(a_m,k_m,t_m))
	\end{split}
	\end{equation*}
	holds when all other $z_\ell\ne 1$, $\ell=1,\cdots,s-1,s+1,\cdots,m-1$, are fixed. Here we remind that the parameters for $\mathcal{D}_Y(z_1,\cdots,z_{m-1})$ are $(a_\ell,k_\ell,t_\ell)$ for $1\le \ell\le m$.
\end{prop}

\begin{prop}
	\label{cor:consistence}
	If $a_s+k_s=\min\{a_\ell+k_\ell:1\le \ell\le m\}<y_N+N$, then
	\begin{equation*}
	\oint_{|z_{m-1}|<1}\frac{1}{1-z_{m-1}} \mathcal{D}_Y(z_1,\cdots,z_{m-1})\ddbar{z_{m-1}} =\mathcal{D}_Y(z_1,\cdots,z_{m-2};(a_1,k_1,t_1),\cdots,(a_{m-1},k_{m-1},t_{m-1}))
	\end{equation*}
	if $s=m$, and
	\begin{equation*}
	\begin{split}
	&\oint_{|z_s|<1} \frac{1}{1-z_s}\mathcal{D}_Y(z_1,\cdots,z_{m-1})\ddbar{z_s}\\
	&=\mathcal{D}_Y(z_1,\cdots,z_{s-1},z_{s+1},\cdots,z_{m-1};(a_1,k_1,t_1),\cdots,(a_{s-1},k_{s-1},t_{s-1}),(a_{s+1},k_{s+1},t_{s+1}),\cdots,(a_m,k_m,t_m))
	\end{split}
	\end{equation*}
	if $1\le s\le m-1$. 
\end{prop}

The proofs of Proposition~\ref{lm:lemma4} and~\ref{cor:consistence} are given in Sections~\ref{sec:proof_lemma4} and~\ref{sec:proof_consistence_cor} respectively.

\subsection{Limit theorems for TASEP with step or flat initial conditions}
\label{sec:limit_theorems}

As an application of Theorem~\ref{thm:main1}, we compute the multi-point limiting distribution of TASEP with two classic initial conditions: the step initial condition and the flat initial condition. We will state the result in terms of the height function of TASEP. Denote $\mathcal{H}$ the space of all possible functions $h:\intZ\to\intZ$
satisfying
\begin{enumerate}
	\item $h(x+1) - h(x) \in \{-1, 1\}$, for all $x\in\intZ$,
	\item $h(0) \in 2\intZ$.
\end{enumerate}
It is well known that TASEP can be viewed as a growth model in $\mathcal{H}$ (it is called the corner growth model). More precisely, we start from some initial function $H(x,0)\in\mathcal{H}$, and let $H(x,t)$ evolve in the following way. We assign each integer site an independent clock. Once the clock associated to some $i$ rings, we increase $H(i,t)$ by $2$ (and keep all other $H(x,t)$ unchanged) if the resulting function $H(x,t)$ is still in $\mathcal{H}$, otherwise we do not change $H(i,t)$. Then we reset the clock. The function $H(x,t)$ is called the height function.

One could also translate the height function $H(x,t)$ in terms of particle locations. See the equation~\eqref{eq:bijection} and the discussions afterward.

\subsubsection{Step initial condition}
\label{sec:limit_thm_step_IC}

We assume that the initial height function is given by 
\begin{equation}
\label{eq:step_IC}
H(x,0)= |x|, \quad x\in\intZ.
\end{equation}
This corresponds to the step initial condition in TASEP. Suppose $m$ is a fixed positive integer, $(x_1,\tau_1),\cdots,(x_m,\tau_m)$ are $m$ distinct points in the half space-time plane $\realR\times\realR_{>0}$ satisfying
\begin{equation*}
\tau_1\le \tau_{2} \le \cdots\le \tau_m
\end{equation*}
and $x_i<x_{i+1}$ if $\tau_i=\tau_{i+1}$ for some $1\le i\le m-1$. Suppose $h_1,\cdots,h_m$ are $m$ fixed real numbers.
\begin{thm}
	\label{thm:limit_step}
	Assume the parameters $m$ and $x_\ell, \tau_\ell, h_\ell$ $(\ell=1,\cdots,m)$ are described above.  With the initial condition~\eqref{eq:step_IC}, we have
	\begin{equation*}
	\lim_{T\to\infty}\prob\left( \bigcap_{\ell=1}^m \left\{ \frac{H\left( 2x_\ell T^{2/3}, 2\tau_\ell T \right) - \tau_\ell T} {-T^{1/3}}  \le h_\ell   \right\} \right)
	=F_{\mathrm{step}}\left(h_1,\cdots,h_m; (x_1,\tau_1),\cdots,(x_m,\tau_m)\right),
	\end{equation*}
	where the function $F_{\mathrm{step}}$ is given by
	\begin{equation}
	\label{eq:Fstep}
	F_{\mathrm{step}}\left(h_1,\cdots,h_m; (x_1,\tau_1),\cdots,(x_m,\tau_m)\right) =\oint\cdots\oint \left[\prod_{\ell=1}^{m-1}\frac{1}{1-z_\ell}\right]\mathrm{D}_{\mathrm{step}}(z_1,\cdots,z_{m-1})\ddbar{z_1}\cdots\ddbar{z_{m-1}}
	\end{equation}
	with $\boldsymbol{z}=(z_1,\cdots,z_{m-1})$, and the contours of integration are circles centered at the origin and of radii less than $1$. The function $\mathrm{D}_{\mathrm{step}}$ in given Definition~\ref{def:Dstep}.
\end{thm}
\begin{rmk}
	Recently, Johansson and Rahman obtained the limiting multi-time distribution for discrete polynuclear growth model \cite{Johansson-Rahman19}, which is the same as a discrete TASEP with  step initial condition. We expect that the above formula~\eqref{eq:Fstep} is equivalent to their result when $\tau_1<\cdots<\tau_m$. However, at the moment we do not have a proof of this equivalence due to the complexity of the formulas. We will consider this proof as a future project.
\end{rmk}
\begin{rmk}
	It is well known that the limiting process along the spatial direction of the height function of TASEP with step initial condition is given by the Airy$_2$ process minus a parabola \cite{Johansson03}. Thus~\eqref{eq:Fstep} when $\tau_1=\cdots=\tau_m$ gives a new formula of the finite dimensional distribution function of the Airy$_2$ process minus a parabola. However, we do not have a direct proof that this formula is equivalent to the original one in the definition of Airy$_2$ process.
\end{rmk}

\subsubsection{Flat initial condition}

We assume that the initial height function is given by 
\begin{equation}
\label{eq:flat_IC}
H(x,0)= \begin{dcases}
1,& x \text{ is odd},\\
0,& x \text{ is even}.
\end{dcases}
\end{equation}
This corresponds to the flat initial condition in TASEP.

\begin{thm}
	\label{thm:limit_flat}
	Assume the parameters $m$ and $x_\ell, \tau_\ell, h_\ell$ $(\ell=1,\cdots,m)$ are the same as in Theorem~\ref{thm:limit_step}.  With the initial condition~\eqref{eq:flat_IC}, we have
	\begin{equation*}
	\lim_{T\to\infty}\prob\left( \bigcap_{\ell=1}^m \left\{ \frac{H\left( 2x_\ell T^{2/3}, 2\tau_\ell T \right) - \tau_\ell T} {-T^{1/3}}  \le h_\ell   \right\} \right)
	=F_{\mathrm{flat}}\left(h_1,\cdots,h_m; (x_1,\tau_1),\cdots,(x_m,\tau_m)\right),
	\end{equation*}
	where the function $F_{\mathrm{flat}}$ is given by
	\begin{equation*}
	F_{\mathrm{flat}}\left(h_1,\cdots,h_m; (x_1,\tau_1),\cdots,(x_m,\tau_m)\right) =\oint\cdots\oint \left[\prod_{\ell=1}^{m-1}\frac{1}{1-z_\ell}\right] \mathrm{D}_{\mathrm{flat}}(z_1,\cdots,z_{m-1})\ddbar{z_1}\cdots\ddbar{z_{m-1}}
	\end{equation*}
	with $\boldsymbol{z}=(z_1,\cdots,z_{m-1})$, and the contours of integration are circles centered at the origin and of radii less than $1$. The function $\mathrm{D}_{\mathrm{flat}}$ is given in Definition~\ref{def:Dflat}.
\end{thm}

\begin{rmk}
	Similarly to the step case, limiting process along the spatial direction of the height function of TASEP with flat initial condition is known and it is called the Airy$_1$ process \cite{Sasamoto05,Borodin-Ferrari-Prahofer-Sasamoto07}. Thus our above result when $\tau_1=\cdots=\tau_m$ gives an equivalent formula  for the finite dimensional  distribution function of the Airy$_1$ process.
\end{rmk}

\begin{rmk}
	Here we only considered the flat case when the particle density $\rho$ is $1/2$. For the general flat case with an arbitrary particle density $\rho$, the one point limiting distribution has been proved in \cite{Ferrari-Occelli18} and it is the same as the case of $\rho=1/2$. We expect the multi-point limiting distribution for the general flat initial condition does not depend on $\rho$ as well and our result above holds for the general flat case.
\end{rmk}
\subsubsection{Functions $\mathrm{D}_{\mathrm{step}}$ and $\mathrm{D}_{\mathrm{flat}}$}

Similarly to their finite time analogs, both functions  $\mathrm{D}_{\mathrm{step}}$ and $\mathrm{D}_{\mathrm{flat}}$ have different representations. Below we only provide a Fredholm determinant representation for each function.

Denote two regions of the complex plane
\begin{equation*}
\complexC_\LL:=\{\zeta\in\complexC: \Re(\zeta)<0\}, \quad \text{ and }
\quad \complexC_\RR:=\{\zeta\in\complexC: \Re(\zeta)>0\}.
\end{equation*}

We first assume that $\tau_1<\cdots<\tau_m$. Later we will need to bend the contours in the definition of $\mathrm{D}_{\mathrm{step}}$ to extend it to the case $\tau_1\le\cdots\le\tau_m$ with extra assumption that $x_i<x_{i+1}$ when $\tau_i=\tau_{i+1}$.

 Let $\mathrm{C}_{m,\LL}^{\out}$, $\cdots$, $\mathrm{C}_{2,\LL}^{\out}$, $\mathrm{C}_{1,\LL}$, $\mathrm{C}_{2,\LL}^{\inn}$, $\cdots$, $\mathrm{C}_{m,\LL}^{\inn}$ be $2m-1$ ``nested'' contours in the region $\complexC_\LL$. They are all unbounded contours from $\infty e^{-2\pi\ii/3}$ to $\infty e^{2\pi\ii/3}$. Moreover, they are located from the right (corresponding to the superscript ``$\out$'') to the left (``$\inn$''). The superscripts ``$\out$'' and ``$\inn$'' should be understood with respect to the point $-\infty$. Similarly, let $\mathrm{C}_{m,\RR}^{\out}$,  $\cdots$, $\mathrm{C}_{2,\RR}^{\out}$, $\mathrm{C}_{1,\RR}$, $\mathrm{C}_{2,\RR}^{\inn}$, $\cdots$, $\mathrm{C}_{m,\RR}^{\inn}$ be $2m-1$ ``nested'' contours from left to right on the half plane $\complexC_\RR$.  They are from $\infty e^{-\pi\ii/3}$ to $\infty e^{\pi\ii/3}$. Their superscripts ``$\out$'' and ``$\inn$'' could be understood with respect to the point $+\infty$. See Figure~\ref{fig:contours_limit} for an illustration of the contours. We remark that these contours are limits of the contours $\Sigma_{\ell,\LL}^{\star}$ and $\Sigma_{\ell,\RR}^{\star}$ near the critical point $-1/2$, here $1\le\ell\le m$ and $\star$ represents the superscript $\out$ or $\inn$ or empty script (when $\ell=1$). The orientations of the contours $\mathrm{C}_{\ell,\RR}^{\star}$ are reversed compared to the contours $\Sigma_{\ell,\RR}^{\star}$. This will lead to a sign difference which passes to the kernel $\mathrm{K}_{\mathrm{step}}$ in Definition~\ref{def:Dstep} or $\mathrm{K}_{\mathrm{flat}}$ in Definition~\ref{def:Dflat}.
 
 	\begin{figure}[t]
 	\centering
 	\begin{tikzpicture}[scale=1]
 	\draw [line width=0.4mm,lightgray] (-3,0)--(3,0) node [pos=1,right,black] {$\realR$};
 	\draw [line width=0.4mm,lightgray] (0,-1)--(0,1) node [pos=1,above,black] {$i\realR$};
 	\fill (0,0) circle[radius=2.5pt] node [below,shift={(0pt,-3pt)}] {$0$};

 	\path [dashed, draw=black,thick,postaction={mid3 arrow={black,scale=1.5}}]	(3,-5.1/6) 
 	to [out=160,in=-90] (1,0)
 	to [out=90,in=-160] (3,5.1/6);
 	\path [draw=black,thick,postaction={mid3 arrow={black,scale=1.5}}]	(2.6,-6.1/6) 
 	to [out=160,in=-90] (0.6,0)
 	to [out=90,in=-160] (2.6,6.1/6);
 	\path [dashed, draw=black,thick,postaction={mid3 arrow={black,scale=1.5}}]	(2.3,-7.1/6) 
 	to [out=160,in=-90] (0.3,0)
 	to [out=90,in=-160] (2.3,7.1/6);
 	
 		\path [draw=black,thick,postaction={mid3 arrow={black,scale=1.5}}]	(-3,-5.1/6) 
 	to [out=20,in=-90] (-1,0)
 	to [out=90,in=-20] (-3,5.1/6);
 	\path [dashed, draw=black,thick,postaction={mid3 arrow={black,scale=1.5}}]	(-2.6,-6.1/6) 
 	to [out=20,in=-90] (-0.6,0)
 	to [out=90,in=-20] (-2.6,6.1/6);
 	\path [draw=black,thick,postaction={mid3 arrow={black,scale=1.5}}]	(-2.3,-7.1/6) 
 	to [out=20,in=-90] (-0.3,0)
 	to [out=90,in=-20] (-2.3,7.1/6);
 	\end{tikzpicture}
 	\caption{Illustration of the contours for $m=2$: The three contours in the left half plane from left to right are $\mathrm{C}_{2,\LL}^{\inn},\mathrm{C}_{1,\LL},\mathrm{C}_{2,\LL}^{\out}$ respectively, the three contours in the right half plane from left to right are $\mathrm{C}_{2,\RR}^{\out},\mathrm{C}_{1,\RR},\mathrm{C}_{2,\RR}^{\inn}$ respectively. $\mathrm{S}_1$ is the union of the three dashed contours, and $\mathrm{S}_2$ is the union of the three solid contours.}\label{fig:contours_limit}
 \end{figure}

We define
\begin{equation*}
\mathrm{C}_{\ell,\LL}:=\mathrm{C}_{\ell,\LL}^\out\cup \mathrm{C}_{\ell,\LL}^\inn, \qquad \mathrm{C}_{\ell,\RR}:=\mathrm{C}_{\ell,\RR}^\out\cup \mathrm{C}_{\ell,\RR}^\inn,\qquad \ell=2,\cdots,m,
\end{equation*}
and
\begin{equation*}
\mathrm{S}_1:= \mathrm{C}_{1,\LL} \cup \mathrm{C}_{2,\RR} \cup \cdots \cup \begin{dcases}
\mathrm{C}_{m,\LL}, & \text{ if $m$ is odd},\\
\mathrm{C}_{m,\RR}, & \text{ if $m$ is even},
\end{dcases}
\end{equation*}
and 
\begin{equation*}
\mathrm{S}_2:= \mathrm{C}_{1,\RR} \cup \mathrm{C}_{2,\LL} \cup \cdots \cup \begin{dcases}
\mathrm{C}_{m,\RR}, & \text{ if $m$ is odd},\\
\mathrm{C}_{m,\LL}, & \text{ if $m$ is even}.
\end{dcases}
\end{equation*}

We introduce a measure on these contours in the same way as in~\eqref{eq:def_dmu}. Let
\begin{equation*}
\dd \mu(\zeta) = \dd \mu_{\boldsymbol{z}} (\zeta) :=
\begin{dcases}
\frac{-z_{\ell-1}}{1-z_{\ell-1}} \ddbarr{\zeta}, & \zeta \in \mathrm{C}_{\ell,\LL}^\out \cup \mathrm{C}_{\ell,\RR}^\out, \quad \ell=2,\cdots,m,\\
\frac{1}{1-z_{\ell-1}} \ddbarr{\zeta}, & \zeta\in \mathrm{C}_{\ell,\LL}^\inn \cup \mathrm{C}_{\ell,\RR}^\inn, \quad \ell=2,\cdots,m,\\
\ddbarr{\zeta}, & \zeta \in \mathrm{C}_{1,\LL} \cup \mathrm{C}_{1,\RR}.
\end{dcases}
\end{equation*}

We will define $\mathrm{D}_{\mathrm{step}}$ and $\mathrm{D}_{\mathrm{flat}}$ in terms of Fredholm determinants. Recall the $Q_1$ and $Q_2$ functions defined in~\eqref{eq:Q},
\begin{equation*}
Q_1(j) :=\begin{dcases}
1-z_{j}, & \text{ if $j$ is odd and $j<m$},\\
1-\frac{1}{z_{j-1}}, & \text{ if $j$ is even},\\
1,& \text{if $j=m$ is odd},
\end{dcases}
\qquad 
Q_2(j) :=\begin{dcases}
1-z_{j}, & \text{ if $j$ is even and $j<m$},\\
1-\frac{1}{z_{j-1}}, & \text{ if $j$ is odd and $j>1$},\\
1,& \text{if $j=m$ is even, or $j=1$}.
\end{dcases}
\end{equation*}

\begin{defn}
	\label{def:Dstep}
	We define
	\begin{equation*}
	\mathrm{D}_{\mathrm{step}}(z_1,\cdots,z_{m-1})=\det\left(I-\mathrm{K}_1\mathrm{K}_{\mathrm{step}}\right),
	\end{equation*}
	where the operators
	\begin{equation*}
	\mathrm{K}_1: L^2(\mathrm{S}_2,\dd\mu) \to L^2(\mathrm{S}_1,\dd\mu),\qquad \mathrm{K}_{\mathrm{step}}: L^2(\mathrm{S}_1,\dd\mu)\to L^2(\mathrm{S}_2,\dd\mu)
	\end{equation*}
	are defined by their kernels
	\begin{equation}
	\label{eq:K1_lim}
	\mathrm{K}_1(\zeta,\zeta'):= \left(\delta_i(j) + \delta_i( j+ (-1)^i)\right) \frac{ \mathrm{f}_i(\zeta) }{\zeta-\zeta'} Q_1(j)
	\end{equation}
	and 
	\begin{equation*}
	\mathrm{K}_{\mathrm{step}}(\zeta',\zeta):=
	\left(\delta_j (i) + \delta_j(i - (-1)^j)\right) \frac{ \mathrm{f}_j(\zeta') }{-\zeta'+\zeta} Q_2(i) 
	\end{equation*}
	for any $\zeta\in (\mathrm{C}_{i,\LL}\cup \mathrm{C}_{i,\RR}) \cap \mathrm{S}_1$ and $\zeta'\in (\mathrm{C}_{j,\LL}\cup \mathrm{C}_{j,\RR}) \cap \mathrm{S}_2$ with $1\le i,j\le m$. Here the function
	\begin{equation}
	\label{eq:def_fi}
	\mathrm{f}_i(\zeta):=\begin{dcases}
	\frac{\mathrm{F}_i(\zeta)}{ \mathrm{F}_{i-1}(\zeta)}, & \Re(\zeta)<0,\\
	\frac{ \mathrm{F}_{i-1}(\zeta)}{\mathrm{F}_i(\zeta)}, & \Re(\zeta)>0,
	\end{dcases}
	\end{equation}
	with
	\begin{equation}
	\label{eq:Fi}
	\mathrm{F}_i(\zeta) := \begin{dcases}
	e^{-\frac{1}{3}\tau_i \zeta^3 + x_i \zeta^2 + h_i \zeta}, & i=1,\cdots,m,\\
	1, & i=0.
	\end{dcases}
	\end{equation}
\end{defn}

Now we extend the function $\mathrm{D}_{\mathrm{step}}$ to the case when some of the $\tau_i$'s are equal. It is extended as follows. We adjust the angles of the contours on the right half plane. We let $\mathrm{C}_{m,\RR}^{\out}$,  $\cdots$, $\mathrm{C}_{2,\RR}^{\out}$, $\mathrm{C}_{1,\RR}$, $\mathrm{C}_{2,\RR}^{\inn}$, $\cdots$, $\mathrm{C}_{m,\RR}^{\inn}$ be  from $\infty e^{-\pi\ii/5}$ to $\infty e^{\pi\ii/5}$. We keep the contours on the left plane unchanged. Note that this adjustment does not affect the $\mathrm{D}_{\mathrm{step}}(z_1,\cdots,z_{m-1})$ when all the $\tau_i$'s are different, in other words, we could have chosen these contours from the beginning but we did not make this choice since the contours are not symmetric anymore. This asymmetry intuitively comes from the fact that we have two different orders of $x_i$ and $x_{i+1}$ when $\tau_i=\tau_{i+1}$. If we choose a different choice of order, $x_i>x_{i+1}$ for all $i$ satisfying $\tau_i=\tau_{i+1}$, we need to bend all the contours on the left plane instead.

With this adjustment of the contours, it is easy to verify that the functions $\mathrm{f}_i$ decay super-exponentially fast along all these contours. The function $\mathrm{D}_{\mathrm{step}}(z_1,\cdots,z_{m-1})$ hence is well-defined. Moreover, since the integrand are continuous on the parameters $h_\ell,x_\ell,\tau_\ell$, $1\le \ell\le m$, the function $\mathrm{D}_{\mathrm{step}}(z_1,\cdots,z_{m-1})$ is also continuous on these parameters in the following way. If the parameters are continuous functions of $t$, more explicitly, they move along continuous curves $h_\ell(t),x_\ell(t),\tau_\ell(t)$, $1\le\ell\le m$, $0\le t\le 1$, satisfying $\tau_1(t)\le\cdots\le\tau_m(t)$ and $x_i(t)<x_{i+1}(t)$ when $\tau_i(t)=\tau_{i+1}(t)$, then $\mathrm{D}_{\mathrm{step}}(z_1,\cdots,z_{m-1})$ is also continuous in $t$.

\bigskip

Similarly, we first define $\mathrm{D}_{\mathrm{flat}}$ when $\tau_1<\cdots<\tau_m$.

\begin{defn}
	\label{def:Dflat}
    In order to define $\mathrm{D}_{\mathrm{flat}}$, we further assume that two contours $\mathrm{C}_{1,\LL}$ and $\mathrm{C}_{1,\RR}$ are symmetric about the imaginary axis. In other words, $\mathrm{C}_{1,\LL}=-\mathrm{C}_{1,\RR}:=\{-\eta: \eta\in\mathrm{C}_{1,\RR}\}$. We define
    \begin{equation*}
    \mathrm{D}_{\mathrm{flat}}(z_1,\cdots,z_{m-1})=\det\left(I-\mathrm{K}_1\mathrm{K}_{\mathrm{flat}}\right),
    \end{equation*}
    where the operators
	\begin{equation*}
	\mathrm{K}_1: L^2(\mathrm{S}_2,\dd\mu) \to L^2(\mathrm{S}_1,\dd\mu),\qquad \mathrm{K}_{\mathrm{flat}}: L^2(\mathrm{S}_1,\dd\mu)\to L^2(\mathrm{S}_2,\dd\mu)
	\end{equation*}
	are defined by their kernels described as follows. The kernel $\mathrm{K}_1$ is the same as in~\eqref{eq:K1_lim}, while $\mathrm{K}_{\mathrm{flat}}$ is defined by
	\begin{equation*}
	\mathrm{K}_{\mathrm{flat}}(\zeta',\zeta):= \begin{dcases}
	\left(\delta_j (i) + \delta_j(i - (-1)^j)\right) \frac{ \mathrm{f}_j(\zeta') }{-\zeta'+\zeta} Q_2(i), & i\ge 2,\\
	-\delta_j(1) \mathrm{f}_j(\zeta') \delta(-\zeta', \zeta), & i=1,
	\end{dcases}
	\end{equation*}
	for any $\zeta\in (\mathrm{C}_{i,\LL}\cup \mathrm{C}_{i,\RR}) \cap \mathrm{S}_1$ and $\zeta'\in (\mathrm{C}_{j,\LL}\cup \mathrm{C}_{j,\RR}) \cap \mathrm{S}_2$ with $1\le i,j\le m$, where $\mathrm{f}_i$ is the same as in~\eqref{eq:def_fi} and the kernel $\delta$ is a delta kernel defined  by
	\begin{equation*}
	\int_{\mathrm{C}_{1,\LL}} \delta(-\eta,\xi) f(\xi) \ddbarr{\xi} = f(-\eta)
	\end{equation*}
	for any function  $f\in L^2(\mathrm{C}_{1,\LL},\ddbarr{\xi})$ and any $\eta\in\mathrm{C}_{1,\RR}$.
\end{defn}

The extension of $\mathrm{D}_{\mathrm{flat}}$ to $\tau_1\le\cdots\le\tau_m$ is more complicated. We need to use the series expansion
  \begin{equation*}
	\mathrm{D}_{{\mathrm{flat}}}(z_1,\cdots,z_{m-1}):=\sum_{\boldsymbol{n}\in(\intZ_{\ge 0})^m}\frac{1}{(\boldsymbol{n}!)^2}\mathrm{D}_{\boldsymbol{n},{{\mathrm{flat}}}}(z_1,\cdots,z_{m-1})
\end{equation*}
with 
\begin{equation}
\label{eq:2021_10_06}
	\begin{split}
		&\mathrm{D}_{\boldsymbol{n},{{\mathrm{flat}}}}(z_1,\cdots,z_{m-1})\\
		&=\prod_{\ell=1}^m\prod_{i_\ell=1}^{n_\ell} \int_{\mathrm{C}_{\ell,\LL}}\mathrm{d}\mu_{\boldsymbol{z}}(\xi_{i_\ell}^{(\ell)})\int_{\mathrm{C}_{\ell,\RR}}\mathrm{d}\mu_{\boldsymbol{z}}(\eta_{i_\ell}^{(\ell)})\\
		&\quad \left[(-1)^{n_1(n_1+1)/2}\frac{\Delta(\boldsymbol{\xi}^{(1)};\boldsymbol{\eta}^{(1)})}{\Delta(\boldsymbol{\xi}^{(1)})\Delta(\boldsymbol{\eta}^{(1)})} \det\left[\delta(-\eta_i^{(1)},\xi_j^{(1)})\right]_{i,j=1}^{n_1}\right]\\
		&\quad\cdot \left[\prod_{\ell=1}^m\frac{(\Delta(\boldsymbol{\xi}^{(\ell)}))^2(\Delta(\boldsymbol{\eta}^{(\ell)}))^2}{(\Delta(\boldsymbol{\xi}^{(\ell)};\boldsymbol{\eta}^{(\ell)}))^2}\mathrm{f}_\ell(\boldsymbol{\xi}^{(\ell)}) \mathrm{f}_\ell(\boldsymbol{\eta}^{(\ell)})\right]\\
		&\quad\cdot \left[\prod_{\ell=1}^{m-1} \frac{\Delta(\boldsymbol{\xi}^{(\ell)};\boldsymbol{\eta}^{(\ell+1)})\Delta(\boldsymbol{\eta}^{(\ell)};\boldsymbol{\xi}^{(\ell+1)})}{\Delta(\boldsymbol{\xi}^{(\ell)};\boldsymbol{\xi}^{(\ell+1)})\Delta(\boldsymbol{\eta}^{(\ell)};\boldsymbol{\eta}^{(\ell+1)})}\left(1-z_\ell\right)^{n_\ell}\left(1-\frac{1}{z_\ell}\right)^{n_{\ell+1}}\right].
	\end{split}
\end{equation}
Here similar to the step case, we let $\mathrm{C}_{m,\RR}^{\out}$,  $\cdots$, $\mathrm{C}_{2,\RR}^{\out}$, $\mathrm{C}_{1,\RR}$, $\mathrm{C}_{2,\RR}^{\inn}$, $\cdots$, $\mathrm{C}_{m,\RR}^{\inn}$ be  from $\infty e^{-\pi\ii/5}$ to $\infty e^{\pi\ii/5}$. We keep the contours on the left plane unchanged. Note that in this case we need to understand the $\delta(-\eta_i^{(1)},\xi_j^{(1)})$ in the following way since the contour $\mathrm{C}_{1,\LL}$ is not $-\mathrm{C}_{1,\RR}$ at this moment.

For $\ell\ge 2$, we write each combination of integrals as
\begin{equation}
\label{eq:aux_101}
	\frac{1}{1-z_{\ell-1}}\int_{\mathrm{C}_{\ell,\LL}^\inn} \ddbarr{\xi_{i_\ell}^{(\ell)}}
	-\frac{z_{\ell-1}}{1-z_{\ell-1}}\int_{\mathrm{C}_{\ell,\LL}^\out} \ddbarr{\xi_{i_\ell}^{(\ell)}} = \int_{\mathrm{C}_{\ell,\LL}^\out} \ddbarr{\xi_{i_\ell}^{(\ell)}} -\frac{1}{1-z_{\ell-1}}\left(\int_{\mathrm{C}_{\ell,\LL}^\out} \ddbarr{\xi_{i_\ell}^{(\ell)}}-\int_{\mathrm{C}_{\ell,\LL}^\inn} \ddbarr{\xi_{i_\ell}^{(\ell)}}\right)
\end{equation}
and then expand the integrals accordingly. After writing the $\int_{\mathrm{C}_{\ell,\LL}^\out}-\int_{\mathrm{C}_{\ell,\LL}^\inn}$ as a residue evaluation at $\xi_{i_\ell}^{(\ell)}=\xi_{i_{\ell-1}}^{(\ell-1)}$ at some pole $\xi_{i_{\ell-1}}^{(\ell-1)}$ or zero if there is no such a pole, we end with a summation of $2^{n_2+\cdots+n_m}$ possible combinations, each of which is a combination of integrals without involving the contours $\int_{\mathrm{C}_{\ell,\LL}^\inn}$:
\begin{equation}
\label{eq:aux_102}
\sum \prod_{\substack{\text{some }i_\ell\\ \ell\ge 2}}\int_{\mathrm{C}_{\ell,\LL}^\out} \ddbarr{\xi_{i_\ell}^{(\ell)}}\prod_{{i_1}=1}^{n_{1}} \int_{\mathrm{C}_{1,\LL}} \ddbarr{\xi_{i_1}^{(1)}}\prod_{\ell=2}^{m}  \prod_{{i_\ell}=1}^{n_{\ell}} \left[\frac{1}{1-z_{\ell-1}}\int_{\mathrm{C}_{\ell,\RR}^\inn} \ddbarr{\eta_{i_\ell}^{(\ell)}}
-\frac{z_{\ell-1}}{1-z_{\ell-1}}\int_{\mathrm{C}_{\ell,\RR}^\out} \ddbarr{\eta_{i_\ell}^{(\ell)}} \right]
\cdot \prod_{{i_1}=1}^{n_{1}} \int_{\mathrm{C}_{1,\RR}} \ddbarr{\eta_{i_1}^{(1)}}.
\end{equation}
Here we ignored the integrand which is the same as in~\eqref{eq:2021_10_06} except that we evaluated the residues for some $\xi_{i_\ell}^{(\ell)}$'s that are from the contours $\mathrm{C}_{\ell,\LL}^\inn$.
Now we bend the contour of $\mathrm{C}_{1,\LL}$ such that $\mathrm{C}_{1,\LL}=-\mathrm{C}_{1,\RR}$. Note that  $\mathrm{C}_{\ell,\LL}^\out$ is wider (from $\infty e^{-2\mathrm{i}\pi/3}$ to $\infty e^{2\mathrm{i}\pi/3}$) and always lies outside of $\mathrm{C}_{1,\LL}$ when we bend $\mathrm{C}_{1,\LL}$.

We remark that for the parameters satisfying $\tau_1\le\cdots\le\tau_m$ and $x_i<x_{i+1}$ when $\tau_i=\tau_{i+1}$, the integrand decays super-exponentially fast along the remaining contours. Hence $\mathrm{D}_{\mathrm{flat}}(z_1,\cdots,z_{m-1})$ is well-defined. It also matches Definition~\ref{def:Dflat} since the above adjustment of the contours does not affect the integrals when $\tau_1<\cdots<\tau_m$.

Similarly to $\mathrm{D}_{\mathrm{step}}$, the function $\mathrm{D}_{\mathrm{flat}}(z_1,\cdots,z_{m-1})$ is also continuous on the parameters $h_\ell,x_\ell,\tau_\ell$, $1\le\ell\le m$, in the following way. If the parameters are continuous functions of $t$, more explicitly, they move along continuous curves $h_\ell(t),x_\ell(t),\tau_\ell(t)$, $1\le\ell\le m$, $0\le t\le 1$, satisfying $\tau_1(t)\le\cdots\le\tau_m(t)$ and $x_i(t)<x_{i+1}(t)$ when $\tau_i(t)=\tau_{i+1}(t)$, then $\mathrm{D}_{\mathrm{flat}}(z_1,\cdots,z_{m-1})$ is also continuous in $t$.
\section{Periodic TASEP with large period}
\label{sec:pTASEP}

Periodic TASEP can be viewed as TASEP on a periodic domain
\begin{equation*}
\conf_N(L):=\{(x_1,\cdots,x_N)\in\intZ^N: x_N<x_{N-1}<\cdots<x_1<x_N+L\},
\end{equation*}
where $L$ is some integer larger than $N$. We call $L$ the \emph{period} of the system, and $N$ is the number of particles of the system. We label the particles from right to left, and denote $x_i^{\period}(t)$ the location of the $i$-th particle, $1\le i\le N$. Here the superscript $\period$ indicates that it is for periodic TASEP with period $L$. The evolution of the system is exactly the same as TASEP, except that the rightmost particle cannot make its jump if its distance to the leftmost particle is exact $L-1$ at the moment of attempting to jump. In other words, the rightmost particle could also be blocked by the leftmost particle such that their distance is always less than the period $L$. One could naturally make infinitely many copies of these particles and place them in all the intervals of length $L$ in a periodic way. With this setting, each particle moves independently to the right and can only be blocked by its right neighboring particle, except for the dependence from the periodicity $x_i^{\period}(t) = x_{i+N}^{\period}(t) + L$, $i\in\intZ, t\in\realR_{\ge0}$. This explains why we call this model periodic TASEP.

Recently, Baik and Liu studied periodic TASEP in a sequence of papers \cite{Baik-Liu16, Baik-Liu16b, Liu16,Baik-Liu19, Baik-Liu21}. In their most recent work \cite{Baik-Liu19,Baik-Liu21}, they obtained two multi-point distribution formulas for periodic TASEP, both of which are in terms of multiple contour integrals on the complex plane. The two formulas differ in their integrands: One involves a Toeplitz-like determinant of large size, with entries given by a huge summations over the so-called Bethe roots, while the other involves a Fredholm determinant on a space of Bethe roots. They then evaluated the limit of this multi-point distribution in the so-called relaxation time scale by using the second formula, with certain assumptions on the initial condition. They were also able to verify that several classic initial conditions satisfy these assumptions.

The main goal of this section is to investigate how the Fredholm determinant formula of multi-point distribution for periodic TASEP behaves when the period becomes large. It is known that periodic TASEP has the same dynamics as TASEP when the periodicity constraint does not take effect. In other words, the finite time distributions of periodic TASEP should be equal to their analogs of TASEP when the period becomes large. This is the key fact and the starting point of this paper.

\vspace{0.3cm}

We use  $\prob_Y^{\period}$ to denote the probability of periodic TASEP, here $Y\in\conf_N(L)$ is the initial configuration of particle locations. We will also use $\prob_Y=\prob_Y^{(\infty)}$ to denote the probability of TASEP with initial configuration $Y\in\conf_N=\conf_N(\infty)$.

\begin{thm}\cite{Baik-Liu19}
	\label{thm:comparison}
	Suppose $Y=(y_1,\cdots,y_N)\in\conf_N$. Let $L>N$ such that $Y\in\conf_N(L)$. In other words, $L\ge y_1-y_N+1$. Consider periodic TASEP with period $L$ and initial configuration $Y$, and an independent TASEP with the same initial configuration. We use  $x_i^{\period}(t)$ and $x_i(t)$ to denote the particle locations in the two models respectively. Suppose $m$ is a positive integer, $k_1,\cdots,k_m$ are $m$ integers in $\{1,\cdots,N\}$, and $t_1,\cdots,t_m$ are $m$ non-negative real numbers. Then for any integers $a_1,\cdots,a_m$ we have
	\begin{equation*}
	\prob_Y^{\period}\left(\bigcap_{\ell=1}^m\left\{ x_{k_\ell}^{\period}(t_\ell) \ge a_\ell\right\}\right)=	\prob_Y \left(\bigcap_{\ell=1}^m\left\{ x_{k_\ell} (t_\ell) \ge a_\ell\right\}\right)
	\end{equation*}
	provided
	\begin{equation}
	\label{eq:condition_L}
	L \ge \max\{a_1+k_1,\cdots,a_m+k_m\} - y_N.
	\end{equation}
\end{thm}

Intuitively, this theorem means that if the considered particles $x_\ell^{(L)}(t_\ell)$ have not been delayed by the leftmost particle of the previous period $y_N+L$, then the dynamics of these particles are the same as if the previous periods do not exist. An equivalent theorem which considers the probability of events $\{x_{k_\ell}^{\period}(t_\ell) \le a_\ell\}$ was given in Lemma 8.1 of \cite{Baik-Liu19}. The statement we present above was also discussed there, see the equations (8.5) and (8.6) after Lemma 8.1 in \cite{Baik-Liu19}.   We remark that the particle labels in \cite{Baik-Liu19} are from left to right, which is different from this paper. Thus one needs to change the particle labels accordingly in (8.5) and (8.6) of that paper to match Theorem~\ref{thm:comparison}.

\vspace{0.3cm}
The above theorem implies that the formula of multi-point distribution in periodic TASEP should be independent of the parameter $L$ when $L$ satisfies~\eqref{eq:condition_L}. However, the existing formulas in \cite{Baik-Liu19,Baik-Liu21} all have a discrete feature and contain the parameter $L$. Below we provide a new multi-point distribution formula for periodic TASEP when~\eqref{eq:condition_L} holds. This formula is independent of the parameter $L$ and does not have a discrete structure involving the so-called Bethe roots.

\begin{thm}[Multi-point distribution of periodic TASEP with large period]
	\label{thm:periodic_TASEP_large_period}
	With the same setting as Theorem~\ref{thm:comparison}. Suppose the period $L$ satisfies~\eqref{eq:condition_L}.
	Then
	\begin{equation}
	\label{eq:multi_TASEP}
	\prob_Y^{\period}\left(\bigcap_{\ell=1}^m\left\{ x_{k_\ell}^{\period}(t_\ell) \ge a_\ell\right\}\right)= \oint\cdots\oint \left[\prod_{\ell=1}^{m-1}\frac{1}{1-z_\ell}\right] \mathcal{D}_Y(z_1,\cdots,z_{m-1}) \ddbar{z_1}\cdots\ddbar{z_{m-1}},
	\end{equation}
	where the contours of integration are circles centered at the origin and of radii less than $1$. The function $\mathcal{D}_Y(z_1,\cdots,z_{m-1})$ is defined in terms of a Fredholm determinant in Definition~\ref{def:operators_K1Y}, or equivalently in terms of series expansion in Definition~\ref{def:D_Y_series}.
\end{thm}

We remark that although Theorem~\ref{thm:main1} is the main result of the paper, technically Theorem~\ref{thm:periodic_TASEP_large_period} is the key result. The main challenging part to obtain such a theorem is (1) to understand why the discrete structure does not play a role in the formulas obtained in \cite{Baik-Liu19,Baik-Liu21} when~\eqref{eq:condition_L} holds, and (2) to find an alternate formula which preserves all other features except for the discreteness structure. This formula is exactly the right hand side of~\eqref{eq:multi_TASEP}. Finding this formula is constructive: It is not obtained by taking the large $L$ limit of periodic formulas\footnote{It might be able to take a large $L$ limit and find the limit of periodic TASEP formulas. However, in our opinion, it is the algebraic structure instead of asymptotic behavior that allows us to remove the $L$ parameter. The condition~\eqref{eq:condition_L} indeed provides a hint: The lower bound of $L$ to remove the discreteness is a finite number instead of going to infinity.}. Instead, it is obtained by construction and then proved by induction.

The proof of Theorem~\ref{thm:periodic_TASEP_large_period} is given in Section~\ref{sec:proof_Theorem}.

\section{Cauchy-type summation over nested roots}
\label{sec:Cauchy_sum}

This is a key portion of the proof of Theorem~\ref{thm:periodic_TASEP_large_period}. More explicitly, the main results, Propositions~\ref{prop:Cauchy_sum_over_roots} and~\ref{prop:Cauchy_sum_over_roots_ext} in this section will be used in the proof of Lemma~\ref{lm:lemma2}, which is the most technical part in the proof of Theorem~\ref{thm:periodic_TASEP_large_period}. Since it is independent of the TASEP model, and it might be applicable to other problems (see \cite{Liu21} for an application to find the distribution of the geodesic in the directed last passage percolation), we put it in this separate section.

In this section, we will study a multiple sum over roots of $q(w)=\hat z_i$ for some $\hat z_i$'s with decreasing magnitudes, where $q(w)$ is an analytic function in the considered domain with some assumptions around its zero. Examples of such $q(w)$ functions are $q(w)=w^N$ and $q(w)=(w+1)^{L-N}$ for one region case, or $q(w)=w^N(w+1)^{L-N}$ for the two-region case which we will consider later. The summand involves factors
\begin{equation}
\label{eq:def_cauchy}
\Ch(W;W'):= \frac{\Delta(W) \Delta(W')}{\Delta(W;W')}
\end{equation}
for some vectors $W$ and $W'$, whose coordinates will be chosen from the roots of $q(w)=\hat z$ and $q(w)=\hat z'$ respectively. The notations $\Delta(W)$ and $\Delta(W;W')$ are introduced at the beginning of Section~\ref{sec:series_expansion}. We remind that
\begin{equation*}
\Delta(W)= \prod_{1\le i<j\le n}(w_j-w_i),\qquad \Delta(W;W')=\prod_{i=1}^n \prod_{i'=1}^{n'} (w_i -w'_{i'}),
\end{equation*}
where $n$ and $n'$ are the sizes of the vectors $W$ and $W'$ respectively, and $w_i (1\le i\le n)$, $w'_{i'} (1\le i'\le n')$ are the coordinates of $W$ and $W'$ respectively. Here we allow $n=0$ or $n'=0$ by defining the empty product to be $1$.

Especially, when $W$ and $W'$ have the same size, $\Ch(W;W')$ is the Cauchy determinant up to the sign
\begin{equation*}
\Ch(W;W')=(-1)^{n(n-1)/2}\det\left[\frac{1}{w_i-w'_{i'}}\right]_{\substack{1\le i\le n\\ 1\le i'\le n'=n}}.
\end{equation*}
Hence we call~\eqref{eq:def_cauchy} the \emph{Cauchy-type factor}, and the summation involving these factors \emph{Cauchy-type summation}. 

To explicitly state the Cauchy-type summation to be considered, we introduce some notations.

Let $m\ge 1$ be a fixed integer. Suppose $n_1,\cdots,n_m$ are non-negative integers. We also suppose $W^{(\ell)} = (w_1^{(\ell)},\cdots, w_{n_\ell}^{(\ell)})$ be a vector of $n_\ell$ variables, $1\le \ell \le m$. 

Assume $I^{(1)}\cdots,I^{(m-1)}$ and $J^{(2)},\cdots,J^{(m)}$ are $2m-2$ sets satisfying
\begin{equation}
\label{eq:IJ_ell}
I^{(\ell)} \subseteq \{1,\cdots,n_\ell\}, \quad J^{(\ell+1)}\subseteq \{1,\cdots,n_{\ell+1}\}
\end{equation}
for each $\ell=1,\cdots,m-1$. We also introduce a convention that $W_I$ is a vector obtained by keeping all the coordinates of $W$ whose indices are in the set $I$ and removing all other variables. For example, if $W=(w_1,\cdots,w_{10})$, then $W_{\{2,3\}} = (w_2,w_3)$. Thus by using this convention, $W^{(\ell)}_{I^{(\ell)}}$ is a vector with coordinates in $W^{(\ell)}$ whose subscripts are in $I^{(\ell)}$, and $W^{(\ell+1)}_{J^{(\ell+1)}}$ is similarly a vector with coordinates in $W^{(\ell+1)}$ whose subscripts are in $J^{(\ell+1)}$. 

We will consider the following summand
\begin{equation}
\label{eq:H}
H(W^{(1)},\cdots,W^{(m)};z_0,\cdots,z_{m-1}):= \left[ \prod_{\ell=1}^{m-1} \Ch\left(W^{(\ell)}_{I^{(\ell)}}; W^{(\ell+1)}_{J^{(\ell+1)}}\right) \right] \cdot A(W^{(1)},\cdots,W^{(m)};z_0,\cdots,z_{m-1}),
\end{equation}
where $A$ is a function satisfying certain analyticity on its variables (the coordinates of all $W^{(\ell)}$ vectors and complex numbers $z_\ell$'s). Note that $H$ defined above  is dependent on the sets $I^{(\ell)}, J^{(\ell+1)}$, $1\le \ell \le m-1$, and the function $A$.

\vspace{0.2cm}

Now we introduce the space where the above summand is defined.

Let $r_\mm\in (0,1)$ be a fixed number. We assume that $z_0\in \disk(r_\mm)$ and $z_\ell\in \disk=\disk(1)$. Here the notation 
\begin{equation}
\label{eq:diskr}
\disk(r):=\{z\in\complexC: |z|<r\}.
\end{equation}
We also denote
\begin{equation}
\label{eq:disk0r}
\disk_0(r):=\{z\in\complexC: 0<|z|<r\}
\end{equation}
the punctured open disk with radius $r$ and centered at the origin.

Suppose $\Omega$ is a simply connected region in the complex plane which contains $0$. Let
\begin{equation*}
\Omega_0:=\Omega\setminus\{0\}.
\end{equation*}

We assume that $A$ is an analytic function defined on $(\Omega_0)^d\times \disk(r_\mm)\times \disk^{m-1}$, with $d=d\left(W^{(1)},\cdots,W^{(m)}\right)$ is the total dimension of the vectors. Here we have $d=n_1+\cdots+n_m$ since $W^{(\ell)}$ has $n_\ell$ coordinates.

With the above assumption, it is clear that $H$ is analytic function on $(\Omega_0)^d\times \disk(r_\mm)\times \disk^{m-1}$ except that it has poles at $w_{i}^{(\ell)}= w_{j}^{(\ell+1)}$ for $(i,j)\in I^{(\ell)}\times J^{(\ell+1)}$ and some $1\le \ell\le m-1$.

\vspace{0.3cm}

We will take the sum over discrete sets determined by a function $q(w)$. Now we introduce $q(w)$ and the discrete sets.

Assume that $q(w)$ is an analytic function of $w\in\Omega$ such that the ``level curves" of $q(w)$ in $\Omega$, the $\Gamma_r$'s defined below for $0<r<r_\mm$, are nested simple closed contours enclosing $0$. More precisely, for any $0<r<r_\mm$, 
\begin{equation}
\label{eq:Gamma_r}
\Gamma_r:=\{w\in\Omega: |q(w)| = r\}
\end{equation}
is a simple closed contour enclosing $0$, and $\Gamma_r$ encloses $\Gamma_{r'}$ if $0<r'<r<r_\mm$. We also define
\begin{equation}
\label{eq:roots_hat}
\roots_{\hat z} :=\{ w\in \Omega: q(w) = \hat z\}
\end{equation}
for any $|\hat z|<r_\mm$. It is obvious that all elements of  $\roots_{\hat z}$ lie on the contour $\Gamma_{|\hat z|}$. We remark that these assumptions imply that $q(0)=0$. Thus we  set $\Gamma_0=\{0\}$ and $\roots_{0}=\{0\}$. By using the property that $\Gamma_r$ are nested simple closed contours for $0<r<r_\mm$, we know that $q'(w)\ne 0$ for all $w\in\roots_{\hat z}$ provided $\hat z\in\disk_0(r_\mm)$.

\vspace{0.5cm}
Finally we are ready to introduce the summation. We assume $z_0\in\disk_0(r_\mm)$ and $z_\ell\in\disk_0$ for $1\le \ell\le m-1$. In other words, $0<|z_0|<r_\mm$ and $0<|z_\ell|<1$ for $1\le \ell\le m-1$. We define
\begin{equation}
\label{eq:GG}
\GG(z_0,\cdots,z_{m-1}):= \sum_{\substack{W^{(1)} \in \roots^{n_1}_{\hat z_1} }}\cdots\sum_{\substack{W^{(m)} \in \roots^{n_m}_{\hat z_m} }} \left[\prod_{\ell=1}^mJ(W^{(\ell)})\right]\cdot H(W^{(1)},\cdots,W^{(m)};z_0,\cdots,z_{m-1}),
\end{equation}
where $J$ is defined via $q$ as follows
\begin{equation}
\label{eq:J}
J(W^{(\ell)}):=\prod_{i_\ell=1}^{n_\ell} J(w_{i_\ell}^{(\ell)}) := \prod_{i_\ell=1}^{n_\ell} \frac{q(w_{i_\ell}^{(\ell)})}{q'(w_{i_\ell}^{(\ell)})},
\end{equation}
and $\hat z_\ell$'s are defined via $z_\ell$'s
\begin{equation}
\label{eq:hatz}
\hat z_\ell = z_0z_1\cdots z_{\ell-1},\quad 1\le \ell\le m.
\end{equation}
Note that our assumption on $z_\ell$'s implies $\hat z_1,\cdots,\hat z_m$ are $m$ points in $\disk_0(r_\mm)$ with decreasing norms: $0<|\hat z_m|<\cdots<|\hat z_1|<r_\mm$.

Recall the definition of the function $H$ in~\eqref{eq:H}, which is analytic for $z_\ell\in\disk$ and for coordinates of $W^{(\ell)}$'s except for the possible poles at $w_i^{(\ell)}=w_j^{(\ell+1)}$. Since $\hat z_\ell$ are distinct for all $(z_0,\cdots,z_m)\in\disk_0(r_\mm)\times\disk_0^{m-1}$, and the coordinates of $W^{(\ell)}$ are roots of $q(w)=\hat z_\ell$ which depends on $z_0,\cdots,z_{\ell-1}$ analytically, the summand in~\eqref{eq:GG} can be viewed as an analytic function for $(z_0,\cdots,z_m)\in\disk_0(r_\mm)\times\disk_0^{m-1}$. 
Thus $\GG(z_0,\cdots,z_{m-1})$ is analytic in this region as well. The main goal of this section is to investigate the behavior of $\GG$ when $z_\ell\to 0$ and see whether the analyticity of $\GG$ can be extended to the space
$\disk(r_\mm)\times\disk^{m-1}$. Note that $z_\ell=0$ implies $\hat z_{\ell+1}=\cdots=\hat z_m=0$, which is not considered in the definition $\GG$ in~\eqref{eq:GG}. To extend the function $\GG$ to $z_\ell=0$, we need to consider 
possible singularities: $A(W^{(1)},\cdots,W^{(m)};z_0,\cdots,z_{m-1})$ may have singularities at $w_{i_\ell}^{(\ell)} = 0\in\roots_0$, and the Cauchy-type factors in $H$ bring singularities at $w_{i}^{(\ell)}=w_{j}^{(\ell+1)}$. It turns out that if $q(w)$ is a function such that these singularities disappear, then $\GG$ can be analytically extended to $z_\ell=0$ for all $\ell$. More surprisingly, for such $q(w)$ functions, $\GG(0,z_1,\cdots,z_{m-1})$ is actually independent of $q$.

\vspace{0.2cm}
To explain the conditions of $q$ such that $\GG$ can be analytically extended to $\disk(r_\mm)\times\disk^{m-1}$, we introduce the following concepts.

\begin{defn}
	We call a sequence of variables $w_{i_k}^{(k)},w_{i_{k+1}}^{(k+1)},\cdots,w_{i_{k'}}^{(k')}$ \emph{a Cauchy chain} with respect to the variables $W^{(\ell)}$'s and sets $I^{(\ell)},J^{(\ell)}$'s, if
	\begin{equation*}
	(w_{i_k}^{(k)}-w_{i_{k+1}}^{(k+1)})(w_{i_{k+1}}^{(k+1)}-w_{i_{k+2}}^{(k+2)})\cdots (w_{i_{k'-1}}^{(k'-1)}-w_{i_{k'}}^{(k')})
	\end{equation*}
	appears as a factor in the denominator of $\prod_{\ell=1}^{m-1}\Ch\left(W^{(\ell)}_{I^{(\ell)}}; W^{(\ell+1)}_{J^{(\ell+1)}}\right)$. In other words,
	\begin{equation*}
	(i_k,i_{k+1})\in I^{(k)}\times J^{(k+1)}, (i_{k+1},i_{k+2})\in I^{(k+1)}\times J^{(k+2)},\cdots, (i_{k'-1},i_{k'})\in I^{(k'-1)}\times J^{(k')}.
	\end{equation*}
	We also call any single variable $w_{i_k}^{(k)}$ a Cauchy chain.
\end{defn}

We remark that one important property of Cauchy chain is that it could accumulate singularities of $A(W^{(1)},\cdots,W^{(m)},z_0,\cdots,z_{m-1})$ at $w_{i_\ell}^{(\ell)}=0$ if $w_{i_\ell}^{(\ell)}$ is on the chain by evaluating the residues from the Cauchy factors successively.

\begin{defn}
	\label{def:dominate}
	We call $q(w)$ \emph{dominates} $H(W^{(1)},\cdots,W^{(m)};z_1,\cdots,z_m)$ at $w=0$ provided that for any Cauchy chain $w_{i_k}^{(k)},w_{i_{k+1}}^{(k+1)},\cdots,w_{i_{k'}}^{(k')}$, 
	\begin{equation*}
	q(w)\cdot \left. A(W^{(1)},\cdots,W^{(m)};z_0,\cdots,z_{m-1}) \right|_{w_{i_k}^{(k)}=w_{i_{k+1}}^{(k+1)}=\cdots =w_{i_{k'}}^{(k')}=w}
	\end{equation*}
	is analytic at $w=0$, for any fixed other $w_{i_\ell}^{(\ell)}$ variables in $\Omega_0$, and fixed $(z_0,\cdots,z_{m-1})\in \disk(r_\mm)\times\disk^{m-1}$.
\end{defn}

We also remark that if $q(w)$ dominates $H$, then $q(w)\left. A(W^{(1)},\cdots,W^{(m)};z_0,\cdots,z_{m-1})\right|_{w_{i_\ell}^{(\ell)}=w}$ is analytic at $w=0$ since a single variable forms a Cauchy chain. In other words, the singularities of $A$ at each $w_{i_\ell}^{(\ell)}=0$ are dominated by the order of $q(w)$ at $w=0$. Furthermore, the total singularities of $A$ at $w_{i_k}^{(k)}=0$, $w_{i_{k+1}}^{(k+1)}=0$, $\cdots$, $w_{i_{k'}}^{(k')}=0$ along any Cauchy chain $w_{i_k}^{(k)},w_{i_{k+1}}^{(k+1)},\cdots,w_{i_{k'}}^{(k')}$  are dominated by the order of $q(w)$ at $w=0$.

\vspace{0.2cm}

Now we are ready to state the main proposition.
\begin{prop}
	\label{prop:Cauchy_sum_over_roots}
	Suppose $A$ is analytic for each $w_{i_\ell}^{(\ell)}\in \Omega_0$ and $(z_0,\cdots,z_{m-1})\in\disk(r_\mm)\times\disk^{m-1}$. Suppose $q(w)$ is analytic for $w\in\Omega$ with the nested level curve assumption described before. If $q(w)$ dominates $H$ at $w=0$ as defined above, then $\GG(z_0,\cdots,z_{m-1})$ can be analytically defined for for $(z_0,\cdots,z_{m-1})\in\disk(r_\mm)\times\disk^{m-1}$. Moreover, $\GG(0,z_1,\cdots,z_{m-1})$ is independent of $q(w)$, and equals to
	\begin{equation}
	\label{eq:def_G0}
	\prod_{\ell=2}^{m}  \prod_{{i_\ell}=1}^{n_\ell} \left[\frac{1}{1-z_{\ell-1}}\int_{\Sigma_{\ell}^\inn} \ddbarr{w_{i_\ell}^{(\ell)}}
	-\frac{z_{\ell-1}}{1-z_{\ell-1}}\int_{\Sigma_{\ell}^\out} \ddbarr{w_{i_\ell}^{(\ell)}} \right]
	\cdot \prod_{{i_1}=1}^{n_1} \int_{\Sigma_1} \ddbarr{w_{i_1}^{(1)}}H(W^{(1)},\cdots,W^{(m)};0,z_1,\cdots,z_{m-1}),
	\end{equation}
	where $\Sigma_{m}^{\out},,\cdots,\Sigma_{2}^{\out},\Sigma_1,\Sigma_{2}^\inn,\cdots,\Sigma_{m}^\inn$, from outside to inside, are arbitrary $2m-1$ nested simple closed contours in $\Omega$ each of which encloses $w=0$.	
\end{prop}
Although in this proposition we only considered the case when $q(w)$ has a root at $w=0$ and $q(w)$ dominates $H$ at $w=0$, it is easy to see (by a change of variables $w\to w+a$) that the same proposition holds if we consider the case when $q(w)$ has a root at $w=a$ and $q(w)$ dominates $H$ at $w=a$.

The proof of Proposition~\ref{prop:Cauchy_sum_over_roots} is given in Section~\ref{sec:proof_Cauchy_sum}. We point out that the most challenging part of this proposition is to find the explicit expression for $\GG(0,z_1,\cdots,z_{m-1})$. We actually construct the formula~\eqref{eq:def_G0} and prove the proposition by induction. See Section~\ref{sec:proof_Cauchy_sum} for the details. Similarly to Proposition~\ref{prop:invariance_spaces}, we are able to change the nesting order of contours of integration (and the $z_\ell$ weights accordingly) in~\eqref{eq:def_G0} and obtain different formulas of $\GG(0,z_{1},\cdots,z_{m-1})$. This fact could be proved in a similar way as in the proof of Proposition~\ref{prop:invariance_spaces}, or modifying the proof of Proposition~\ref{prop:Cauchy_sum_over_roots} in Section~\ref{sec:proof_Cauchy_sum} accordingly for the different formula of $\GG(0,z_{1},\cdots,z_{m-1})$.

\vspace{0.4cm}

Proposition~\ref{prop:Cauchy_sum_over_roots} only includes the case of one region $\Omega$ and one set of nested roots (or contours) around (or enclosing, respectively) the unique root of $q(w)$ within $\Omega$. There is no difficulty to extend it to more regions and more sets of nested contours, where each set of contours enclose a different root of $q(w)$. Especially for the purpose of proving Theorem~\ref{thm:periodic_TASEP_large_period}, we need a version of two regions and two sets of nested roots enclosing two different roots $-1$ and $0$ of the Bethe polynomial $q(z)$ respectively. We state the result below for this use and prove it by using Proposition~\ref{prop:Cauchy_sum_over_roots}. 

Let $\Omega_\LL$ and $\Omega_\RR$ be two disjoint regions including $-1$ and $0$ respectively. Let $n_{\ell,\LL}$ and $n_{\ell,\RR}$, $1\le \ell\le m$, be $2m$ non-negative integers. $U^{(\ell)} = (u_{1}^{(\ell)},\cdots,u_{n_{\ell,\LL}}^{(\ell)})$ and $V^{(\ell)} =(v_{1}^{(\ell)},\cdots,v_{n_{\ell,\RR}}^{(\ell)})$ are $2m$ vectors. We use $U$, $u$ and $V$, $v$ to denote the vectors, variables associated with $\LL$ and $\RR$ respectively to avoid too many scripts. This is also consistent with the notations in the series expansions of $\mathcal{D}_Y$ in Theorem~\ref{thm:periodic_TASEP_large_period}. Similar to~\eqref{eq:IJ_ell}, we introduce $I_{\LL}^{(\ell)}, J_{\LL}^{(\ell)}$ and $I_{\RR}^{(\ell)}, J_{\RR}^{(\ell+1)}$ for each $1\le \ell \le m$. Then the analog of~\eqref{eq:H} is
\begin{equation*}
\begin{split}
&H(U^{(1)},\cdots,U^{(m)};V^{(1)},\cdots,V^{(m)};z_0,\cdots,z_{m-1})\\
&:=\left[ \prod_{\ell=1}^{m-1} \Ch\left(U^{(\ell)}_{I^{(\ell)}_\LL}; U^{(\ell+1)}_{J^{(\ell+1)}_\LL}\right) \Ch\left(V^{(\ell)}_{I^{(\ell)}_\RR}; V^{(\ell+1)}_{J^{(\ell+1)}_\RR}\right) \right] \cdot A(U^{(1)},\cdots,U^{(m)};V^{(1)},\cdots,V^{(m)};z_0,\cdots,z_{m-1}),
\end{split}
\end{equation*}
where $A$ is a function analytic for all $u_{i_\ell}^{(\ell)}$ in $\Omega_\LL\setminus\{-1\}$, all $v_{i'_{\ell'}}^{(\ell')} \in =\Omega_\RR\setminus \{0\}$, and all $(z_0,\cdots,z_{m-1})\in\disk(r_\mm)\times\disk^{m-1}$.

Let $q(w)$ be a function defined on $\Omega_\LL\cup\Omega_\RR$ such that its ``level curves'' in $\Omega_\LL$ and $\Omega_\RR$ are nested simple closed contours enclosing $-1$ and $0$ respectively. Note that we do not require $q(w)$ is well defined elsewhere. Let $\roots_{\hat z,\LL}:=\{u\in\Omega_\LL: q(u)=\hat z\}$ and $\roots_{\hat z,\RR}:=\{v\in\Omega_\RR: q(v)=\hat z\}$. We define, for $(z_0,\cdots,z_{m-1})\in\disk_0(r_\mm)\times \disk_0^{m-1}$,
\begin{equation}
\label{eq:GG1}
\begin{split}
&\GG(z_0,\cdots,z_{m-1})\\
&=\sum_{\substack{U^{(1)} \in \roots_{\hat z_1,\LL}^{n_{1,\LL}}\\ \vdots \\ U^{(m)} \in \roots_{\hat z_m,\LL}^{n_{m,\LL}}}} \sum_{\substack{V^{(1)} \in \roots_{\hat z_1,\RR}^{n_{1,\RR}}\\ \vdots \\ V^{(m)} \in \roots_{\hat z_m,\RR}^{n_{m,\RR}}}} \left[\prod_{\ell=1}^mJ(U^{(\ell)})J(V^{(\ell)})\right]\cdot H(U^{(1)},\cdots,U^{(m)};V^{(1)},\cdots,V^{(m)};z_0,\cdots,z_{m-1}),
\end{split}
\end{equation}
where $J$ is defined the same way as in~\eqref{eq:J}, and $\hat z_\ell$ as in~\eqref{eq:hatz}. 

We could similarly define the terminologies of ``Cauchy chain'' and ``dominating''. More explicitly, a Cauchy chain is either a sequence of variables $u_{i_k}^{(k)},u_{i_{k+1}}^{(k+1)},\cdots,u_{i_{k'}}^{(k')}$ such that $(u_{i_k}^{(k)}-u_{i_{k+1}}^{(k+1)})\cdots(u_{i_{k'-1}}^{(k'-1)}-u_{i_{k'}}^{(k')})$ appears in the denominator of $\prod_{\ell=1}^{m-1} \Ch\left(U^{(\ell)}_{I^{(\ell)}_\LL}; U^{(\ell+1)}_{J^{(\ell+1)}_\LL}\right)$, or  a sequence of variables $v_{i_k}^{(k)},v_{i_{k+1}}^{(k+1)},\cdots,v_{i_{k'}}^{(k')}$ such that $(v_{i_k}^{(k)}-v_{i_{k+1}}^{(k+1)})\cdots(v_{i_{k'-1}}^{(k'-1)}-v_{i_{k'}}^{(k')})$ appears in the denominator of $\prod_{\ell=1}^{m-1} \Ch\left(V^{(\ell)}_{I^{(\ell)}_\RR}; V^{(\ell+1)}_{J^{(\ell+1)}_\RR}\right)$. We still allow that a Cauchy chain could be a single variable. We say $q$ dominates $H$ at $w=-1$, if $q(w)\cdot A$ is analytic at $w=-1$ when we take the variables on any $u_{i_\ell}^{(\ell)}$-Cauchy chain to be $w$ but all other variables fixed. Similarly, $q$ dominates $H$ at $w=0$ if $q(w)\cdot A$ is analytic at $w=0$ when we take the variables on any $v_{i_\ell}^{(\ell)}$-Cauchy chain to be $w$ but all other variables fixed.

With these setting, the two-region version of Proposition~\ref{prop:Cauchy_sum_over_roots} is as follows.
\begin{prop}
	\label{prop:Cauchy_sum_over_roots_ext}
	Suppose $A$ is analytic for each $u_{i_\ell}^{(\ell)}$ in $\Omega_{\LL}\setminus\{-1\}$, each $v_{i'_{\ell'}}^{\ell'} \in \Omega_{\RR}\setminus\{0\}$, and each $(z_0,\cdots,z_{m-1})\in\disk(r_\mm)\times \disk^{m-1}$. Suppose $q(w)$ is analytic for $w\in\Omega_\LL\cup \Omega_\RR$ with the nested level curve assumption described above. If $q(w)$ dominates $H$ at $w=-1$ and $w=0$. Then $\GG(z_0,\cdots,z_{m-1})$ can be analytically extended to $\disk(r_\mm)\times \disk^{m-1}$. Moreover, $\GG(0,z_1,\cdots,z_{m-1})$ is independent of $q(w)$, and equals to
	\begin{equation}
	\label{eq:def_G0ext}
	\begin{split}
	&\prod_{\ell=2}^{m}  \prod_{{i_\ell}=1}^{n_{\ell,\LL}} \left[\frac{1}{1-z_{\ell-1}}\int_{\Sigma_{\ell,\LL}^\inn} \ddbarr{u_{i_\ell}^{(\ell)}}
	-\frac{z_{\ell-1}}{1-z_{\ell-1}}\int_{\Sigma_{\ell,\LL}^\out} \ddbarr{u_{i_\ell}^{(\ell)}} \right]
	\cdot \prod_{{i_1}=1}^{n_{1,\LL}} \int_{\Sigma_{1,\LL}} \ddbarr{u_{i_1}^{(1)}}\\
	&\prod_{\ell=2}^{m}  \prod_{{i_\ell}=1}^{n_{\ell,\RR}} \left[\frac{1}{1-z_{\ell-1}}\int_{\Sigma_{\ell,\RR}^\inn} \ddbarr{v_{i_\ell}^{(\ell)}}
	-\frac{z_{\ell-1}}{1-z_{\ell-1}}\int_{\Sigma_{\ell,\RR}^\out} \ddbarr{v_{i_\ell}^{(\ell)}} \right]
	\cdot \prod_{{i_1}=1}^{n_{1,\RR}} \int_{\Sigma_{1,\RR}} \ddbarr{v_{i_1}^{(1)}}\\
	&
	H(U^{(1)},\cdots,U^{(m)};V^{(1)},\cdots,V^{(m)};0,z_1,\cdots,z_{m-1}),
	\end{split}
	\end{equation}
	where $\Sigma_{m,\LL}^{\out},\cdots,\Sigma_{2,\LL}^{\out},\Sigma_{1,\LL},\Sigma_{2,\LL}^{\inn},\cdots,\Sigma_{m,\LL}^{\inn}$, from outside to inside, are arbitrary $2m-1$ nested simple closed contours in $\Omega_\LL$ each of which encloses $u=-1$, and $\Sigma_{m,\RR}^{\out},\cdots,\Sigma_{2,\RR}^{\out},\Sigma_{1,\RR},\Sigma_{2,\RR}^{\inn},\cdots,\Sigma_{m,\RR}^{\inn}$, from outside to inside, are arbitrary $2m-1$ nested simple closed contours in $\Omega_\RR$ each of which encloses $v=0$.
\end{prop}
\begin{proof}[Proof of Proposition~\ref{prop:Cauchy_sum_over_roots_ext}]
	It follows by applying Proposition~\ref{prop:Cauchy_sum_over_roots} twice. First for any fixed $U^{(\ell)}$'s, we consider  
	\begin{equation*}
	\tilde H(U^{(1)},\cdots,U^{(m)};z_0,\cdots,z_{m-1}):=
	\sum_{\substack{V^{(1)} \in \roots_{\hat z_1,\RR}^{n_{1,\RR}}\\ \vdots \\ V^{(m)} \in \roots_{\hat z_m,\RR}^{n_{m,\RR}}}} \left[\prod_{\ell=1}^m J(V^{(\ell)})\right]\cdot H(U^{(1)},\cdots,U^{(m)};V^{(1)},\cdots,V^{(m)};z_0,\cdots,z_{m-1}).
	\end{equation*}
	This function is analytic for $(z_0,\cdots,z_{m-1})\in \disk(r_\mm)\times \disk^{m-1}$ by Proposition~\ref{prop:Cauchy_sum_over_roots}. It is also analytic for $u_{i_\ell}^{(\ell)}\in\Omega_{\LL}\setminus\{-1\}$. Thus we could apply Proposition~\ref{prop:Cauchy_sum_over_roots} again for
	\begin{equation*}
	\GG(z_0,\cdots,z_{m-1})=\sum_{\substack{U^{(1)} \in \roots_{\hat z_1,\LL}^{n_{1,\LL}}\\ \vdots \\ U^{(m)} \in \roots_{\hat z_m,\LL}^{n_{m,\LL}}}}\left[\prod_{\ell=1}^mJ(U^{(\ell)})\right]\cdot \tilde H(U^{(1)},\cdots,U^{(m)};z_0,\cdots,z_{m-1}).
	\end{equation*}
	This proves the analyticity of $\GG(z_0,\cdots,z_{m-1})$ in $\disk(r_\mm)\times \disk^{m-1}$. The formula for $\GG(0,z_1,\cdots,z_{m-1})$ follows in a similar way.
\end{proof}

\section{Proof of Theorem~\ref{thm:periodic_TASEP_large_period}}
\label{sec:proof_Theorem}

In this section, we prove Theorem~\ref{thm:periodic_TASEP_large_period}. We will first reduce the proof of the theorem to two lemmas, Lemma~\ref{lm:lemma1} and Lemma~\ref{lm:lemma2} below. Then we prove these two lemmas in Section~\ref{sec:proof_lemma1} and Section~\ref{sec:proof_lemma2} respectively.

We first assume 
\begin{equation}
\label{eq:assumption_a}
a_\ell +k_\ell \ge y_N+N,\qquad \ell=1,\cdots,m.
\end{equation}
We claim that it is sufficient to prove Theorem~\ref{thm:periodic_TASEP_large_period} with the above assumption. In fact, if there exists some $i$ such that $a_i +k_i=\min\{a_\ell+k_\ell:1\le\ell\le m\}< y_N+N$, then $a_i +k_i< y_{k_i}+k_i=x_{k_i}(0)+k_i$, and
\begin{equation*}
\prob_Y^{\period}\left(\bigcap_{\ell=1}^m\left\{ x_{k_\ell}^{\period}(t_\ell) \ge a_\ell\right\}\right) = \prob_Y^{\period}\left(\bigcap_{\substack{1\le \ell\le m\\ \ell\ne i}}\left\{ x_{k_\ell}^{\period}(t_\ell) \ge a_\ell\right\}\right)
\end{equation*}
since $\{x_{k_i}^{(L)}(t_i)\ge a_i\}$ is an event with probability $1$. On the other hand, by Proposition~\ref{cor:consistence} we have
\begin{equation*}
\begin{split}
&\oint\cdots\oint \left[\prod_{\ell=1}^{m-1}\frac{1}{1-z_\ell}\right]\mathcal{D}_Y(z_1,\cdots,z_{m-1}) \ddbar{z_1}\cdots\ddbar{z_{m-1}}\\
&=\oint\cdots\oint \left[\prod_{\substack{1\le \ell\le m-1\\ \ell\ne i}}\frac{1}{1-z_\ell}\right]\mathcal{D}_Y(z_1,\cdots,z_{i-1},z_{i+1},\cdots,z_{m-1}) \ddbar{z_1}\cdots\ddbar{z_{i-1}}\ddbar{z_{i+1}}\cdots\ddbar{z_{m-1}}.
\end{split}
\end{equation*}
Thus it is sufficient to prove the statement with the index $i$ removed. By repeating this procedure and removing all such indices $i$, we only need to prove the statement with all indices $\ell$ satisfying~\eqref{eq:assumption_a}.

From now on throughout this section, we assume~\eqref{eq:assumption_a} holds.

 It has been shown in \cite{Baik-Liu21} (and \cite{Baik-Liu19} for the case of the step initial condition) that the multi-point distribution of periodic TASEP has an explicit formula in terms of multiple contour integrals
\begin{equation*}
\prob_Y^{\period}\left(\bigcap_{\ell=1}^m\left\{ x_{k_\ell}^{\period}(t_\ell) \ge a_\ell\right\}\right)= \oint\cdots\oint \mathscr{C}_Y(\hat z_1,\cdots,\hat z_m) \mathscr{D}_Y(\hat z_1,\cdots,\hat z_m) \ddbar{\hat z_1}\cdots\ddbar{\hat z_{m}},
\end{equation*}
where the contours are nested circles centered the origin with decreasing radii $0<|\hat z_m|<\cdots<|\hat z_1|<r_\mm$ for some constant $r_\mm>0$ to be determined later. The explicit formula of $\mathscr{C}_Y$ and $\mathscr{D}_Y$ will be given in Section~\ref{sec:proof_lemma1} and Section~\ref{sec:proof_lemma2} respectively. By changing the variables
\begin{equation}
\label{eq:hat_z}
\hat z_\ell = \prod_{j=0}^{\ell-1} z_j,\quad \ell=1,\cdots,m,
\end{equation}
where $z_0,z_1,\cdots,z_{m-1}$ are new variables satisfying $|z_\ell|< 1$ for $1\le \ell\le m-1$ and $0<|z_0|<r_\mm$, we write
\begin{equation}
\label{eq:aux_002}
\prob_Y^{\period}\left(\bigcap_{\ell=1}^m\left\{ x_{k_\ell}^{\period}(t_\ell) \ge a_\ell\right\}\right)= \oint\cdots\oint \tilde{\mathscr{C}}_Y\left(z_0,\cdots,z_{m-1}\right) \tilde{\mathscr{D}}_Y\left(z_0,\cdots,z_{m-1}\right) \ddbar{z_0}\cdots\ddbar{ z_{m-1}}.
\end{equation}
Here $\tilde{\mathscr{C}}_Y\left(z_0,\cdots,z_{m-1}\right):=\mathscr{C}_Y(\hat z_1,\cdots,\hat z_m)$ and $\tilde{\mathscr{D}}_Y\left(z_0,\cdots,z_{m-1}\right):=\mathscr{D}_Y(\hat z_1,\cdots,\hat z_m) $ with $\hat z_\ell$ defined by~\eqref{eq:hat_z}. The contours of integration are circles centered at the origin with radii satisfying $0<|z_0|<r_\mm$ and $|z_\ell|<1$ for $1\le\ell\le m-1$. 

It turns out that the $z_0$ integral can be evaluated explicitly in~\eqref{eq:aux_002}. The key facts are that both functions $ \tilde{\mathscr{C}}_Y$ and $ \tilde{\mathscr{D}}_Y$ can be analytically extended to $z_0=0$ and that their values at $z_0=0$ can be explicitly evaluated. These are given in the following two lemmas. We recall that the notations $\disk(r)$ in~\eqref{eq:diskr} and  $\disk_0(r)=\disk(r)\setminus\{0\}$ in~\eqref{eq:disk0r}. When $r=1$, we simply write $\disk$ and $\disk_0$ for $\disk(1)$ and $\disk_0(1)$ respectively.
\begin{lm}
	\label{lm:lemma1}
	The function $\tilde{\mathscr{C}}_Y\left(z_0,\cdots,z_{m-1}\right)$ is analytic for $z_0\in\disk(r_\mm)$ and $z_\ell\in\disk$, $1\le \ell\le m-1$. Moreover, 
	\begin{equation*}
	\tilde{\mathscr{C}}_Y\left(0,z_1,\cdots,z_{m-1}\right)=\prod_{\ell=1}^{m-1} \frac{1}{1-z_\ell}
	\end{equation*}
	for any fixed $z_1,\cdots,z_{m-1}\in\disk$.
\end{lm}

\begin{lm}
	\label{lm:lemma2}
	 Assume~\eqref{eq:assumption_a} and $L \ge \max\{a_1+k_1,\cdots,a_m+k_m\} - y_N$.	Then the function $\tilde{\mathscr{D}}_Y\left(z_0,\cdots,z_{m-1}\right)$ is analytic for $z_0\in\disk(r_\mm)$ and $z_\ell\in\disk_0$, $1\le \ell\le m-1$. Moreover, 
	\begin{equation*}
	\tilde{\mathscr{D}}_Y\left(0,z_1,\cdots,z_{m-1}\right)=\mathcal{D}_Y(z_1,\cdots,z_{m-1})
	\end{equation*}
	for any fixed $z_1,\cdots,z_{m-1}\in\disk_0$. Here the function $\mathcal{D}_Y(z_1,\cdots,z_{m-1})$ is defined in terms of a Fredholm determinant in Definition~\ref{def:operators_K1Y}, or equivalently in terms of series expansion in Definition~\ref{def:D_Y_series}.
\end{lm}

The proofs of these two lemmas are given in Section~\ref{sec:proof_lemma1} and Section~\ref{sec:proof_lemma2} respectively.

By applying these two lemmas above and taking the residue at $z_0=0$ in~\eqref{eq:aux_002}, we write~\eqref{eq:aux_002} as
\begin{equation*}
\begin{split}
\prob_Y^{\period}\left(\bigcap_{\ell=1}^m\left\{ x_{k_\ell}^{\period}(t_\ell) \ge a_\ell\right\}\right)&= \oint\cdots\oint \tilde{\mathscr{C}}_Y\left(0,z_1\cdots,z_{m-1}\right) \tilde{\mathscr{D}}_Y\left(0,z_1,\cdots,z_{m-1}\right) \ddbar{z_1}\cdots\ddbar{ z_{m-1}}\\
&= \oint\cdots\oint \left[\prod_{\ell=1}^{m-1}  \frac{1}{1-z_\ell}\right] \mathcal{D}_Y(z_1,\cdots,z_{m-1}) \ddbar{z_1}\cdots\ddbar{ z_{m-1}}.
\end{split}
\end{equation*}
This proves Theorem~\ref{thm:periodic_TASEP_large_period}.

\vspace{0.3cm}

In Sections~\ref{sec:preliminaries},~\ref{sec:proof_lemma1} and~\ref{sec:proof_lemma2} below, we will introduce the functions  $ \mathscr{C}_Y(\hat z_1,\cdots,\hat z_m)$, $ \mathscr{D}_Y(\hat z_1,\cdots,\hat z_m)$ and prove Lemma~\ref{lm:lemma1} and Lemma~\ref{lm:lemma2}. We would like to emphasize that although most of the functions are already defined in \cite{Baik-Liu21}, there are some modifications due to the different settings of two papers. One could match our definitions in this paper with their analogs in \cite{Baik-Liu21} by doing the following changes in their paper: $k_\ell\to N+1-k_\ell$, $y_i\to y_i+1$ and $a_i\to a_i+1$. The first change is due to the different ordering of the particles, the other changes are related to a shift of all particles by $1$ in order to make our formula as simple as possible.

\subsection{Preliminaries on Bethe roots and some functions involving the initial condition $Y$}
\label{sec:preliminaries}

Before we define the functions $ \mathscr{C}_Y(\hat z_1,\cdots,\hat z_m)$, $ \mathscr{D}_Y(\hat z_1,\cdots,\hat z_m)$ and prove Lemma~\ref{lm:lemma1} and Lemma~\ref{lm:lemma2}, we introduce the concepts of Bethe roots, and some functions involving the initial condition $Y$.

\subsubsection{Bethe roots}

Let
\begin{equation*}
q(w):= w^N (w+1)^{L-N}.
\end{equation*}
The \emph{Bethe equation} associated to the periodic TASEP of period $L$ and particle numbers $N$ is defined to be
\begin{equation}
\label{eq:qz}
q_z(w) = q(w) - z = w^{N}(w+1)^{L-N} -z
\end{equation}
for any $z\in\complexC$.

We remark that this is slightly different from the function $q_z(w)$ in \cite{Baik-Liu16,Baik-Liu19,Baik-Liu21} which is defined by $w^N(w+1)^{L-N}-z^L$. The main reason the authors used $z^L$ instead of $z$ in their papers is for the purpose of asymptotic analysis in the so-called relaxation time scale: The roots of $w^N(w+1)^{L-N}-z^L$ are on level curves which only depend on two parameters, the ratio $N/L$ and the magnitude of $z$, and these two parameters are chosen to be independent of $L$ in the asymptotic analysis. However, in this paper we only consider the finite time case for periodic TASEP and we expect that the parameter $L$ will disappear in the probability distributions as we claimed in Theorem~\ref{thm:comparison} and Theorem~\ref{thm:periodic_TASEP_large_period}. Thus it is more natural to use~\eqref{eq:qz}.

We also introduce the set of \emph{Bethe roots}
\begin{equation*}
\roots_{z}:=\{w\in\complexC: q_z(w)=0\},\quad\text{ or equivalently, }\quad\roots_{z}:=\{ w\in\complexC: q(w)=z\},
\end{equation*}
and the \emph{level curves} of $q(w)$
\begin{equation*}
\Gamma_r:=\{ w\in\complexC: |q(w)| = r\}.
\end{equation*}
Note that the definitions above imply that all the roots in $\roots_z$ are on the level curve $\Gamma_{|z|}$.
\begin{center}
\begin{figure}
\centering
\includegraphics[width=8cm, height=6cm]{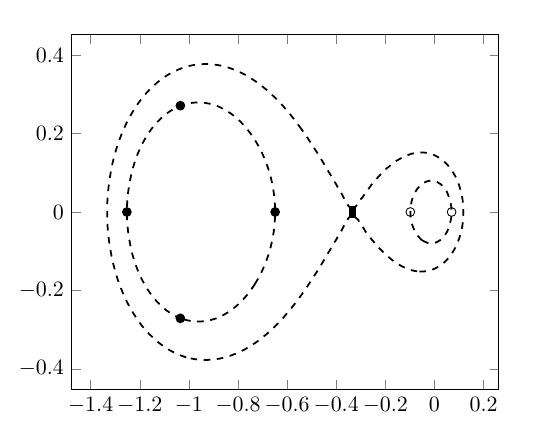}
%\begin{tikzpicture}
%	\begin{axis}
%		\addplot +[no markers,
%		raw gnuplot,
%		thick,
%		dashed,
%		black,
%		smooth,
%		empty line = jump 
%		] gnuplot{
%			set contour base;
%			set cntrparam levels discrete 0.0000001;
%			unset surface;
%			set view map;
%			set isosamples 1000;
%			set samples 1000;
%			splot (x^2 +y^2)*((x+1)^2+y^2)^2 - 0.02194787;
%			splot (x^2 +y^2)*((x+1)^2+y^2)^2 -0.0064;
%		};
%		\draw [fill=black] (-0.32,-0.01333) rectangle (-0.346666,0.01333);
%		\draw (0.06989,0) circle[radius=2pt];
%		\draw (-1.03494,-0.27121) circle[radius=2pt];
%		\fill (-1.03494,-0.27121) circle[radius=2pt];
%		\draw (-1.03494,0.27121) circle[radius=2pt];
%		\fill (-1.03494,0.27121) circle[radius=2pt];
%		\draw (-1.2527,0) circle[radius=2pt];
%		\fill (-1.2527,0) circle[radius=2pt];
%		\draw (-0.64887,0) circle[radius=2pt];
%		\fill (-0.64887,0) circle[radius=2pt];
%		\draw (-0.098419,0) circle[radius=2pt];
%	\end{axis}
%\end{tikzpicture}
\caption{An illustration of the Bethe roots and level curves for $q(w)=w^2(w+1)^4$: $\Gamma_{r=0.0064}$ consists of the two inner dashed curves. On these two curves, the four black dots represent the four points in $\mathcal{R}_{z=0.0064,\LL}$, and the two white dots represent the two points in $\mathcal{R}_{z=0.0064,\RR}$. $\Gamma_{r_c=2^4/3^6}$ consists of the two outer dashed curves separated by $w_c=-1/3$ (the black square in the middle), within which are the two domains $\Omega_\LL$ and $\Omega_\RR$ respectively.}
\label{figure:Bethe_roots}
\end{figure}
\end{center}

It was known that (see the related discussions in \cite{Baik-Liu16,Baik-Liu19,Baik-Liu21} for examples) the level curves of $q(w)$ are nested contours: $\Gamma_{r}$ encloses $\Gamma_{r'}$ if $r>r'$. Moreover, when $r>r_c$ for some $r_c$ defined by
\begin{equation*}
r_c  := \frac{N^N(L-N)^{L-N}}{L^L},
\end{equation*}
$\Gamma_r$ is a simple closed contour enclosing both $-1$ and $0$. When $r=r_c$, $\Gamma_r$ is a self-intersect contour with the intersection point
\begin{equation*}
w_c:=-N/L.
\end{equation*}
When $0<r<r_c$, $\Gamma_r$ splits into two disjoint simple closed contours, one of which encloses $-1$ but not $0$ and the other encloses $0$ but not $-1$. We denote these two contours $\Gamma_{r,\LL}$ and $\Gamma_{r,\RR}$ respectively. Moreover, $\Gamma_{r,\LL}$ and $\Gamma_{r,\RR}$ stay on two different sides of $w_c$.

For $0<r<r_c$, we denote $\Omega_{r,\LL}$ ($\Omega_{r,\RR}$, respectively) the region enclosed by the contour $\Gamma_{r,\LL}$ ($\Gamma_{r,\RR}$, respectively). Then we define
\begin{equation*}
\Omega_\LL=\cup_{0<r<r_c} \Omega_{r,\LL},\quad \text{and} \quad \Omega_\RR=\cup_{0<r<r_c} \Omega_{r,\RR}.
\end{equation*}
These are two non-intersecting open regions which are on the two sides of $w_c$ respectively. Moreover, $-1\in\Omega_\LL$ and $0\in\Omega_\RR$.

Now we return to the discussion of Bethe roots. When $0<|z|<r_c$, we denote
\begin{equation*}
\roots_{z,\LL} = \roots_z\cap \Omega_\LL,\quad \text{and}\quad \roots_{z,\RR} = \roots_z\cap \Omega_\RR.
\end{equation*}
It is easy to see that $\roots_{z,\LL}$ and $\roots_{z,\RR}$ consists of $L-N$ and $N$ elements respectively. 
These elements converge to $-1$ and $0$ respectively when $z\to 0$.

See Figure~\ref{figure:Bethe_roots} for an illustration of the Bethe roots, the level curves and the domains $\Omega_\LL$ and $\Omega_\RR$. 

We define
\begin{equation*}
q_{z,\LL}(w)=\prod_{u\in \roots_{z,\LL}}(w-u),\qquad\text{and}\quad  q_{z,\RR}(w) =\prod_{v\in\roots_{z,\RR}}(w-v).
\end{equation*}
By the discussions on $\roots_{z,\LL}$ and $\roots_{z,\RR}$ above, we know that $q_{z,\LL}(w)\to (w+1)^{L-N}$ and $q_{z,\RR}(w)\to w^N$ as $z\to 0$.  Hence we introduce the ``normalized'' version of $q_{z,\LL}$ and $q_{z,\RR}$ below
\begin{equation*}
\mathfrak{q}_{z,\LL}(w):=\frac{q_{z,\LL}(w)}{(w+1)^{L-N}},\quad\text{and}\quad 
\mathfrak{q}_{z,\RR}(w):=\frac{q_{z,\RR}(w)}{w^N}.
\end{equation*}
We further write
\begin{equation}
\label{eq:def_frakh}
\mathfrak{h}(w;z)=\begin{dcases}
\mathfrak{q}_{z,\LL}(w),& w\in\Omega_\RR\\
\mathfrak{q}_{z,\RR}(w),& w\in\Omega_\LL.
\end{dcases}
\end{equation}

It is easy to see that $\mathfrak{h}(w;z)$ is analytic for $(w,z)$ in both $\Omega_\LL\times \disk(r_c)$ and $\Omega_\RR\times \disk(r_c)$. Moreover, it is always nonzero in the above domain. Finally, $\mathfrak{h}(w;0)=1$  for all $w\in\Omega_\RR\cup\Omega_\LL$.

\subsubsection{Functions involving the initial condition $Y$}

We introduce some functions involving the initial condition $Y$. The two functions $\mathcal{E}_Y (z)$ and $\iich_{Y}(v,u;z)$ were introduced in \cite{Baik-Liu21}. We slightly modified their formulas below due to the relabeling of particles. 
One could replace $y_i$ by $y_{N+1-i} -1$ in the formulas below to recover the versions in \cite{Baik-Liu21}.

\begin{defn}
	\label{def:energy}
	Suppose $0<|z|<r_c$. Let
	\begin{equation}
	\label{eq:def_E}
	\mathcal{E}_Y (z) := \prod_{v\in\roots_{z,\RR}}(v+1)^{y_N+N}\cdot \mathcal{G}_{\boldsymbol{\lambda}(Y)} (\roots_{z,\RR}),
	\end{equation}
	where $\boldsymbol{\lambda}(Y)=(\lambda_1,\cdots,\lambda_N)$ with $\lambda_i = (y_i+i) - (y_N+N)$, and the function $\mathcal{G}_{\boldsymbol{\lambda}}$ is defined in~\eqref{eq:def_G_ftn}. It can also be expressed as
	\begin{equation*}
	\mathcal{E}_Y (z) = \frac{\det\left[v_i^{-j} (v_i+1)^{y_j+j}\right]_{i,j=1}^N}{\det\left[v_i^{-j}\right]_{i,j=1}^N},
	\end{equation*}	
	where $v_1,\cdots,v_N$ are all the elements of $\roots_{z,\RR}$. We also define $\mathcal{E}_Y(0)=1$.
\end{defn}

Since all the elements in $\roots_{z,\RR}$ go to $0$ as $z\to 0$, it is easy to see (for example, using the equations~\eqref{eq:def_E} and~\eqref{eq:expan_G1} below) that $\mathcal{E}_Y(z)$ is analytic for $z$ within $\{z: |z|<r_c\}$.

Since $\mathcal{E}_Y(0)=1$, there exists some positive constant $r_\mm$, such that $r_\mm<r_c$ and
\begin{equation}
\label{eq:r_max}
\mathcal{E}_Y(z) \ne 0,\qquad \text{for all $z$ satisfying }|z|<r_\mm.
\end{equation}
Note the this also implies $\mathcal{G}_{\boldsymbol{\lambda}(Y)} (\roots_{z,\RR}) \ne 0$ for all $z\in\disk(r_\mm)=\{z\in\complexC:|z|<r_\mm\}$. 

\vspace{0.3cm}

\begin{defn}
	\label{def:ich_periodic}
	Suppose $0<|z|<r_\mm$. For any $u\in\Omega_\LL\setminus\{-1\}$ and $v\in\roots_{z,\RR}$, we define
	\begin{equation}
	\label{eq:def_ich_periodic}
	\iich_{Y}(v,u;z) =\left(\frac{u+1}{v+1}\right)^{y_N+N} \frac{\mathcal{G}_{\boldsymbol{\lambda}(Y)} ((\roots_{z,\RR}\setminus \{v\}) \cup \{u\})}{\mathcal{G}_{\boldsymbol{\lambda}(Y)} (\roots_{z,\RR})},
	\end{equation}
	where $\boldsymbol{\lambda}(Y)=(\lambda_1,\cdots,\lambda_N)$ with $\lambda_i = (y_i+i) - (y_N+N)$.
\end{defn}

We remark that the definition above is only valid for discrete points $v\in\roots_{z,\RR}$. Below we re-express the formula such that it is well defined for all $v\in\Omega_\RR\setminus\{0\}$. More explicitly, we have

\begin{lm}
	\label{lm:iich_reformulation}
	There exists a function $\hh_Y(v,u;z)$ analytically defined on $\Omega_\RR\times\Omega_\LL\times \disk(r_\mm)$ such that
	\begin{equation}
	\label{eq:iich_ana_ext}
	\iich_Y(v,u;z) = \left(\frac{u+1}{v+1}\right)^{y_N+N}\cdot \left(\ich_{\boldsymbol{\lambda}(Y)}(v,u) + \hh_Y(v,u;z)\right)
	\end{equation}
	for all $|z|<r_\mm$ and $(v,u)\in\roots_{z,\RR}\times(\Omega_\LL\setminus\{-1\})$. Here $\ich_{\boldsymbol{\lambda}}$ is a polynomial of $v$ and $u$ defined in Definition~\ref{def:ich}, and $\boldsymbol{\lambda}(Y)=(\lambda_1,\cdots,\lambda_N)$ with $\lambda_i=(y_i+i)-(y_N+N)$ for $1\le i\le N$. The function $\hh_Y$ also satisfies
	\begin{equation*}
	\hh_Y(v,u;0)=0
	\end{equation*}
	for all $(u,v)\in\Omega_\LL\times\Omega_\RR$.
\end{lm}

Note that the right hand side of~\eqref{eq:iich_ana_ext} is defined on a larger domain than the left hand side $\iich_Y(v,u;z)$, but they agree on the set where $\iich_Y(v,u;z)$ is defined.
\begin{proof}[Proof of Lemma~\ref{lm:iich_reformulation}]
	The idea is to reformulate $\mathcal{G}_{\boldsymbol{\lambda}(Y)} ((\roots_{z,\RR}\setminus \{v\}) \cup \{u\})$ and analytically extend it to $\Omega_\RR\times\Omega_\LL\times \disk(r_\mm)$.
	
	First we express the symmetric function $\mathcal{G}_{\boldsymbol{\lambda}(Y)} (v_1,\cdots,v_N)$ in terms of finitely many power sum symmetric functions. More explicitly, we write
	\begin{equation}
	\label{eq:expan_G1}
	\mathcal{G}_{\boldsymbol{\lambda}(Y)} (v_1,\cdots,v_N) =1+ \sum_{\boldsymbol{\mu}=(\mu_1,\cdots)}\tilde c_{\boldsymbol{\lambda},\boldsymbol{\mu}} p_{\boldsymbol{\mu}}(v_1,\cdots,v_N),
	\end{equation}
	where $\boldsymbol{\mu}=(\mu_1,\cdots)$ satisfies $N\ge \mu_1\ge \cdots$, $|\boldsymbol{\mu}|=\mu_1+\cdots \le |\boldsymbol{\lambda}|$, and $\mu_1\ge 1$. The function
	\begin{equation*}
	p_{\boldsymbol{\mu}}(v_1,\cdots,v_N)=\prod_{\mu_k\ge 1}(v_1^{\mu_k}+\cdots+v_N^{\mu_k}).
	\end{equation*}
	We remark that the expansion~\eqref{eq:expan_G1} might be different from~\eqref{eq:expan_G} since the number of variables in~\eqref{eq:expan_G} is assumed to be larger than $|\boldsymbol{\lambda}|$. We use $\tilde c_{\boldsymbol{\lambda},\boldsymbol{\mu}}$ here to mark the possible difference. In the case when $ |\boldsymbol{\lambda}|\le N$, these coefficients are identical to  $c_{\boldsymbol{\lambda},\boldsymbol{\mu}}$'s in~\eqref{eq:expan_G}.
	
	We will take two different sets of variables in~\eqref{eq:expan_G1} and obtain an identity between $\mathcal{G}_{\boldsymbol{\lambda}(Y)} ((\roots_{z,\RR}\setminus \{v\}) \cup \{u\})$ and $\ich_{\boldsymbol{\lambda}(Y)}(v,u)$. The first set of variables is  $\{v_1,\cdots,v_N\}=(\roots_{z,\RR}\setminus \{v\}) \cup \{u\}$. This gives
	\begin{equation}
	\label{eq:extention_G1}
	\begin{split}
	\mathcal{G}_{\boldsymbol{\lambda}(Y)} ((\roots_{z,\RR}\setminus \{v\}) \cup \{u\}) 
	&=1+\sum_{\boldsymbol{\mu}=(\mu_1,\cdots)} \tilde c_{\boldsymbol{\lambda},\boldsymbol{\mu}}\prod_{\mu_k\ge 1} (h_{\mu_k}(z) +u^{\mu_k}-v^{\mu_k})\\
	&=1+\tilde h_Y(v,u;z)+ \sum_{\boldsymbol{\mu}=(\mu_1,\cdots)} \tilde c_{\boldsymbol{\lambda},\boldsymbol{\mu}}\prod_{\mu_k\ge 1} (u^{\mu_k}-v^{\mu_k}),
	\end{split}
	\end{equation}
	where the function
	$$
	h_j(z):=\sum_{v'\in\roots_{z,\RR}}(v')^j,\qquad j\ge 1
	$$
	is an analytic function of $z\in\disk(r_\mm)$ with $h_j(0)=0$, and
	\begin{equation*}
	\tilde h_Y(v,u;z) = \sum_{\boldsymbol{\mu}=(\mu_1,\cdots)} \tilde c_{\boldsymbol{\lambda},\boldsymbol{\mu}}\left(\prod_{\mu_k\ge 1} (h_{\mu_k}(z) +u^{\mu_k}-v^{\mu_k})-\prod_{\mu_k\ge 1} (u^{\mu_k}-v^{\mu_k})\right)
	\end{equation*}
	is an analytic function of $(v,u,z)\in\Omega_\RR\times\Omega_\LL\times \disk(r_\mm)$, which actually is a polynomial of $v$ and $u$, with $\tilde h_Y(v,u;0)=0$ for any pair $(v,u)$. 
	
	The other set of variables we insert in~\eqref{eq:expan_G1} is $\{v_1,\cdots,v_N\}=\{u,v\xi,v\xi^2,\cdots,v\xi^{N-1}\}$ with $\xi=e^{2\pi\ii/N}$. This formula includes the desired term $\ich_{\boldsymbol{\lambda}}(v,u)$. More explicitly, by applying~\eqref{eq:alter_ich_ext}, we have
	\begin{equation}
	\label{eq:ich01}
	\begin{split}
	\ich_{\boldsymbol{\lambda}}(v,u) &= \mathcal{G}_{\boldsymbol{\lambda}}(u,v\xi,\cdots,v\xi^{N-1})+ v^N\cdot r(v,u)\\
	&= 1+\sum_{\boldsymbol{\mu}=(\mu_1,\cdots)}\tilde c_{\boldsymbol{\lambda},\boldsymbol{\mu}} p_{\boldsymbol{\mu}}(u,v\xi,v\xi^2,\cdots,v\xi^{N-1}) + v^N\cdot r(v,u)
	\end{split}
	\end{equation}
	for some polynomial $r(v,u)$. Note that $(v\xi)^{\mu_k}+(v\xi^2)^{\mu_k}+\cdots+(v\xi^{N-1})^{\mu_k}=-v^{\mu_k}$ if $1\le \mu_k\le N-1$, or $(N-1)v^N$ if $\mu_k=N$. We have
	\begin{equation*}
	\begin{split}
	p_{\boldsymbol{\mu}}(u,v\xi,v\xi^2,\cdots,v\xi^{N-1})&= \prod_{1\le\mu_k\le N-1} (u^{\mu_k}-v^{\mu_k})\prod_{\mu_k=N}(u^{\mu_k}-v^{\mu_k} +Nv^N)\\
	&=\prod_{\mu_k\ge 1}(u^{\mu_k}-v^{\mu_k}) + v^N\cdot \text{a polynomial of $v$ and $u$}.
	\end{split}
	\end{equation*}
	By inserting this in~\eqref{eq:ich01} we immediately obtain
	\begin{equation}
	\label{eq:ich02}
	\ich_{\boldsymbol{\lambda}}(v,u)= 1 + v^N\cdot \tilde r(v,u)+ \sum_{\boldsymbol{\mu}=(\mu_1,\cdots)} \tilde c_{\boldsymbol{\lambda},\boldsymbol{\mu}}\prod_{\mu_k\ge 1} (u^{\mu_k}-v^{\mu_k}),
	\end{equation}
	where $\tilde r(v,u)$ is a polynomial of $v$ and $u$.
	
	Now we combine~\eqref{eq:extention_G1} and~\eqref{eq:ich02} and write
	\begin{equation*}
	\mathcal{G}_{\boldsymbol{\lambda}(Y)} ((\roots_{z,\RR}\setminus \{v\}) \cup \{u\}) =\ich_{\boldsymbol{\lambda}}(v,u) -v^N\cdot \tilde r(v,u) +\tilde h_Y(v,u;z).
	\end{equation*}
	We further express $v^N=\frac{z}{(v+1)^{L-N}}$ since $v\in\roots_{z,\RR}$. This gives
	\begin{equation}
	\label{eq:G_01}
	\mathcal{G}_{\boldsymbol{\lambda}(Y)} ((\roots_{z,\RR}\setminus \{v\}) \cup \{u\}) =\ich_{\boldsymbol{\lambda}}(v,u) -z\cdot \frac{\tilde r(v,u)}{(v+1)^{L-N}} +\tilde h_Y(v,u;z).
	\end{equation}
	Note that the expression on the right is analytically defined for $(v,u,z)\in\Omega_\RR\times\Omega_\LL\times\disk(r_\mm)$.
	
	Finally we prove the lemma. Note that the function
	\begin{equation*}
	\tilde g_Y(z):=\mathcal{G}_{\boldsymbol{\lambda}(Y)} (\roots_{z,\RR})
	\end{equation*}
	is an analytic function for $z\in\disk(r_\mm)$ with $\tilde g_Y(0)=1$. Moreover, it is nonzero in the disk $\disk(r_\mm)$ by the assumption of $r_\mm$ (see~\eqref{eq:r_max}). Thus by the definition of $\iich_Y$ and the equation~\eqref{eq:G_01}, we have the expression~\eqref{eq:iich_ana_ext} with
	\begin{equation*}
	h_Y(v,u;z):= \frac{\ich_{\boldsymbol{\lambda}}(v,u) -z\cdot \tilde r(v,u)\cdot (v+1)^{-L+N} +\tilde h_Y(v,u;z)}{\tilde g_Y(z)}-\ich_{\boldsymbol{\lambda}}(v,u).
	\end{equation*}
	This function is analytically defined for $(v,u,z)\in\Omega_\RR\times\Omega_\LL\times\disk(r_\mm)$ since each term is analytic and the denominator is nonzero. Moreover, we have $h_Y(v,u;0)=0$ for all $(v,u)\in\Omega_\RR\times\Omega_\LL$ by using the facts $\tilde h_Y(v,u;0)=0$ and $\tilde g_Y(0)=1$. This finishes the proof.
\end{proof}

\subsection{Function $ \mathscr{C}_Y(\hat z_1,\cdots,\hat z_m) $ and proof of Lemma~\ref{lm:lemma1}}
\label{sec:proof_lemma1}

The function $\mathscr{C}_Y(\hat z_1,\cdots,\hat z_m)$ is defined to be (see \cite[Definition 3.9 and 3.13]{Baik-Liu21})
\begin{equation*}
\begin{split}
\mathscr{C}_Y(\hat z_1,\cdots,\hat z_m)
=\left[\prod_{\ell=2}^m \frac{\hat z_{\ell-1}}{\hat z_{\ell-1} -\hat z_\ell}\right] \cdot \mathcal{E}_Y(\hat z_1) \cdot \mathscr{A}(\hat z_1,\cdots,\hat z_m),
\end{split}
\end{equation*}
where $\mathcal{E}_Y(\hat z_1)$ is defined in Definition~\ref{def:energy}, $\mathscr{A} = \mathscr{A}_{1} \cdot \mathscr{A}_{2} \cdot \mathscr{A}_{3}$ with
\begin{equation*}
\begin{split}
\mathscr{A}_{1}(\hat z_1,\cdots,\hat z_m)&:= \prod_{\ell=1}^m \left[ \prod_{u\in\roots_{\hat z_\ell,\LL}} (-u)^{k_{\ell-1} - k_\ell} \prod_{v\in\roots_{\hat z_\ell,\RR}} (v+1)^{(a_{\ell-1}+k_{\ell-1}) - (a_\ell+k_\ell)} e^{(t_\ell -t_{\ell-1})v}\right],\\
\mathscr{A}_{2}(\hat z_1,\cdots,\hat z_m)&:= \prod_{\ell=1}^m \frac{ \prod_{u\in\roots_{\hat z_\ell,\LL}} (-u)^N \prod_{v\in\roots_{\hat z_\ell,\RR}} (v+1)^{L-N}}
{\prod_{(u,v)\in\roots_{\hat z_\ell,\LL} \times \roots_{\hat z_\ell,\RR}} (v-u)},\\
\mathscr{A}_{3}(\hat z_1,\cdots,\hat z_m)&:= \prod_{\ell=2}^m \frac{\prod_{(u,v)\in\roots_{\hat z_{\ell-1},\LL} \times \roots_{\hat z_\ell,\RR}} (v-u)}{ \prod_{u\in\roots_{\hat z_{\ell-1},\LL}} (-u)^N \prod_{v\in\roots_{\hat z_\ell,\RR}} (v+1)^{L-N}}.
\end{split}
\end{equation*}
In the definition of $\mathscr{A}_{1}$ above, we set $a_0=k_0=t_0=0$.

It is obvious that $\mathscr{A}_{i}$ functions are analytic for $(\hat z_1,\cdots,\hat z_m)\in (\disk_0(r_m))^m\subset (\disk_0(r_c))^m$ since locally each Bethe root $w\in\roots_{\hat z,\RR}\cup\roots_{\hat z,\LL}$ as a function of $\hat z$ is analytic  when $\hat z\in\disk_0(r_c)$.  Moreover, recall that all Bethe roots in $\roots_{\hat z,\LL}$ go to $-1$ and all Bethe roots in $\roots_{\hat z,\RR}$ go to $0$ when $\hat z\to 0$. We know that these $\mathscr{A}_{i}$ functions could be analytically extended to $(\disk (r_m))^m$, i.e., they are all well defined if some $\hat z_\ell=0$. By replacing all $u$'s by $-1$ and all $v$'s  by $0$ in the formulas, we have
\begin{equation*}
\mathscr{A}_{1}(0,\cdots,0)=\mathscr{A}_{2}(0,\cdots,0)=\mathscr{A}_{3}(0,\cdots,0)=1.
\end{equation*}
Also recall that $\mathcal{E}_Y(\hat z_1)$ is analytic within $|\hat z_1|<r_\mm$ with $\mathcal{E}_Y(0)=1$. We conclude that $\tilde{\mathscr{C}}_Y(z_0,\cdots,z_{m-1}) = \mathscr{C}_Y(z_0,z_0z_1,\cdots,z_0\cdots z_{m-1})$ is analytic for $(z_0,\cdots,z_{m-1})\in \disk(r_\mm)\times \disk^{m-1}$. Moreover,
\begin{equation*}
\tilde{\mathscr{C}}_Y(0,z_1,\cdots,z_{m-1}) = \prod_{\ell=1}^{m-1} \frac{1}{1-z_\ell}.
\end{equation*}
This finishes the proof of Lemma~\ref{lm:lemma1}.

\subsection{Function  $\mathscr{D}_Y(\hat z_1,\cdots,\hat z_m)$ and proof of Lemma~\ref{lm:lemma2} }
\label{sec:proof_lemma2}

Similar to $\mathcal{D}_Y$, the function $\mathscr{D}_Y(\hat z_1,\cdots,\hat z_m)$ has both Fredholm determinant and series expansion representations. We only use the series expansion representation of $\mathscr{D}_Y(\hat z_1,\cdots,\hat z_m)$ to prove Lemma~\ref{lm:lemma2}.

We remind the notation conventions $\Delta(W),\Delta(W;W')$ and $f(W)$ we introduced at the beginning of Section~\ref{sec:series_expansion}. 

The following definition is the series expansion representation  (by applying Proposition~\ref{prop:equivalence_Fredholm_series}) of Definition 3.10 of \cite{Baik-Liu21}. This series expansion formula for the case of step initial condition was introduced in \cite[Lemma 4.4]{Baik-Liu19}.
\begin{defn}
	\label{def:D_Y_series_period}
	We define\begin{equation*}
	\mathscr{D}_Y(\hat z_1,\cdots, \hat z_{m}):=\sum_{\boldsymbol{n}\in(\intZ_{\ge 0})^m}\frac{1}{(\boldsymbol{n}!)^2}\mathscr{D}_{\boldsymbol{n},Y}(\hat z_1,\cdots,\hat z_{m})
	\end{equation*}
	with $\boldsymbol{n}!=n_1!\cdots n_m!$ for $\boldsymbol{n}=(n_1,\cdots,n_m)$. Here
	\begin{equation*}
	\begin{split}
	&\mathscr{D}_{\boldsymbol{n},Y}(\hat z_1,\cdots,\hat z_{m})\\
	&
	=\sum_{\substack{U^{(\ell)}=(u_1^{(\ell)},\cdots,u_{n_\ell}^{(\ell)}) \in (\roots_{\hat z_\ell,\LL})^{n_\ell} \\ 
					  V^{(\ell)}=(v_1^{(\ell)},\cdots,v_{n_\ell}^{(\ell)})\in (\roots_{\hat z_\ell,\RR})^{n_\ell} \\ 
				     \ell=1,\cdots,m	}} \left[(-1)^{n_1(n_1+1)/2}\frac{\Delta(U^{(1)};V^{(1)})}{\Delta(U^{(1)})\Delta(V^{(1)})} \det\left[\frac{\iich_{Y}(v_i^{(1)},u_j^{(1)};\hat z_1)}{v_i^{(1)}-u_j^{(1)}}\right]_{i,j=1}^{n_1}\right]\\
	&\qquad\cdot \left[\prod_{\ell=1}^m\frac{(\Delta(U^{(\ell)}))^2(\Delta(V^{(\ell)}))^2}{(\Delta(U^{(\ell)};V^{(\ell)}))^2}f_\ell(U^{(\ell)}) f_\ell(V^{(\ell)}) \cdot \left(\mathfrak{h}(U^{(\ell)},z_\ell)\right)^2\left( \mathfrak{h}(V^{(\ell)},z_\ell)\right)^2  \cdot J(U^{(\ell)}) J(V^{(\ell)}) \right]\\
	&\qquad\cdot \left[\prod_{\ell=1}^{m-1} \frac{\Delta(U^{(\ell)};V^{(\ell+1)})\Delta(V^{(\ell)};U^{(\ell+1)})}{\Delta(U^{(\ell)};U^{(\ell+1)})\Delta(V^{(\ell)};V^{(\ell+1)})}\cdot \frac{\left(1-\hat z_{\ell+1}/\hat z_\ell\right)^{n_\ell}\left(1-\hat z_\ell/\hat z_{\ell+1}\right)^{n_{\ell+1}}}{\mathfrak{h}(U^{(\ell)};\hat z_{\ell+1})\mathfrak{h}(V^{(\ell)};\hat z_{\ell+1})\mathfrak{h}(U^{(\ell+1)};\hat z_{\ell})\mathfrak{h}(V^{(\ell+1)};\hat z_{\ell})}\right].
	\end{split}
	\end{equation*}
	The function $\iich_{Y}$ is defined in~\eqref{eq:def_ich_periodic}. The function $f_\ell$ is defined by	
	\begin{equation*}
%	\label{eq:fi}
	f_\ell(w):=\begin{dcases}
	\frac{F_\ell(w)}{ F_{\ell-1}(w)}, & w\in \Omega_\LL\setminus\{-1\},\\
	\frac{F_{\ell-1}(w)}{ F_\ell(w)}, & w\in \Omega_\RR\setminus\{0\},
	\end{dcases}
	\end{equation*}
	with
	\begin{equation*}
	F_\ell(w) := \begin{dcases}
	w^{k_\ell} (w+1)^{-a_\ell-k_\ell} e^{t_\ell w}, & \ell=1,\cdots,m,\\
	1, & \ell=0.
	\end{dcases}
	\end{equation*}
	This is consistent with~\eqref{eq:fi}. The function $\mathfrak{h}$ is defined by~\eqref{eq:def_frakh}. We also clarify that the notation
	\begin{equation*}
	\mathfrak{h}(W,\hat z):=\mathfrak{h}(w_1,\hat z)\cdots \mathfrak{h}(w_n,\hat z)
	\end{equation*}
	for any vector $W=(w_1,\cdots,w_n)$ and any complex number $|\hat z|<r_\mm$.  Finally, the function
	\begin{equation*}
	J(w)=\frac{w(w+1)}{Lw+N}=\frac{q(w)}{q'(w)}
	\end{equation*}
	is consistent with~\eqref{eq:J}.
\end{defn}

Since $\roots_{\hat z_\ell,\LL}$ and $\roots_{\hat z_\ell,\RR}$  have finite sizes  $L-N$ and $N$ respectively, the factor $\Delta(U^{(\ell)})\Delta(V^{(\ell)})=0$ if $n_\ell > \min\{N,L-N\}$. Thus $\mathscr{D}_{\boldsymbol{n},Y}(\hat z_1,\cdots,\hat z_m)=0$ if $|\boldsymbol{n}|=n_1+\cdots+n_m> m\cdot \min\{N,L-N\}$. This implies the summation in the definition of $\mathscr{D}_Y(\hat z_1,\cdots,\hat z_m)$ only involves finitely many nonzero terms.

\vspace{0.3cm}

Now we proceed to prove Lemma~\ref{lm:lemma2} by using Proposition~\ref{prop:Cauchy_sum_over_roots_ext}. We need to rewrite $\mathscr{D}_{\boldsymbol{n},Y}(\hat z_1,\cdots,\hat z_m)$ in the form of $\GG(z_0,\cdots,z_{m-1})$ defined in~\eqref{eq:GG1}. Here the variables $z_0,\cdots,z_{m-1}$ were introduced before as in~\eqref{eq:hat_z}, which also match~\eqref{eq:hatz} in the setting of Proposition~\ref{prop:Cauchy_sum_over_roots_ext}. They satisfy
\begin{equation*}
\hat z_\ell =\prod_{j=0}^{\ell-1}z_j,\qquad \ell=1,\cdots,m.
\end{equation*}We also rewrite $\iich_Y(v_i^{(1)},u_j^{(1)};\hat z_1)$ in the summand of $\mathscr{D}_{\boldsymbol{n},Y}(\hat z_1,\cdots,\hat z_m)$  by its analytical extension using Lemma~\ref{lm:iich_reformulation}. We write
\begin{equation}
\label{eq:tilde_scrD}
\begin{split}
\tilde{\mathscr{D}}_{\boldsymbol{n},Y}(z_0,\cdots,z_{m-1})&:=\mathscr{D}_{\boldsymbol{n},Y}(\hat z_1,\cdots,\hat z_m)\\
&=\left[\prod_{\ell=1}^{m-1} \left(1-z_\ell\right)^{n_{\ell}} \left(1-\frac{1}{z_\ell}\right)^{n_{\ell+1}} \right]\cdot \GG_{\boldsymbol{n},Y}\left(z_0,\cdots,z_{m-1}\right)
\end{split}
\end{equation}
with
\begin{equation*}
\begin{split}
&\GG_{\boldsymbol{n},Y}(z_0,\cdots,z_{m-1})\\
&:= \sum_{\substack{U^{(\ell)} \in \left(\roots_{\hat z_\ell,\LL}\right)^{n_\ell}\\
		V^{(\ell)} \in \left(\roots_{\hat z_\ell,\RR}\right)^{n_\ell}\\
		\ell=1,\cdots,m}} 
\left[\prod_{\ell=1}^m J(U^{(\ell)}) J(V^{(\ell)})\right]
H_Y\left(U^{(1)},\cdots,U^{(m)};V^{(1)},\cdots,V^{(m)};z_0,\cdots,z_{m-1}\right).
\end{split}
\end{equation*}
The function
\begin{equation*}
\begin{split}
&H_Y\left(U^{(1)},\cdots,U^{(m)};V^{(1)},\cdots,V^{(m)};z_0,\cdots,z_{m-1}\right)\\
&:=\left[\prod_{\ell=1}^{m-1} \Ch\left(U^{(\ell)};U^{(\ell+1)}\right)\Ch\left(V^{(\ell)};V^{(\ell+1)}\right)\right]
\cdot A_Y\left(U^{(1)},\cdots,U^{(m)}; V^{(1)},\cdots,V^{(m)};z_0,\cdots,z_{m-1}\right),
\end{split}
\end{equation*}
where $\Ch(W;W')=\frac{\Delta(W)\Delta(W')}{\Delta(W;W')}$ is the Cauchy-type factor defined in~\eqref{eq:def_cauchy}, and
\begin{equation}
\label{eq:AY}
\begin{split}
&A_Y\left(U^{(1)},\cdots,U^{(m)}; V^{(1)},\cdots,V^{(m)};z_0,\cdots,z_{m-1}\right)\\
&:=
\left[(-1)^{n_1(n_1+1)/2}\Delta(U^{(1)};V^{(1)}) \det\left[\left(\frac{u_j^{(1)}+1}{v_i^{(1)}+1}\right)^{y_N+N}\cdot\frac{\ich_{\boldsymbol{\lambda}(Y)}(v_i^{(1)},u_j^{(1)}) + \hh_Y(v_i^{(1)},u_j^{(1)};\hat z_1)}{v_i^{(1)}-u_j^{(1)}}\right]_{i,j=1}^{n_1}\right]\\
&\qquad \cdot \left[\Delta(U^{(m)})\Delta(V^{(m)})\right]
\cdot \left[\prod_{\ell=1}^m\frac{f_\ell(U^{\ell}) f_\ell(V^{(\ell)})}{\left(\Delta(U^{(\ell)};V^{(\ell)})\right)^2}\cdot \left(\mathfrak{h}(U^{(\ell)};\hat z_\ell) \mathfrak{h}(V^{(\ell)};\hat z_{\ell})\right)^2 \right]\\
&\qquad
\cdot \left[\prod_{\ell=1}^{m-1}\frac{\Delta(V^{(\ell)};U^{(\ell+1)}) \Delta(U^{(\ell)};V^{(\ell+1)})}{\mathfrak{h}(U^{(\ell+1)};\hat z_\ell) \mathfrak{h}(U^{(\ell)};\hat z_{\ell+1}) \mathfrak{h}(V^{(\ell+1)};\hat z_\ell) \mathfrak{h}(V^{(\ell)};\hat z_{\ell+1})}\right].
\end{split}
\end{equation}
Recall that $\mathfrak{h}(w;\hat z)$ is analytic and nonzero for $(w,\hat z)\in (\Omega_\LL\cup\Omega_\RR)\times\disk(r_\mm)$, see the discussions after the equation~\eqref{eq:def_frakh}. The function $h_Y(v,u;\hat z)$ is analytic for $(v,u,\hat z)\in\Omega_\RR\times \Omega_\LL\times\disk(r_\mm)$ by Lemma~\ref{lm:iich_reformulation}. We also recall that $f_\ell(w)$ is analytic for $w\in(\Omega_\LL\setminus\{-1\})\cup(\Omega_\RR\setminus\{0\})$. $v-u$ is nonzero for $(v,u)\in \Omega_\RR\times\Omega_\LL$ since $\Omega_\LL\cap\Omega_\RR=\emptyset$. $\ich_{\boldsymbol{\lambda}(Y)}(v,u)$ is a polynomial by Definition~\ref{def:ich}. Moreover, $\hat z_\ell$ depends on $z_0,\cdots,z_{m-1}$ analytically.  These facts imply that $A$ is analytic for each $u_{i_\ell}^{(\ell)}\in\Omega_{\LL}\setminus\{-1\}$, each $v_{i_\ell}^{(\ell)}\in\Omega_{\LL}\setminus\{0\}$, and each $z_0\in\disk(r_\mm)$ and $z_\ell\in\disk$.

Now we assume that $q(w)$ dominates $H_{Y}$ at $w=-1$ and $w=0$. The proof of this assumption will be postponed to the end of this section. With this assumption, Proposition~\ref{prop:Cauchy_sum_over_roots_ext} is applicable here. We obtain that $\GG_{\boldsymbol{n},Y}(z_0,\cdots,z_{m-1})$ is analytic for $(z_0,\cdots,z_{m-1})\in\disk(r_\mm)\times\disk^{m-1}$, and
\begin{equation*}
\begin{split}
&\GG_{\boldsymbol{n},Y}(0,z_1,\cdots,z_{m-1})\\
&=\prod_{\ell=2}^{m}  \prod_{{i_\ell}=1}^{n_{\ell}} \left[\frac{1}{1-z_{\ell}}\int_{\Sigma_{\ell,\LL}^\inn} \ddbarr{u_{i_\ell}^{(\ell)}}
-\frac{z_{\ell}}{1-z_{\ell}}\int_{\Sigma_{\ell,\LL}^\out} \ddbarr{u_{i_\ell}^{(\ell)}} \right]
\cdot \prod_{{i_1}=1}^{n_{1}} \int_{\Sigma_{1,\LL}} \ddbarr{u_{i_1}^{(1)}}\\
&\quad\prod_{\ell=2}^{m}  \prod_{{i_\ell}=1}^{n_{\ell}} \left[\frac{1}{1-z_{\ell}}\int_{\Sigma_{\ell,\RR}^\inn} \ddbarr{v_{i_\ell}^{(\ell)}}
-\frac{z_{\ell}}{1-z_{\ell}}\int_{\Sigma_{\ell,\RR}^\out} \ddbarr{v_{i_\ell}^{(\ell)}} \right]
\cdot \prod_{{i_1}=1}^{n_{1}} \int_{\Sigma_{1,\RR}} \ddbarr{v_{i_1}^{(1)}}\\
&
\quad H_Y(U^{(1)},\cdots,U^{(m)};V^{(1)},\cdots,V^{(m)};0,z_1,\cdots,z_{m-1}).
\end{split}
\end{equation*}
Here the contours are the same as in Proposition~\ref{prop:Cauchy_sum_over_roots_ext} and Section~\ref{sec:spaces_operators}. On the other hand, by using the following facts
$\mathfrak{h}(w;0)=1$, $h_Y(v,u;0)=0$, and $\hat z_\ell=0$ for all $\ell$ if $z_0=0$, we immediately have
\begin{equation*}
\begin{split}
&A_Y\left(U^{(1)},\cdots,U^{(m)}; V^{(1)},\cdots,V^{(m)};0,z_1\cdots,z_{m-1}\right)\\
&=
\left[(-1)^{n_1(n_1+1)/2}\Delta(U^{(1)};V^{(1)}) \det\left[\left(\frac{u_j^{(1)}+1}{v_i^{(1)}+1}\right)^{y_N+N}\cdot\frac{\ich_{\boldsymbol{\lambda}(Y)}(v_i^{(1)},u_j^{(1)}) }{v_i^{(1)}-u_j^{(1)}}\right]_{i,j=1}^{n_1}\right]\\
&\quad \cdot \left[\Delta(U^{(m)})\Delta(V^{(m)})\right]
\cdot \left[\prod_{\ell=1}^m\frac{f_\ell(U^{\ell}) f_\ell(V^{(\ell)})}{\left(\Delta(U^{(\ell)};V^{(\ell)})^2\right)}  \right]
\cdot \left[\prod_{\ell=1}^{m-1}\Delta(V^{(\ell)};U^{(\ell+1)}) \Delta(U^{(\ell)};V^{(\ell+1)})\right],
\end{split}
\end{equation*}
and, by inserting $\Kess_Y$ defined in Definition~\ref{def:KY_ess},
\begin{equation*}
\begin{split}
&H_Y\left(U^{(1)},\cdots,U^{(m)};V^{(1)},\cdots,V^{(m)};0,z_1,\cdots,z_{m-1}\right)\\
&= \left[(-1)^{n_1(n_1+1)/2}\frac{\Delta(U^{(1)};V^{(1)})}{\Delta(U^{(1)})\Delta(V^{(1)})} \det\left[\Kess_Y(v_i^{(1)},u_j^{(1)})\right]_{i,j=1}^{n_1}\right]\\
&\quad\cdot \left[\prod_{\ell=1}^m\frac{(\Delta(U^{(\ell)}))^2(\Delta(V^{(\ell)}))^2}{(\Delta(U^{(\ell)};V^{(\ell)}))^2}f_\ell(U^{(\ell)}) f_\ell(V^{(\ell)})\right]\cdot \left[\prod_{\ell=1}^{m-1} \frac{\Delta(U^{(\ell)};V^{(\ell+1)})\Delta(V^{(\ell)};U^{(\ell+1)})}{\Delta(U^{(\ell)};U^{(\ell+1)})\Delta(V^{(\ell)};V^{(\ell+1)})}\right].
\end{split}
\end{equation*}

\vspace{0.2cm}
Now we come back to~\eqref{eq:tilde_scrD}. By the above results of $\mathcal{G}_{\boldsymbol{n},Y}$, we know that $\tilde{\mathscr{D}}_{\boldsymbol{n},Y}(z_0,\cdots,z_{m-1})$ is analytic in $\disk(r_\mm)\times\disk_0^{m-1}$, with
\begin{equation*}
\tilde{\mathscr{D}}_{\boldsymbol{n},Y}(0,z_1,\cdots,z_{m-1})=\mathcal{D}_{\boldsymbol{n},Y}(z_1,\cdots,z_{m-1}).
\end{equation*}
Here $\mathcal{D}_{\boldsymbol{n},Y}$ is defined in~\eqref{eq:D_nY}. Recall the definition of $\tilde{\mathscr{D}}_Y(z_0,\cdots,z_{m-1})$ after the equation~\eqref{eq:aux_002}, $$\tilde{\mathscr{D}}_Y(z_0,\cdots,z_{m-1})=\mathscr{D}_Y(\hat z_1,\cdots, \hat z_{m})=\sum_{\boldsymbol{n}\in(\intZ_{\ge 0})^m}\frac{1}{(\boldsymbol{n}!)^2}\mathscr{D}_{\boldsymbol{n},Y}(\hat z_1,\cdots,\hat z_{m})=\sum_{\boldsymbol{n}\in(\intZ_{\ge 0})^m}\frac{1}{(\boldsymbol{n}!)^2}\tilde{\mathscr{D}}_{\boldsymbol{n},Y}(z_0,\cdots,z_{m-1}).$$ We immediately obtain that $\tilde{\mathscr{D}}_Y(z_0,\cdots,z_{m-1})$ is analytic in $\disk(r_\mm)\times\disk_0^{m-1}$ with 
$$\tilde{\mathscr{D}}_Y(0,z_1\cdots,z_{m-1})=\sum_{\boldsymbol{n}\in(\intZ_{\ge 0})^m}\frac{1}{(\boldsymbol{n}!)^2}\mathcal{D}_{\boldsymbol{n},Y}(z_1,\cdots,z_{m-1})=\mathcal{D}_Y(z_1,\cdots,z_{m-1}).$$ This finishes the proof of Lemma~\ref{lm:lemma2}.

\vspace{0.4cm}

It remains to prove the assumption that $q(w)$ dominates $H_{Y}$ at $w=-1$ and $w=0$. The two cases for $w=-1$ and $w=0$ are similar. Hence we only provide the proof for $w=-1$ and omit the other case.

For $w=-1\in\Omega_\LL$, we need to verify that along any Cauchy chain $u_{j_s}^{(s)},u_{j_{s+1}}^{(s+1)},\cdots,u_{j_{s'}}^{(s')}$, 
\begin{equation}
\label{eq:check_dominating}
q(w)\left.A_Y\left(U^{(1)},\cdots,U^{(m)}; V^{(1)},\cdots,V^{(m)};z_0,\cdots,z_{m-1}\right)\right|_{u_{j_s}^{(s)}=u_{j_{s+1}}^{(s+1)}=\cdots=u_{j_{s'}}^{(s')}=w}
\end{equation}
is analytic at $w=-1$, when all other coordinates of $u_{i_\ell}^{(\ell)}$'s are fixed in $\Omega_\LL\setminus\{-1\}$, $v_{i_\ell}^{(\ell)}$'s are fixed in $\Omega_\RR\setminus\{0\}$, and $(z_0,\cdots,z_{m-1})\in\disk(r_\mm)\times \disk^{m-1}$. Here $1\le s\le s'\le m$ and $j_s,\cdots,j_{s'}$ are positive numbers less than
$n_s,\cdots,n_{s'}$ respectively.

By the formula of $A_Y$ in~\eqref{eq:AY}, we could find that all the singularities for $u_{i_\ell}^{(\ell)}=-1$ are coming from the function $f_\ell(u_{i_\ell}^{(\ell)})= (u_{i_\ell}^{(\ell)})^{k_\ell-k_{\ell-1}} (u_{i_\ell}^{(\ell)}+1)^{(a_{\ell-1}+k_{\ell-1})-(a_{\ell}+k_\ell)}e^{(t_\ell-t_{\ell-1})u_{i_\ell}^{(\ell)}}$ for $\ell\ge 1$, and a possible extra singularity from $(u_{i_\ell}^{(\ell)}+1)^{y_N+N}$ factor when $\ell=1$. On the other hand, $q(w)=w^N(w+1)^{L-N}$ has the factor $(w+1)^{L-N}$. Thus the order of $(w+1)$ in~\eqref{eq:check_dominating} is at least $(L-N) + (a_{s-1}+k_{s-1}) - (a_{s'}+k_{s'})$ for $s > 1$ and $(L-N)  - (a_{s'}+k_{s'}) + y_N+N$ for $s = 1$. Both numbers are non-negative by~\eqref{eq:assumption_a} and the assumption $L\ge \max \{a_1+k_1,\cdots,a_m+k_m\} -y_N$. Thus~\eqref{eq:check_dominating} is analytic at $w=-1$ when other coordinates are fixed.

\section{Proof of Proposition~\ref{prop:Cauchy_sum_over_roots}}
\label{sec:proof_Cauchy_sum}

In this section, we prove Proposition~\ref{prop:Cauchy_sum_over_roots} by induction on $\sum_{\ell=1}^{m-1}|I^{(\ell)}\times J^{(\ell+1)}|$, which is also the total degree of denominators in the Cauchy-type factors $\prod_{\ell=1}^{m-1}\Ch\left(W^{(\ell)}_{I^{(\ell)}}; W^{(\ell+1)}_{J^{(\ell+1)}}\right)$.

\subsection{Base step: $\sum_{\ell=1}^{m-1}|I^{(\ell)}\times J^{(\ell+1)}|=0$}
\label{sec:basic_step}

If $\sum_{\ell=1}^{m-1}|I^{(\ell)}\times J^{(\ell+1)}|=0$, then we have either $I^{(\ell)}=\emptyset$ or $J^{(\ell+1)}=\emptyset$ for each $\ell$. Thus $\Ch\left(W^{(\ell)}_{I^{(\ell)}}; W^{(\ell+1)}_{J^{(\ell+1)}}\right)=\Delta\left(W^{(\ell+1)}_{J^{(\ell+1)}}\right)$ or $\Delta\left(W^{(\ell)}_{J^{(\ell)}}\right)$, which are polynomials of the coordinates. Thus without loss of generality (up to modifying the function $A$), we only consider the case when $\Ch\left(W^{(\ell)}_{I^{(\ell)}}; W^{(\ell+1)}_{J^{(\ell+1)}}\right)=1$ for each $\ell$, and
\begin{equation*}
H(W^{(1)},\cdots,W^{(m)};z_0,\cdots,z_{m-1})=  A(W^{(1)},\cdots,W^{(m)};z_0,\cdots,z_{m-1}).
\end{equation*}

Now we reformulate $\GG(z_0,z_1,\cdots,z_m)$ in~\eqref{eq:GG}, the summation of $A \cdot \prod J(W^{(\ell)})$ over all $W^{(\ell)}\in \roots_{\hat z_\ell}^{n_\ell}$. Recall that $J(w):=q(w)/q'(w)$, and $\roots_{\hat z_\ell}$ is defined in~\eqref{eq:roots_hat} which are the roots of $q(w)=\hat z_\ell$ within $\Omega$. 

For each $\hat z_{\ell}\in\disk_0(r_\mm)$, we have the following roots summation formula
\begin{equation}
\label{eq:roots_sum_single}
\sum_{w\in\roots_{\hat z_\ell}} J(w)f(w)= \int_{\Gamma_{|\hat z_\ell|+\epsilon}} \frac{q(w)}{q(w)-\hat z_{\ell}} f(w) \ddbarr{w} -\int_{\Gamma_{|\hat z_\ell|-\epsilon}}\frac{q(w)}{q(w)-\hat z_{\ell}} f(w) \ddbarr{w}
\end{equation}
for any $f$ which is analytic within a neighborhood of $\Gamma_{|\hat z_\ell|}$, and $\epsilon>0$ is a sufficiently small positive number. Recall that $\Gamma_{r}=\{w\in\Omega: |q(w)| = r\}$ is  the contour defined in~\eqref{eq:Gamma_r}. The above formula~\eqref{eq:roots_sum_single} follows from evaluating the residues of $\frac{q(w)}{q(w)-\hat z_\ell} f(w)$ when deforming the contours from $\Gamma_{|\hat z_\ell|+\epsilon}$ to $\Gamma_{|\hat z_\ell|-\epsilon}$.

By applying~\eqref{eq:roots_sum_single} for all the coordinates of $W^{(\ell)}$'s, we obtain
\begin{equation*}
\begin{split}
&\GG(z_0,\cdots,z_{m-1})\\
&= \sum_{\substack{W^{(1)} \in \roots^{n_1}_{\hat z_1} }}\cdots\sum_{\substack{W^{(m)} \in \roots^{n_m}_{\hat z_m} }} \left[\prod_{\ell=1}^mJ(W^{(\ell)})\right]\cdot H(W^{(1)},\cdots,W^{(m)};z_0,\cdots,z_{m-1})\\
&=\prod_{\ell=1}^m\prod_{i_\ell=1}^{n_\ell}\left[ \int_{\Gamma_{|\hat z_\ell|+\epsilon}} \frac{q(w_{i_\ell}^{(\ell)})}{q(w_{i_\ell}^{(\ell)})-\hat z_{\ell}} \ddbarr{w_{i_\ell}^{(\ell)}}-\int_{\Gamma_{|\hat z_\ell|-\epsilon}} \frac{q(w_{i_\ell}^{(\ell)})}{q(w_{i_\ell}^{(\ell)})-\hat z_{\ell}} \ddbarr{w_{i_\ell}^{(\ell)}}\right] H(W^{(1)},\cdots,W^{(m)};z_0,\cdots,z_{m-1}).
\end{split}
\end{equation*}

Now we apply the assumption that $q(w)$ dominates $H$. It implies that the integrand is analytic for $w_{i_\ell}^{\ell}$ within the region bounded by $\Gamma_{|\hat z_\ell|-\epsilon}$ when all other coordinates are fixed. Thus the integral along $\Gamma_{|\hat z_\ell|-\epsilon}$ with respect to $w_{i_\ell}^{(\ell)}$ vanishes. We have
\begin{equation}
\label{eq:GG0_base}
\begin{split}
\GG(z_0,\cdots,z_{m-1})&=\prod_{\ell=1}^m\prod_{i_\ell=1}^{n_\ell}\left[ \int_{\Gamma_{|\hat z_\ell|+\epsilon}} \frac{q(w_{i_\ell}^{(\ell)})}{q(w_{i_\ell}^{(\ell)})-\hat z_{\ell}} \ddbarr{w_{i_\ell}^{(\ell)}}\right]H(W^{(1)},\cdots,W^{(m)};z_0,\cdots,z_{m-1})\\
&=\prod_{\ell=1}^m\prod_{i_\ell=1}^{n_\ell}\left[ \int_{\Gamma_{r_\mm -\epsilon'}} \frac{q(w_{i_\ell}^{(\ell)})}{q(w_{i_\ell}^{(\ell)})-\hat z_{\ell}} \ddbarr{w_{i_\ell}^{(\ell)}}\right]H(W^{(1)},\cdots,W^{(m)};z_0,\cdots,z_{m-1}),
\end{split}
\end{equation}
where we deformed the contours $\Gamma_{|\hat z_\ell|+\epsilon}$ to $\Gamma_{r_\mm -\epsilon'}$ for any sufficiently small $\epsilon'>0$ without encountering any pole. Recall that $\hat z_\ell=z_0z_1\cdots z_{\ell-1}$ and the factor $\frac{q(w_{i_\ell}^{(\ell)})}{q(w_{i_\ell}^{(\ell)})-\hat z_{\ell}}$ is analytic in $\hat z_\ell$ for $|\hat z_\ell|<r_\mm -\epsilon'$. Moreover, $H$ is analytic in $z_\ell$'s for given $w_{i_\ell}^{(\ell)}$'s on the contours of integration. Thus the formula~\eqref{eq:GG0_base} for $\GG(z_0,\cdots,z_{m-1})$ is analytic when $|z_0|<r_\mm-\epsilon'$ and $|z_1|<1,\cdots,|z_{m-1}|<1$. We could also drop this $\epsilon'$ since it could be chosen arbitrarily small. This proves that $\GG(z_0,\cdots,z_{m-1})$ is can be analytically extended to $\disk(r_\mm)\times \disk^{m-1}$.

Now we evaluate $\GG(0,z_1,\cdots,z_{m-1})$ in~\eqref{eq:GG0_base}. This gives all $\hat z_\ell=0$ by the definition of $\hat z_\ell$ in~\eqref{eq:hatz}. Hence
\begin{equation*}
\GG(0,z_1,\cdots,z_{m-1})= \prod_{\ell=1}^m \prod_{i_\ell=1}^{n_\ell}\left[ \int_{\Gamma_{r_\mm -\epsilon'}}  \ddbarr{w_{i_\ell}^{(\ell)}}\right]H(W^{(1)},\cdots,W^{(m)};0,z_1,\cdots,z_{m-1}).
\end{equation*}
Since the integrand is analytic for each $w_{i_\ell}^{(\ell)} \in \Omega\setminus\{0\}$, we could rewrite
\begin{equation*}
\int_{\Gamma_{r_\mm -\epsilon'}}  \ddbarr{w_{i_\ell}^{(\ell)}}=\frac{1}{1-z_{\ell-1}}\int_{\Sigma_{\ell}^\inn} \ddbarr{w_{i_\ell}^{(\ell)}}
-\frac{z_{\ell-1}}{1-z_{\ell-1}}\int_{\Sigma_{\ell}^\out} \ddbarr{w_{i_\ell}^{(\ell)}} 
\end{equation*}
for each $\ell=2,\cdots,m$ and 
\begin{equation*}
\int_{\Gamma_{r_\mm -\epsilon'}}  \ddbarr{w_{i_1}^{(1)}}= \int_{\Sigma_{1}} \ddbarr{w_{i_1}^{(1)}},
\end{equation*}where we omit the integrand $H$ in the above formulas, and the contours $\Sigma_{\ell}^\out$, $\Sigma_{\ell}^\inn$ for $2\le\ell\le m$, and $\Sigma_{1}$ are described in the proposition. They are simple closed contours within $\Omega$ and enclosing $0$. 

Although we did not encounter any poles from the integrand in the above integral decomposition, the idea of decomposing the integrals into two parts with inner and outer contours comes from managing the possible poles arising from the Cauchy-type factors. In the case when such poles are present, we need to keep track of the locations of the contours and treat the outer and inner contours separately.

After the above decomposition we immediately obtain the formula~\eqref{eq:def_G0} for $\GG(0,z_1,\cdots,z_{m-1})$. This finishes the base step of induction.

\subsection{Inductive step}

Now we assume the proposition holds for the cases of $\sum_{\ell=1}^{m-1}|I^{(\ell)}\times J^{(\ell+1)}|\le S-1$, and consider the case when $\sum_{\ell=1}^{m-1}|I^{(\ell)}\times J^{(\ell+1)}|=S \ge 1$.

Since $S\ge 1$, there exists a largest $s$, $2\le s\le m$, such that $I^{(s-1)}\times J^{(s)}$ is nonempty. Without loss of generality (up to relabeling the coordinates of $W^{(s-1)}, W^{(s)}$), we assume that
\begin{equation*}
I^{(s-1)}=\{1,\cdots,a\},\qquad J^{(s)}=\{1,\cdots,b\} 
\end{equation*}
for some $1\le a\le {n_{s-1}}$ and $1\le b\le n_{s}$. Later we will consider the sum over $w_1^{(s)}\in\roots_{\hat z_s}$ so it is convenient to introduce the notation $\hat W^{(s)} = (w_2^{(s)},\cdots, w_{n_s}^{(s)})$, and more generally $\hat W_U^{(s)}=W_{U\setminus\{1\}}^{(s)}$ the vector obtained by removing $w_1^{(s)}$, if it appears, from $W_U^{(s)}$ for any set $U\subseteq\{1,\cdots,n_s\}$. Thus
\begin{equation*}
\hat W_{I^{(s)}}^{(s)} = W_{I^{(s)}\setminus\{1\}}^{(s)}, \quad \hat W_{J^{(s)}}^{(s)}=W_{J^{(s)}\setminus\{1\}}^{(s)}=(w_2^{(s)},\cdots,w_b^{(s)}).
\end{equation*}

By moving all the factors involving $w_1^{(s)}$ out from the Cauchy-type product, and using the assumption of $s$ that it is the largest index satisfying $I^{(s-1)}\times J^{(s)}\ne\emptyset$, we have
\begin{equation}
\label{eq:ch_reformulation}
\begin{split}
&\prod_{\ell=1}^{m-1} \Ch\left(W^{(\ell)}_{I^{(\ell)}}; W^{(\ell+1)}_{J^{(\ell+1)}}\right)
=\frac{h(w_1^{(s)})\cdot \prod_{j=2}^b (w_j^{(s)} -w_1^{(s)})}{\prod_{i\in I^{(s-1)}} (w_i^{(s-1)} -w_{1}^{(s)})}\\
&\qquad\cdot \Ch\left(W^{(1)}_{I^{(1)}}; W^{(2)}_{J^{(2)}} \right) \cdots \Ch\left(W^{(s-1)}_{I^{(s-1)}}; \hat W^{(s)}_{J^{(s)}} \right)\Ch\left(\hat W^{(s)}_{I^{(s)}}; W^{(s+1)}_{J^{(s+1)}} \right)\cdots \Ch\left(W^{(m-1)}_{I^{(m-1)}}; W^{(m)}_{J^{(m)}} \right),
\end{split}
\end{equation}
where $h$ is a polynomial defined by
\begin{equation*}
h(w_1^{(s)})=\begin{dcases}
\prod_{i\in I^{(s)}\setminus\{1\}} (w_i^{(s)}-w_1^{(s)}), &\text{ if } 1\in I^{(s)},\\
1,&\text{if $1$ is not in }I^{(s)}.
\end{dcases}
\end{equation*}
We remark that there is no denominator factor coming from $\Ch\left(W^{(s)}_{I^{(s)}};  W^{(s+1)}_{J^{(s+1)}} \right)$ since $I^{(s)}\times J^{(s+1)}=\emptyset$ by our choice of $s$. This implies $\Ch\left(W^{(s)}_{I^{(s)}}; W^{(s+1)}_{J^{(s+1)}} \right)=h(w_1^{(s)})\cdot \Ch\left(\hat W^{(s)}_{I^{(s)}};  W^{(s+1)}_{J^{(s+1)}} \right)$ and further~\eqref{eq:ch_reformulation}. We also remark that~\eqref{eq:ch_reformulation} does not contain any pole of $w_1^{(s)}$ within the contour $\Gamma_{|\hat z_s|}=\{w: |q(w)|=|\hat z_s|\}$ since all the points $w_i^{(s-1)}$ are outside this contour by the assumption that $|\hat z_{s-1}|>|\hat z_s|$.

These notations above and formula~\eqref{eq:ch_reformulation} will be used later in this section.

\subsubsection{Reformulating $\GG$}

We first need to reformulate $\GG$ such that the resulting formula is suitable for induction hypothesis. This could be done by evaluating the summation of $w_{1}^{(s)} \in \roots_{\hat z_s}$. Recall that
\begin{equation*}
\GG(z_0,\cdots,z_{m-1})=\sum_{\substack{w_{i_\ell}^{(\ell)} \in \roots_{\hat z_\ell}\\
		1\le i_\ell \le n_\ell\\
		 1\le \ell \le m   
}} \left[\prod_{\substack{1\le i_\ell \le n_\ell\\ 1\le \ell \le m
	  }}J(w_{i_\ell}^{(\ell)})\right]\cdot  H(W^{(1)},\cdots,W^{(m)};z_0,\cdots,z_{m-1})
\end{equation*}
with
\begin{equation*}
H(W^{(1)},\cdots,W^{(m)};z_0,\cdots,z_{m-1})=\left[\prod_{\ell=1}^{m-1} \Ch\left(W^{(\ell)}_{I^{(\ell)}}; W^{(\ell+1)}_{J^{(\ell+1)}}\right)\right]\cdot   A(W^{(1)},\cdots,W^{(m)};z_0,\cdots,z_{m-1})
\end{equation*}
for some function $A$ which is analytic for each $w_{i_\ell}^{(\ell)}\in\Omega\setminus\{0\}$ and $(z_0,\cdots,z_{m-1})\in\disk(r_\mm)\times\disk^{m-1}$. This assumption, together with the fact that $\prod_{\ell=1}^{m-1}\Ch\left(W^{(\ell)}_{I^{(\ell)}}; W^{(\ell+1)}_{J^{(\ell+1)}}\right)$ does not have any pole for $w_1^{(s)}$ inside $\Gamma_{|\hat z_s|}$, imply that $H(W^{(1)},\cdots,W^{(m)};z_0,\cdots,z_{m-1})$ is analytic for $w_1^{(s)}$ inside the contour $\Gamma_{|\hat z_s|}$ except for the point $0$.

By applying the formula~\eqref{eq:roots_sum_single} for $w_{1}^{(s)} \in \roots_{\hat z_s}$, we have
\begin{equation*}
\begin{split}
&\sum_{w_{1}^{(s)}\in\roots_{\hat z_s}}J(w_1^{(s)})H\left(W^{(1)},\cdots,W^{(m)};z_0,\cdots,z_{m-1}\right)\\
&=\int_{\Gamma_{|\hat z_s|+\epsilon}} \frac{q(w_{1}^{(s)})}{q(w_{1}^{(s)})-\hat z_{s}} H\left(W^{(1)},\cdots,W^{(m)};z_0,\cdots,z_{m-1}\right)\ddbarr{w_{1}^{(s)}}\\
&+\int_{\Gamma_{|\hat z_s|-\epsilon}} \frac{-q(w_{1}^{(s)})}{q(w_{1}^{(s)})-\hat z_{s}} H\left(W^{(1)},\cdots,W^{(m)};z_0,\cdots,z_{m-1}\right)\ddbarr{w_{1}^{(s)}}
\end{split}
\end{equation*}
for some sufficiently small $\epsilon>0$. By the discussions above, we could deform the second contour sufficiently close to $0$. Note that we assume $q(w_1^{(s)})$ dominates $H$ at $w_1^{(s)}=0$ in the proposition setting. Therefore the second contour integral vanishes and only the first one survives. We could further deform the first contour to be sufficiently close to $\Gamma_{r_\mm}$. Such a contour deformation gives the residues of $w_1^{(s)}=w_{i}^{(s-1)}$ for $i=1,\cdots,a$. Therefore we have
\begin{equation}
\label{eq:H_split}
\begin{split}
&\sum_{w_{1}^{(s)}\in\roots_{\hat z_s}}J(w_1^{(s)})H\left(W^{(1)},\cdots,W^{(m)};z_0,\cdots,z_{m-1}\right)\\
&= H_1\left(W^{(1)},\cdots,\hat W^{(s)},\cdots,W^{(m)};z_0,\cdots,z_{m-1}\right)\\
&\quad+\sum_{k=1}^aH_{2,k}\left(W^{(1)},\cdots,\hat W^{(s)},\cdots,W^{(m)};z_0,\cdots,z_{m-1}\right),
\end{split}
\end{equation}
where
\begin{equation}
\label{eq:hatH1}
\begin{split}
& H_1 \left(W^{(1)},\cdots,\hat W^{(s)},\cdots,W^{(m)};z_0,\cdots,z_{m-1}\right)\\
&\quad =\int_{\Gamma_{r_\mm-0_+}} \frac{q(w_{1}^{(s)})}{q(w_{1}^{(s)})-\hat z_{s}} H\left(W^{(1)},\cdots,W^{(m)};z_0,\cdots,z_{m-1}\right)\ddbarr{w_{1}^{(s)}}
\end{split}
\end{equation}
with  $r_{\mm-0+}$ denotes a number sufficiently close to $r_\mm$ from below, and
\begin{equation*}
\begin{split}
&H_{2,k}\left(W^{(1)},\cdots,\hat W^{(s)},\cdots,W^{(m)};z_0,\cdots,z_{m-1}\right)\\
&=\mathrm{Res}\left(\frac{-q(w_{1}^{(s)})}{q(w_{1}^{(s)})-\hat z_{s}} H\left(W^{(1)},\cdots,W^{(m)};z_0,\cdots,z_{m-1}\right),w_{1}^{(s)}=w_k^{(s-1)}\right)
\end{split}
\end{equation*}
for $1\le k\le a$. Recall that the notation $\hat W^{(s)} = (w_2^{(s)},\cdots,w_{a}^{(s)})$.

We could further evaluate $H_{2,k}$ more explicitly by using the formula~\eqref{eq:ch_reformulation} and write
\begin{equation}
\label{eq:H2k}
\begin{split}
&H_{2,k}\left(W^{(1)},\cdots,\hat W^{(s)},\cdots,W^{(m)};z_0,\cdots,z_{m-1}\right)\\
&=\Ch\left(W^{(1)}_{I^{(1)}}; W^{(2)}_{J^{(2)}} \right) \cdots \Ch\left(W^{(s-1)}_{I^{(s-1)}\setminus\{k\}}; \hat W^{(s)}_{J^{(s)}} \right)\Ch\left(\hat W^{(s)}_{I^{(s)}}; W^{(s+1)}_{J^{(s+1)}} \right)\cdots \Ch\left(W^{(m-1)}_{I^{(m-1)}}; W^{(m)}_{J^{(m)}} \right)\\
&\qquad \cdot (-1)^{k+b}\cdot\frac{1}{1-z_{s-1}} \cdot h\left(w_k^{(s-1)}\right)\cdot \left.A(W^{(1)},\cdots,W^{(m)};z_0,\cdots,z_{m-1})\right|_{w_{1}^{(s)}=w_k^{(s-1)}}.
\end{split}
\end{equation}
Here the factor $(-1)^{b+k}$ comes from evaluating
\begin{equation*}
\left.\frac{\prod_{j=2}^b(w_j^{(s)}-w_1^{(s)})\cdot\Delta\left(W_{I^{(s-1)}}^{(s-1)}\right)}{\prod_{i\ne k}\left(w_i^{(s-1)}-w_1^{(s)}\right)\prod_{j=2}^b\left(w_k^{(s-1)}-w_j^{(s)}\right)\Delta\left(W_{I^{(s-1)}\setminus\{k\}}^{(s-1)}\right)}\right|_{w_{1}^{(s)}=w_k^{(s-1)}},
\end{equation*}
and $\frac{1}{1-z_{s-1}}$ comes from
\begin{equation*}
\frac{1}{1-z_{s-1}}=\frac{\hat z_{s-1}}{\hat z_{s-1}-\hat z_s}= \left. \frac{q(w_{1}^{(s)})}{q(w_{1}^{(s)})-\hat z_{s}}\right|_{w_{1}^{(s)}=w_k^{(s-1)}}.
\end{equation*}

\vspace{0.3cm}
Now we insert the formula~\eqref{eq:H_split} to the definition of $\GG$ and write
\begin{equation*}
\GG(z_0,\cdots,z_{m-1})=\GG_1(z_0,\cdots,z_m)+\sum_{k=1}^a\GG_{2,k}(z_0,\cdots,z_{m-1})
\end{equation*}
with
\begin{equation*}
%\label{eq:GG2}
\begin{split}
&\GG_1(z_0,\cdots,z_{m-1})\\
&=\sum_{\substack{w_{i_\ell}^{(\ell)} \in \roots_{\hat z_\ell}\\
		1\le i_\ell \le n_\ell, 1\le \ell \le m\\
		(i_\ell,\ell)\ne (1,s)      
}} \left[\prod_{\substack{1\le i_\ell \le n_\ell, 1\le \ell \le m\\
		(i_\ell,\ell)\ne (1,s)  }}J(w_{i_\ell}^{(\ell)})\right]\cdot  H_1(W^{(1)},\cdots,\hat W^{(s)},\cdots,W^{(m)};z_0,\cdots,z_{m-1})
\end{split}
\end{equation*}
and
\begin{equation*}
\begin{split}
&\GG_{2,k}(z_0,\cdots,z_{m-1})\\
&=\sum_{\substack{w_{i_\ell}^{(\ell)} \in \roots_{\hat z_\ell}\\
		1\le i_\ell \le n_\ell, 1\le \ell \le m\\
		(i_\ell,\ell)\ne (1,s)      
}} \left[\prod_{\substack{1\le i_\ell \le n_\ell, 1\le \ell \le m\\
		(i_\ell,\ell)\ne (1,s)  }}J(w_{i_\ell}^{(\ell)})\right]\cdot  H_{2,k}(W^{(1)},\cdots,\hat W^{(s)},\cdots,W^{(m)};z_0,\cdots,z_{m-1})
\end{split}
\end{equation*}
for $1\le k\le a$. We will show that both $\GG_1$ and $\GG_{2,k}$ are both suitable for induction hypothesis. We will verify these in Sections~\ref{sec:analyzing_G1} and~\ref{sec:analyzing_G2}.

\subsubsection{Analyzing $\GG_1$ by using induction hypothesis}
\label{sec:analyzing_G1}

 Now we claim that $\GG_1$ is suitable for induction hypothesis. We need to check all the assumptions of Proposition~\ref{prop:Cauchy_sum_over_roots} with different settings and smaller $\sum_{\ell=1}^{m-1} |I^{(\ell)}\times J^{(\ell+1)}|$. We consider the following modification of the settings in Proposition~\ref{prop:Cauchy_sum_over_roots}:
\begin{enumerate}[(1)]
	\item  $\Omega\to\tilde\Omega:=\{w\in\Omega: |q(w)|<r_\mm\}$,
	\item $n_s\to n_s-1$,
	\item $W^{(s)}\to \hat W^{(s)}=\{w_2^{(s)},\cdots,w_{n_s}^{(s)}\}$,
	\item $I^{(s)} \to I^{(s)}\setminus \{1\}$, $J^{(s)}\to J^{(s)}\setminus\{1\}$,
	\item $A(W^{(1)},\cdots,W^{(m)};z_0,\cdots,z_{m-1}) \to A_1(W^{(1)},\cdots,\hat W^{(s)},\cdots,W^{(m)};z_0,\cdots,z_{m-1})$,
	\item $H(W^{(1)},\cdots,W^{(m)};z_0,\cdots,z_{m-1}) \to H_1(W^{(1)},\cdots,\hat W^{(s)},\cdots,W^{(m)};z_0,\cdots,z_{m-1})$,
\end{enumerate}
where
\begin{equation*}
\begin{split}
&A_1(W^{(1)},\cdots,\hat W^{(s)},\cdots,W^{(m)};z_0,\cdots,z_{m-1})\\
&:=\int_{\Gamma_{r_\mm -0_+}} \frac{q(w_{1}^{(s)})}{q(w_{1}^{(s)})-\hat z_{s}}\cdot \frac{h(w_1^{(s)})\cdot \prod_{j=2}^b (w_j^{(s)} -w_1^{(s)})}{\prod_{i\in I^{(s-1)}} (w_i^{(s-1)}-w_1^{(s)} )} \cdot A(W^{(1)},\cdots,W^{(m)};z_0,\cdots,z_{m-1}) \ddbarr{w_1^{(s)}}.
\end{split}
\end{equation*}
Note that by using the formulas~\eqref{eq:ch_reformulation},~\eqref{eq:hatH1} and the definition of $H$ function, we have
\begin{equation*}
\begin{split}
H_1 
&=\Ch\left(W^{(1)}_{I^{(1)}}; W^{(2)}_{J^{(2)}} \right) \cdots \Ch\left(W^{(s-1)}_{I^{(s-1)}}; \hat W^{(s)}_{J^{(s)}} \right)\Ch\left(\hat W^{(s)}_{I^{(s)}}; W^{(s+1)}_{J^{(s+1)}} \right)\cdots \Ch\left(W^{(m-1)}_{I^{(m-1)}}; W^{(m)}_{J^{(m)}} \right)\\
&\quad \cdot A_1(W^{(1)},\cdots,\hat W^{(s)},\cdots,W^{(m)};z_0,\cdots,z_{m-1}).
\end{split}
\end{equation*}
Thus $H_1$ has the same form of~\eqref{eq:H}. Considering the facts that $|\hat z_s|=|z_0\cdots z_{s-1}|<r_\mm$ and that $\hat z_\ell$ depends on $z_0,\cdots,z_{m-1}$ analytically, and using the assumption that $A$ is analytic for each $z_i$, we know that both $A_1$ and $ H_1$ are analytic for all $(z_0,\cdots,z_m)\in\disk(r_\mm)\times \disk^{m-1}$ by their formulas above. Moreover, by using the assumption that $A$ is analytic for each $w_{i_\ell}^{(\ell)}\in\Omega_0$, we know  $A_1$ is also analytic for each $w_{i_\ell}^{(\ell)}\in\tilde\Omega_0:=\tilde\Omega\setminus\{0\}$.

Moreover, we still have $q(w)$ dominates $H_1$ at $w=0$ by using the facts that any Cauchy chain in $H_1$ is a Cauchy chain in $H$ and that $A_1$ has the same singularities as in $A$ for any coordinates $w_{i_\ell}^{(\ell)}$  within $\tilde\Omega$. Here $(i_\ell,\ell)\ne (1,s)$.

Finally, since we reduced $|I^{(s)}\times J^{(s+1)}|$ by $|J^{(s+1)}|=b\ge 1$, we could apply the induction hypothesis on the above new setting.

\vspace{0.2cm}
By applying the induction hypothesis, we know that $\GG_1$ is analytic for $(z_0,\cdots,z_m)\in\disk(r_\mm)\times \disk^{m-1}$. Moreover, we have
\begin{equation*}
\begin{split}
&G_1(0,z_1,\cdots,z_{m-1})\\
&=	\prod_{\substack{1\le i_\ell\le n_\ell\\ 2\le \ell\le m\\ (i_\ell,\ell)\ne (1,s)}} \left[\frac{1}{1-z_{\ell-1}}\int_{\Sigma_{\ell}^\inn} \ddbarr{w_{i_\ell}^{(\ell)}}
-\frac{z_{\ell-1}}{1-z_{\ell-1}}\int_{\Sigma_{\ell}^\out} \ddbarr{w_{i_\ell}^{(\ell)}} \right]
\cdot \prod_{{i_1}=1}^{n_1} \int_{\Sigma_1} \ddbarr{w_{i_1}^{(1)}}\\
&\qquad H_1(W^{(1)},\cdots,\hat W^{(s)},\cdots,W^{(m)};0,z_1,\cdots,z_{m-1})\\
&=	\prod_{\substack{1\le i_\ell\le n_\ell\\ 2\le \ell\le m\\ (i_\ell,\ell)\ne (1,s)}} \left[\frac{1}{1-z_{\ell-1}}\int_{\Sigma_{\ell}^\inn} \ddbarr{w_{i_\ell}^{(\ell)}}
-\frac{z_{\ell-1}}{1-z_{\ell-1}}\int_{\Sigma_{\ell}^\out} \ddbarr{w_{i_\ell}^{(\ell)}} \right]
\cdot \prod_{{i_1}=1}^{n_1} \int_{\Sigma_1} \ddbarr{w_{i_1}^{(1)}}\\
&\qquad \int_{\Gamma_{r_\mm-0_+}} \ddbarr{w_1^{(s)}}H\left(W^{(1)},\cdots,W^{(m)};0,z_1,\cdots,z_{m-1}\right).
\end{split}
\end{equation*}
Here we remark that the above contours of integration are restricted in $\tilde\Omega$ since we applied the induction hypothesis for $\tilde\Omega$. Now we deform the contour of $w_1^{(s)}$ to the contour $\Sigma_{s}^{\out}$ (such deformation will not pass any poles of $w_1^{(s)}$ since the outermost poles of $w_1^{(s)}$ are on the contours $\Sigma_{s-1}^{\out}\cup\Sigma_{s-1}^{\inn}$ which is inside $\Sigma_{s}^{\out}$), we obtain
\begin{equation}
\label{eq:G1}
\begin{split}
&G_1(0,z_1,\cdots,z_{m-1})\\
&=	\prod_{\substack{1\le i_\ell\le n_\ell\\ 2\le \ell\le m\\ (i_\ell,\ell)\ne (1,s)}} \left[\frac{1}{1-z_{\ell-1}}\int_{\Sigma_{\ell}^\inn} \ddbarr{w_{i_\ell}^{(\ell)}}
-\frac{z_{\ell-1}}{1-z_{\ell-1}}\int_{\Sigma_{\ell}^\out} \ddbarr{w_{i_\ell}^{(\ell)}} \right]
\cdot \prod_{{i_1}=1}^{n_1} \int_{\Sigma_1} \ddbarr{w_{i_1}^{(1)}}\cdot \int_{\Sigma_{s}^\out} \ddbarr{w_1^{(s)}}\\
&\qquad H\left(W^{(1)},\cdots,W^{(m)};0,z_1,\cdots,z_{m-1}\right).
\end{split}
\end{equation}

\subsubsection{Analyzing $\GG_{2,k}$ by using induction hypothesis}
\label{sec:analyzing_G2}

We claim that $G_{2,k}$ is suitable for induction hypothesis. Similar to the case of $\GG_1$, we need to make a few modifications in Proposition~\ref{prop:Cauchy_sum_over_roots}. These changes are:
\begin{enumerate}[(1)]
	\item $n_s\to n_s-1$,
	\item $W^{(s)}\to \hat W^{(s)}=\{w_2^{(s)},\cdots,w_{n_s}^{(s)}\}$,
	\item $I^{(s)} \to I^{(s)}\setminus \{1\}$, $J^{(s)}\to J^{(s)}\setminus\{1\}$,
	\item $I^{(s-1)} \to I^{(s-1)}\setminus\{k\}$,
	\item $A(W^{(1)},\cdots,W^{(m)};z_0,\cdots,z_{m-1}) \to A_{2,k}(W^{(1)},\cdots,\hat W^{(s)},\cdots,W^{(m)};z_0,\cdots,z_{m-1})$,
	\item $H(W^{(1)},\cdots,W^{(m)};z_0,\cdots,z_{m-1}) \to H_{2,k}(W^{(1)},\cdots,\hat W^{(s)},\cdots,W^{(m)};z_0,\cdots,z_{m-1})$,
\end{enumerate}
where
\begin{equation*}
\begin{split}
&A_{2,k}(W^{(1)},\cdots,\hat W^{(s)},\cdots,W^{(m)};z_0,\cdots,z_{m-1}) \\
&= (-1)^{k+b}\cdot\frac{1}{1-z_{s-1}}\cdot h\left(w_k^{(s-1)}\right)\cdot \left.A(W^{(1)},\cdots,W^{(m)};z_0,\cdots,z_{m-1})\right|_{w_{1}^{(s)}=w_k^{(s-1)}}.
\end{split}
\end{equation*}
All the other assumptions in Proposition~\ref{prop:Cauchy_sum_over_roots} with the above setting are easy to check, except the assumption that $q(w)$ dominates $H_{2,k}$ at $w=0$, which we verify below.

Consider any Cauchy chain $w_{i_\ell}^{(\ell)},w_{i_{\ell+1}}^{(\ell+1)},\cdots,w_{i_{\ell'}}^{(\ell')}$ in the above setting. We need to verify that 
\begin{equation}
\label{eq:aux01}
q(w)\cdot \left. A_{2,k}(W^{(1)},\cdots,\hat W^{(s)},\cdots,W^{(m)};z_0,\cdots,z_{m-1}) \right|_{w_{i_\ell}^{(\ell)}=w_{i_{\ell+1}}^{(\ell+1)}=\cdots=w_{i_{\ell'}}^{(\ell')}=w}
\end{equation}
is analytic at $w=0$, for any fixed other $w$-variables in $\Omega_0$, and fixed $(z_0,\cdots,z_{m-1})\in \disk(r_\mm)\times\disk^{m-1}$. If $w_{k}^{(s-1)}$ does not appear in this Cauchy chain, then the analyticity of~\eqref{eq:aux01} follows from the fact that
$$
q(w)\cdot \left. A (W^{(1)},\cdots, W^{(m)};z_0,\cdots,z_{m-1}) \right|_{w_{i_\ell}^{(\ell)}=w_{i_{\ell+1}}^{(\ell+1)}=\cdots=w_{i_{\ell'}}^{(\ell')}=w}
$$
is analytic at $w=0$ by the proposition assumption. If $w_{k}^{(s-1)}$ appears in this Cauchy chain, it must be the last variable in the path since $k$ does not appear in $I^{(s-1)}\setminus\{k\}$. Then
\begin{equation}
\label{eq:aux02}
\eqref{eq:aux01} = (-1)^{k+b}\cdot\frac{1}{1-z_{s-1}}\cdot h_1\left(w\right) \cdot q(w)\cdot \left.A(W^{(1)},\cdots,W^{(m)};z_0,\cdots,z_{m-1})\right|_{w_{i_\ell}^{(\ell)}= \cdots =w_k^{(s-1)}=w_{1}^{(s)}=w}.
\end{equation}
On the other hand, $(k,1)\in I^{(s-1)}\times J^{(s)}$. Thus $w_{i_\ell}^{(\ell)}, \cdots, w_k^{(s-1)}, w_{1}^{(s)}$ is a Cauchy chain in the original proposition setting. By the assumption of the proposition,~\eqref{eq:aux02} is analytic at $w=0$ since $q(w)$ dominates $A$ at $w=0$. This finishes the verification of the analyticity of~\eqref{eq:aux01}.

\vspace{0.2cm}

We also note that  $|I^{(s-1)}\times J^{(s)}|$ after modification becomes $|(I^{(s-1)}\setminus\{k\})\times (J^{(s)}\setminus \{1\})|$ which is smaller. Thus we could apply the induction hypothesis for each $k$. These imply  that $G_{2,k}(z_0,\cdots,z_{m-1})$ is analytic for $(z_0,\cdots,z_{m-1})\in\disk(r_\mm)\times \disk^{m-1}$ and
\begin{equation}
\label{eq:G2k}
\begin{split}
&G_{2,k}(0,z_1,\cdots,z_{m-1})\\
&=	\prod_{\substack{1\le i_\ell\le n_\ell\\ 2\le \ell\le m\\ (i_\ell,\ell)\ne (1,s)}} \left[\frac{1}{1-z_{\ell-1}}\int_{\Sigma_{\ell}^\inn} \ddbarr{w_{i_\ell}^{(\ell)}}
-\frac{z_{\ell-1}}{1-z_{\ell-1}}\int_{\Sigma_{\ell}^\out} \ddbarr{w_{i_\ell}^{(\ell)}} \right]
\cdot \prod_{{i_1}=1}^{n_1} \int_{\Sigma_1} \ddbarr{w_{i_1}^{(1)}}\\
&\qquad H_{2,k}(W^{(1)},\cdots,\hat W^{(s)},\cdots,W^{(m)};0,z_1,\cdots,z_{m-1}).
\end{split}
\end{equation}
Thus their sum $\sum_k\GG_{2,k}$ is also analytic for  $(z_0,\cdots,z_{m-1})\in\disk(r_\mm)\times \disk^{m-1}$. Moreover, by inserting the formula~\eqref{eq:H2k}, it is direct to show
\begin{equation*}
\begin{split}
&H_{2,k}(W^{(1)},\cdots,\hat W^{(s)},\cdots,W^{(m)};0,z_1,\cdots,z_{m-1})\\
&=-\frac{1}{1-z_{s-1}}\mathrm{Res}\left( H\left(W^{(1)},\cdots,W^{(m)};0,z_1\cdots,z_{m-1}\right),w_{1}^{(s)}=w_k^{(s-1)}\right).
\end{split}
\end{equation*}
Note that
\begin{equation*}
\begin{split}
&-\frac{1}{1-z_{s-1}}\sum_{k=1}^a\mathrm{Res}\left( H\left(W^{(1)},\cdots,W^{(m)};0,z_1\cdots,z_{m-1}\right),w_{1}^{(s)}=w_k^{(s-1)}\right)\\
&=\left[\frac{1}{1-z_{s-1}}\int_{\Sigma_{s}^{\inn}}\ddbarr{w_1^{(s)}}-\frac{1}{1-z_{s-1}}\int_{\Sigma_{s}^{\out}}\ddbarr{w_1^{(s)}}\right]H\left(W^{(1)},\cdots,W^{(m)};0,z_1\cdots,z_{m-1}\right)
\end{split}
\end{equation*}
provided all $w_i^{(s-1)}$ variables are on the contours $\Sigma_{s-1}^\inn\cup\Sigma_{s-1}^\out$ since these two contours lie between $\Sigma_{s}^\out$ and $\Sigma_s^\inn$.
By plugging the above calculations in the formula~\eqref{eq:G2k}, we have
\begin{equation}
\label{eq:G2}
\begin{split}
&\sum_{k=1}^a\GG_{2,k}(0,z_1,\cdots,z_{m-1})\\
&=\prod_{\substack{1\le i_\ell\le n_\ell\\ 2\le \ell\le m\\ (i_\ell,\ell)\ne (1,s)}} \left[\frac{1}{1-z_{\ell-1}}\int_{\Sigma_{\ell}^\inn} \ddbarr{w_{i_\ell}^{(\ell)}}
-\frac{z_{\ell-1}}{1-z_{\ell-1}}\int_{\Sigma_{\ell}^\out} \ddbarr{w_{i_\ell}^{(\ell)}} \right]
\cdot \prod_{{i_1}=1}^{n_1} \int_{\Sigma_1} \ddbarr{w_{i_1}^{(1)}}\\
&\quad\left[\frac{1}{1-z_{s-1}}\int_{\Sigma_{s}^{\inn}}\ddbarr{w_1^{(s)}}-\frac{1}{1-z_{s-1}}\int_{\Sigma_{s}^{\out}}\ddbarr{w_1^{(s)}}\right]H\left(W^{(1)},\cdots,W^{(m)};0,z_1\cdots,z_{m-1}\right).
\end{split}
\end{equation}

\subsubsection{Finishing the inductive step}

Now we combine the results in Sections~\ref{sec:analyzing_G1} and~\ref{sec:analyzing_G2}. We know that $\GG(z_0,\cdots,z_{m-1})=\GG_1(z_0,\cdots,z_{m-1})+\sum_{k=1}^a\GG_{2,k}(z_0,\cdots,z_{m-1})$ is analytic for  $(z_0,\cdots,z_{m-1})\in\disk(r_\mm)\times \disk^{m-1}$. Moreover, by the formulas~\eqref{eq:G1} and~\eqref{eq:G2} we have $\GG(0,z_1,\cdots,z_{m-1})$ equals to
\begin{equation*}
\prod_{\substack{1\le i_\ell\le n_\ell\\ 2\le \ell\le m}} \left[\frac{1}{1-z_{\ell-1}}\int_{\Sigma_{\ell}^\inn} \ddbarr{w_{i_\ell}^{(\ell)}}
-\frac{z_{\ell-1}}{1-z_{\ell-1}}\int_{\Sigma_{\ell}^\out} \ddbarr{w_{i_\ell}^{(\ell)}} \right]
\cdot \prod_{{i_1}=1}^{n_1} \int_{\Sigma_1} \ddbarr{w_{i_1}^{(1)}}H\left(W^{(1)},\cdots,W^{(m)};0,z_1\cdots,z_{m-1}\right)
\end{equation*}
for any nested simple closed contours $\Sigma_{m}^{\out},\cdots,\Sigma_{2}^{\out},\Sigma_1,\Sigma_{2}^\inn,\cdots,\Sigma_{m}^\inn$  in $\tilde\Omega$ enclosing $0$. By using the analyticity of $H$ for $w_{i_\ell}^{(\ell)}$ in $\Omega_0$, we could deform these contours freely to $\Omega_0$ without changing their orders. This finishes the induction.

\section{Proof of Theorems~\ref{thm:limit_step} and~\ref{thm:limit_flat}}
\label{sec:asymptotics}

We first translate the height function of TASEP into the language of particle locations. It is known that they have the following equivalence relation\footnote{There is a freedom to decide the particle or empty site corresponding to $H(0,0)$, hence the equivalence relation may have different formulations upon a translation. More explicitly, for any fixed integer $C$ and $C'$, we could formulate the equivalence relation as $H(n,T)\ge a \Longleftrightarrow \mathbbm{x}_{\frac{a-n}{2}+C}(T)\ge n+C'$ by simply translating all the particle locations by $C'$ and their labels by $C$ from the beginning, as long as the initial height function matches the particle locations $H(n,0)\ge a \Longleftrightarrow \mathbbm{x}_{\frac{a-n}{2}+C}(0)\ge n+C'$.}
\begin{equation}
\label{eq:bijection}
H(n,T)\ge a\Longleftrightarrow \mathbbm{x}_{\frac{a-n}{2}}(T) \ge n
\end{equation}
for any integers $a$ and $n$ with the same parity, provided the initial height function is defined such that
\begin{equation}
\label{eq:aux09}
H(n,0)\ge a\Longleftrightarrow \mathbbm{x}_{\frac{a-n}{2}}(0) \ge n.
\end{equation}
The proof of this equivalence relation can be found in, for examples, \cite{Baik-Liu16b,Baik-Liu19}. Here in order to avoid confusion we use $\mathbbm{x}_k(t)$, instead of $x_k(t)$, to denote the location of the particle with label $k$ at time $t$.

We first assume $\tau_1<\cdots<\tau_m$. In this case we only prove Theorem~\ref{thm:limit_step}, the proof of Theorem~\ref{thm:limit_flat} is similar. The only difference is that we need to use Proposition~\ref{prop:infinite_flat} for the flat case instead of Theorem~\ref{thm:main1} for the step case.

We consider the step initial condition defined by~\eqref{eq:step_IC}. This corresponds to, by using~\eqref{eq:aux09},
\begin{equation*}
y_i=\mathbbm{x}_i(0)= -i,\qquad i=1,2,\cdots.
\end{equation*}

Note that the desired probability, by using the relation~\eqref{eq:bijection},
\begin{equation*}
\prob_{\mathrm{step}}\left( \bigcap_{\ell=1}^m \left\{ \frac{H\left( 2x_\ell T^{2/3}, 2\tau_\ell T \right) - \tau_\ell T} {-T^{1/3}}  \le h_\ell   \right\} \right)=\prob_{\mathrm{step}}\left(\bigcap_{\ell=1}^m \left\{	\mathbbm{x}_{k_\ell}(t_\ell) \ge a_\ell \right\}\right)
\end{equation*}
with\footnote{To be precise, we need to assume that all the numbers $k_\ell$ and $a_\ell$ are integers or use their integer parts $[k_\ell]$ and $[a_\ell]$ in the argument. However, in the asymptotics an $O(1)$ perturbation on the $a_\ell$ or $k_\ell$ does not change the desired limit. Hence we just use $k_\ell$ and $a_\ell$ with the formula~\eqref{eq:scaling} in the argument without assuming that they are integers.}
\begin{equation}
\label{eq:scaling}
a_\ell = 2x_\ell T^{2/3}, \quad k_\ell = \frac{1}{2}\tau_\ell T - x_\ell T^{2/3} -\frac{1}{2}h_\ell T^{1/3}, \quad t_\ell = 2\tau_\ell T.
\end{equation}

Now we take $N=\max\{k_\ell:\ell=1,\cdots,m\}$. The above probability only depends on the initial locations of the particles with labels less than or equal to $N$. Thus
\begin{equation*}
\prob_{\mathrm{step}}\left(\bigcap_{\ell=1}^m \left\{	\mathbbm{x}_{k_\ell}(t_\ell) \ge a_\ell \right\}\right) = \prob_{Y_{\mathrm{step}}}\left(\bigcap_{\ell=1}^m \left\{	\mathbbm{x}_{k_\ell}(t_\ell) \ge a_\ell \right\}\right)
\end{equation*}
with
\begin{equation*}
Y_{\mathrm{step}}=(y_1,\cdots,y_N)=(-1,-2,\cdots,-N) \in \conf_N.
\end{equation*}

By applying Theorem~\ref{thm:main1}, it is sufficient to show that
\begin{equation}
\label{eq:aux15}
\begin{split}
\lim_{T\to\infty}\oint\cdots\oint \left[\prod_{\ell=1}^{m-1}\frac{1}{1-z_\ell}\right]\det\left(I-\mathcal{K}_1\mathcal{K}_{Y_\mathrm{step}}\right) \ddbar{z_1}\cdots\ddbar{z_{m-1}}\\
=
\oint\cdots\oint \left[\prod_{\ell=1}^{m-1}\frac{1}{1-z_\ell}\right] \det\left(I-\mathrm{K}_1\mathrm{K}_{\mathrm{step}}\right)\ddbar{z_1}\cdots\ddbar{z_{m-1}},
\end{split}
\end{equation}
where we used the Fredholm determinant representation for $\mathcal{D}_{Y_{\mathrm{step}}}(z_1,\cdots,z_{m-1})$ in Section~\ref{sec:special_IC}.

Recall that $f_i(w)$ is defined in terms of $F_i(w)$ in~\eqref{eq:fi}, and by Proposition~\ref{prop:invariance1} the Fredholm determinant $\det\left(I-\mathcal{K}_1\mathcal{K}_{Y_{\mathrm{step}}}\right)$ is unchanged if we replace $F_i(w)$ by
\begin{equation*}
\tilde F_i(w):=\frac{F_i(w)}{F_i(-1/2)}.
\end{equation*}
Hence we could replace $f_i(w)$ by
\begin{equation}
\label{eq:tilde_fi}
\tilde{f}_i(w):=\begin{dcases}
\frac{\tilde F_i(w)}{\tilde F_{i-1}(w)},&w\in\Omega_\LL\setminus\{-1\},\\
\frac{\tilde F_{i-1}(w)}{\tilde F_i(w)},&w\in\Omega_\RR\setminus\{0\}
\end{dcases}
\end{equation}
without changing the Fredholm determinant. Then we apply a conjugation for the kernels and reduce~\eqref{eq:aux15} to a new equation
\begin{equation}
\label{eq:aux10}
\begin{split}
\lim_{T\to\infty}\oint\cdots\oint \left[\prod_{\ell=1}^{m-1}\frac{1}{1-z_\ell}\right]\det\left(I-\tilde{\mathcal{K}}_1\tilde{\mathcal{K}}_{Y_{\mathrm{step}}}\right) \ddbar{z_1}\cdots\ddbar{z_{m-1}}\\
=
\oint\cdots\oint \left[\prod_{\ell=1}^{m-1}\frac{1}{1-z_\ell}\right] \det\left(I-\tilde{\mathrm{K}}_1\tilde{\mathrm{K}}_{\mathrm{step}}\right)\ddbar{z_1}\cdots\ddbar{z_{m-1}},
\end{split}
\end{equation}
where the new kernels
\begin{equation*}
\begin{split}
\tilde{\mathcal{K}}_{Y_{\mathrm{step}}}(w',w)&=\left(\delta_j (i) + \delta_j(i - (-1)^j)\right) \frac{ \sqrt{\tilde f_j(w')}\sqrt{\tilde f_i(w)} }{w'-w} Q_2(i),\\
\tilde{\mathcal{K}}_1(w,w')&=\left(\delta_i(j) + \delta_i( j+ (-1)^i)\right) \frac{ \sqrt{\tilde f_j(w')}\sqrt{\tilde f_i(w)} }{w-w'} Q_1(j),
\end{split}
\end{equation*}
for all $w\in (\Sigma_{i,\LL}\cup\Sigma_{i,\RR})\cap\mathcal{S}_1$ and $w'\in (\Sigma_{j,\LL}\cup\Sigma_{j,\RR})\cap \mathcal{S}_2$, and
\begin{equation*}
\begin{split}
\tilde{\mathrm{K}}_{\mathrm{step}}(\zeta',\zeta)&=\left(\delta_j (i) + \delta_j(i - (-1)^j)\right) \frac{ \sqrt{\mathrm f_j(\zeta')}\sqrt{\mathrm f_i(\zeta)} }{-\zeta'+\zeta} Q_2(i),\\
\tilde{\mathrm{K}}_1(\zeta,\zeta')&=\left(\delta_i(j) + \delta_i( j+ (-1)^i)\right) \frac{ \sqrt{\mathrm f_j(\zeta')}\sqrt{\mathrm f_i(\zeta)} }{\zeta-\zeta'} Q_1(j),
\end{split}
\end{equation*}
for all $\zeta\in (\mathrm{C}_{i,\LL}\cup\mathrm{C}_{i,\RR})\cap\mathrm{S}_1$ and $\zeta'\in (\mathrm{C}_{j,\LL}\cup\mathrm{C}_{j,\RR})\cap \mathrm{S}_2$. The reason we do these conjugations is to ensure the kernels decay sufficiently fast on each variable. We also remark that the choice of the branch cut of the square root does not affect the product of two kernels since each square root term will appear twice when one evaluates the Fredholm determinant.

The proof of~\eqref{eq:aux10} follows from the two lemmas below.
\begin{lm}
	\label{lm:01}
	Assume the scaling~\eqref{eq:scaling}. For each $n$ and fixed $z_1,\cdots,z_{m-1}\in\disk=\{z\in\complexC:|z|<1\}$, we have
	\begin{equation*}
	\lim_{T\to\infty} \mathrm{Tr}\left(\tilde{\mathcal{K}}_1\tilde{\mathcal{K}}_{{Y_{\mathrm{step}}}}\right)^n = \mathrm{Tr}\left(\tilde{\mathrm{K}}_1\tilde{\mathrm{K}}_{\mathrm{step}}\right)^n.
	\end{equation*}
\end{lm}
\begin{lm}
	\label{lm:02} Assume the scaling~\eqref{eq:scaling}. There exists a constant $C$ which does not depend on $T$ and $n$ such that
\begin{equation*}
\left|\int_{\mathcal{S}_1}\dd\mu(w_1)\cdots\int_{\mathcal{S}_1}\dd\mu(w_n)\det\left[\left(\tilde{\mathcal{K}}_1\tilde{\mathcal{K}}_{Y_{\mathrm{step}}}\right)(w_i,w_j) \right]_{i,j=1}^n\right|<C^n.
\end{equation*}
\end{lm}

The proof of both lemmas are standard. Below we just provide the main ideas and necessary calculations, and omit most of the details.

We analyze the function $\tilde f_i(w)$. Recall~\eqref{eq:tilde_fi}, $\tilde f_i(w)$ is defined by $\tilde F_i(w)$ functions with
\begin{equation*}
\tilde F_i(w)=\frac{w^{k_i}(w+1)^{-a_i-k_i} e^{t_i w}}{(-1/2)^{k_i}(1/2)^{-a_i-k_i}e^{-t_i/2}}.
\end{equation*}
By inserting~\eqref{eq:scaling} we have
\begin{equation*}
\begin{split}
&\tilde F_i(w) \\
&= \exp\left( \left(\frac{1}{2}\tau_i T - x_i T^{2/3} -\frac{1}{2}h_i T^{1/3}\right)\log (-2w)-\left(\frac{1}{2}\tau_i T + x_i T^{2/3} -\frac{1}{2}h_i T^{1/3}\right)\log (2w+2) +2\tau_i T(w+1/2)\right).
\end{split}
\end{equation*}
A direct calculation shows that the critical point of $\tilde F_i(w)$ is $w=-\frac12$. Moreover, by using Taylor expansion, we have
\begin{equation*}
\tilde F_i\left(-\frac12+\frac{\zeta}{2T^{1/3}}\right)\approx  \mathrm{F}_i(\zeta)=  \exp\left(-\frac{1}{3}\tau_i \zeta^3 + x_i \zeta^2 + h_i \zeta\right).
\end{equation*}
Here the function $\mathrm{F}_i(\zeta)$ is defined in~\eqref{eq:Fi}.
Now we deform the contours $\Sigma_{m,\LL}^{\out},\cdots,\Sigma_{2,\LL}^{\out}$, $\Sigma_{1,\LL}$, $\Sigma_{2,\LL}^{\inn}$,$\cdots$,$\Sigma_{m,\LL}^\inn$ to be sufficiently close to $-1/2$ (and still enclosing $-1$), such that near the point $-1/2$ after the change of variable  $w=-\frac12+\frac{\zeta}{2T^{1/3}}$ these contours locally converge to $\mathrm{C}_{m,\LL}^{\out}$,  $\cdots$, $\mathrm{C}_{2,\LL}^{\out}$, $\mathrm{C}_{1,\LL}$, $\mathrm{C}_{2,\LL}^{\inn}$, $\cdots$, $\mathrm{C}_{m,\LL}^{\inn}$ respectively. We similarly deform the contours
$\Sigma_{m,\RR}^{\out},\cdots,\Sigma_{2,\RR}^{\out}$, $\Sigma_{1,\RR}$, $\Sigma_{2,\RR}^{\inn},\cdots,\Sigma_{m,\RR}^\inn$ to be sufficiently close to $-1/2$ such that near $-1/2$ they locally converge to $\mathrm{C}_{m,\RR}^{\out}$,   $\cdots$, $\mathrm{C}_{2,\RR}^{\out}$, $\mathrm{C}_{1,\RR}$, $\mathrm{C}_{2,\RR}^{\inn}$, $\cdots$, $\mathrm{C}_{m,\RR}^{\inn}$ respectively. Note that the orientations of $\mathrm{C}_{\ell,\RR}^\star$ contours are reversed compared to $\Sigma_{\ell,\RR}^\star$ contours. Here $\star$ represents the superscript $\out$ or $\inn$ or empty superscript. This reversed orientation will contribute to the different signs between the kernels $\tilde{\mathcal{K}}_{Y_{\mathrm{step}}}$ and $\tilde{\mathrm{K}}_{\mathrm{step}}$.

With the above deformations, it is easy to check that $\tilde f_i(w)\approx \mathrm{f_i}(\zeta)$ for $\zeta\in\mathrm{S}_1\cup\mathrm{S}_2$. Thus locally we have $\tilde{\mathcal{K}}_{Y_{\mathrm{step}}}(w',w)\approx -2T^{1/3} \tilde{\mathrm{K}}_{\mathrm{step}}(\zeta',\zeta)$ and $\tilde{\mathcal{K}}_1(w,w')\approx 2T^{1/3} \tilde{\mathrm{K}}_{1}(\zeta,\zeta')$ for $w=-\frac12+\frac{\zeta}{2T^{1/3}}$ and $w'=-\frac12+\frac{\zeta'}{2T^{1/3}}$. On the other hand, it is direct to see that the kernels $\tilde{\mathcal{K}}_{Y_{\mathrm{step}}}$ and $\tilde{\mathcal{K}}_1$ decay super-exponentially fast when $w,w'$ is away from $-1/2$ along the contours in $\mathcal{S}_1$ and $\mathcal{S}_2$. Hence the main contribution of $\mathrm{Tr}(\tilde{\mathcal{K}}_1\tilde{\mathcal{K}}_{Y_{\mathrm{step}}})^n$ comes from a small neighborhood of $-1/2$. This  heuristically implies $\mathrm{Tr}(\tilde{\mathcal{K}}_1\tilde{\mathcal{K}}_{Y_{\mathrm{step}}})^n\approx \mathrm{Tr}\left(\tilde{\mathrm{K}}_1\tilde{\mathrm{K}}_{\mathrm{step}}\right)^n$. The same argument heuristically implies the boundness $\mathrm{Tr}(\tilde{\mathcal{K}}_1\tilde{\mathcal{K}}_{Y_{\mathrm{step}}})^n$. By using these facts, it is standard to prove both lemmas we list above. This proves Theorem~\ref{thm:limit_step}.

\bigskip

Finally we address the general case for $\tau_1\le\cdots\le \tau_m$ with $x_i<x_{i+1}$ when $\tau_i=\tau_{i+1}$. Note that the limiting field $\lim_{T\to\infty}\frac{H(2xT^{2/3},2\tau T)}{-T^{1/3}}=\mathrm{H}(x,\tau)$, the KPZ fixed point \cite{Matetski-Quastel-Remenik17}, is continuous in both $x$ and $\tau$. On the other hand, we have shown $\mathrm{D}_{\mathrm{step}}$ and $\mathrm{D}_{\mathrm{flat}}$ are continuous on the parameters $h_\ell,x_\ell,\tau_\ell$, $\ell=1,\cdots,m$ in the domain $\tau_\ell\le \tau_{\ell+1}$ or all $\ell$ and $x_\ell<x_{\ell+1}$ for $\ell$ satisfying $\tau_\ell=\tau_{\ell+1}$. Hence $F_{\mathrm{step}}$ and $F_{\mathrm{flat}}$ as we defined in Theorems~\ref{thm:limit_step} and~\ref{thm:limit_flat} are both continuous on these parameters in this domain. By the continuity of the KPZ fixed point, we immediately obtain the limit theorems for the general case.

It is also possible to prove the general case directly. Below we provide the ideas of the proof but ignore the technical details.

For the step case, there is no change in the proof if we use the contours $\Sigma_{\ell,\RR}^{\inn},\Sigma_{\ell,\RR}^{\out}$, $2\le \ell\le m$, and $\Sigma_{1,\RR}$ with the new angles $e^{\pm\pi\ii/5}\infty$ since the functions $\tilde f_i$ still decay super-exponentially fast along these new contours. For the flat case, we need to rewrite the trace or determinant in Lemmas~\ref{lm:01} and~\ref{lm:02} as an expansion of a combination of integrals as in $\mathcal{D}_{\boldsymbol{n},Y_{\mathrm{flat}}}$ in~\eqref{eq:aux_100}. Then we do the same rewriting as in~\eqref{eq:aux_101} and~\eqref{eq:aux_102}. We end with $2^{n_2+\cdots+n_m}$ possible integrals of the structure
\begin{equation}
\label{eq:aux_103}
\sum \prod_{\substack{\text{some }i_\ell\\ \ell\ge 2}}\int_{\Sigma_{\ell,\LL}^\out} \ddbarr{u_{i_\ell}^{(\ell)}}\prod_{{i_1}=1}^{n_{1}} \int_{\Sigma_{1,\LL}} \ddbarr{u_{i_1}^{(1)}}\prod_{\ell=2}^{m}  \prod_{{i_\ell}=1}^{n_{\ell}} \left[\frac{1}{1-z_{\ell-1}}\int_{\Sigma_{\ell,\RR}^\inn} \ddbarr{v_{i_\ell}^{(\ell)}}
-\frac{z_{\ell-1}}{1-z_{\ell-1}}\int_{\Sigma_{\ell,\RR}^\out} \ddbarr{v_{i_\ell}^{(\ell)}} \right]
\cdot \prod_{{i_1}=1}^{n_{1}} \int_{\Sigma_{1,\RR}} \ddbarr{v_{i_1}^{(1)}}.
\end{equation}
Here we ignored the integrand similarly as in~\eqref{eq:aux_102}.
With the assumptions on the contours, we have~\eqref{eq:aux_103} converges to~\eqref{eq:aux_102}.  Adding these combinations gives Lemma~\ref{lm:01}. For Lemma~\ref{lm:02} it follows from the fact that the integrand decays super-exponentially fast along the contours for each possible expression~\eqref{eq:aux_103}. So we have the a bound of $C^n\cdot 2^{n_2+\cdots+n_m}$. On the other hand, the expansion of $\det\left[\left(\tilde{\mathcal{K}}_1\tilde{\mathcal{K}}_{Y_{\mathrm{step}}}\right)(w_i,w_j) \right]_{i,j=1}^n$ only involves terms $\mathcal{D}_{\boldsymbol{n},Y_{\mathrm{flat}}}$ satisfying $n_1+\cdots+n_m=n$. Thus $C^n\cdot 2^{n_2+\cdots+n_m}\le (2C)^n$ and Lemma~\ref{lm:02} follows immediately.

\section{Proof of propositions}
\label{sec:proofs}

Before proving the propositions in Section~\ref{sec:main_results}, we introduce one lemma.
\begin{lm}
		\label{lm:lemma3} 
	Suppose $m\ge 1$ is an integer, and $n_1,\cdots,n_m\ge 0$ are $m$ non-negative integers. For each $1\le \ell\le m$, $W^{(\ell)}=(w_1^{(\ell)},\cdots,w_{n_\ell}^{(\ell)})\in\complexC^{n_\ell}$ is a vector of $n_\ell$ complex variables. Assume $\Omega$ is a simply connected domain in $\complexC$ and $a\in\Omega$ is a point in $\Omega$.  Suppose $F(W^{(1)},\cdots,W^{(m)})$ is a function analytic for each variable $w_{i_\ell}^{(\ell)}\in\Omega\setminus\{a\}$, $1\le i_\ell\le n_\ell, 1\le\ell\le m$. Suppose $t$ and $j_t$ are two fixed numbers such that $1\le t\le m$, $n_t\ge 1$, and $1\le j_t\le n_t$. Assume $F$ satisfies the following analyticity property: For any chain of variables starting or ending at $w_{j_t}^{(t)}$: $w_{j_s}^{(s)},w_{j_{s+1}}^{(s+1)},\cdots,w_{j_{s'}}^{(s')}$ with $t=s\le s'$ or $s\le s'=t$, $j_t$ being fixed but each $j_{\ell}$ ($\ell\ne t$) could be an arbitrary number such that $1\le j_\ell\le n_\ell$, the function
	\begin{equation*}
	\left.F(W^{(1)},\cdots,W^{(m)})\right|_{w_{j_s}^{(s)}=w_{j_{s+1}}^{(s+1)}=\cdots=w_{j_{s'}}^{(s')}=w}
	\end{equation*} 
	is analytic at $w=a$ when all other variables in $\Omega\setminus\{a\}$ are fixed. Then
	\begin{equation*}
	\oint\ddbarr{w_{1}^{(1)}}\cdots\oint\ddbarr{w_{n_m}^{(m)}} \left[\prod_{\ell=1}^{m-1}\Ch(W^{(\ell)};W^{(\ell+1)})\right]\cdot F(W^{(1)},\cdots,W^{(m)})=0,
	\end{equation*}
	where the contours of integration could be any order of nested contours enclosing $a$ in $\Omega$. The function $\Ch$ is the Cauchy-type product defined in~\eqref{eq:def_cauchy}.
\end{lm}
\begin{proof}
	The proof follows from a simple calculation. We first integrate the function $\left[\prod_{\ell=1}^{m-1}\Ch(W^{(\ell)};W^{(\ell+1)})\right]\cdot F(W^{(1)},\cdots,W^{(m)})$ with respect to $w_{j_t}^{(t)}$. Since this integrand as a function of $w_{j_t}^{(t)}$ is analytic at $a$ except for possible poles from the Cauchy-type factors, after this integral only the (possible) residues survive. Now we evaluate the residue at $w_{j_t}^{(t)}=w_{j_{t+1}}^{(t+1)}$ (if the $w_{j_{t+1}}^{(t+1)}$ contour is inside the $w_{j_t}^{(t)}$ contour). By the assumption, if we integrate this residue with respect to $w_{j_{t+1}}^{(t+1)}$, it is zero again except that some residues at $w_{j_{t+1}}^{(t+1)}=w_{j_{t+2}}^{(t+2)}$ may survive\footnote{Here we remind that there are no residues of type $w_{j_{t+1}}^{(t+1)}=w_{j'_t}^{(t)}$ since $w_{j_{t+1}}^{(t+1)}-w_{j'_t}^{(t)}$ does not appear in the Cauchy-type factor after our previous evaluation of residue at $w_{j_t}^{(t)}=w_{j_{t+1}}^{(t+1)}$.}. We repeat this procedure and integrate the residue with respect to the variable $w_{j_{t+2}}^{(t+2)}$. After finitely many steps we stop at some point that the integrand no longer has residues. Thus after this procedure we end at zero. Similarly, the evaluation of the possible residue at $w_{j_t}^{(t)}=w_{j_{t-1}}^{(t-1)}$ gives zero as well. This proves the lemma.
\end{proof}

\subsection{Proof of Proposition~\ref{prop:consistency}}
\label{sec:proof_consistency}
The proof of this proposition depends on Theorem~\ref{thm:main1} and Proposition~\ref{lm:lemma4}.

We prove it by using induction on $|I|$.

When $|I|=0$, it is Theorem~\ref{thm:main1}.

Suppose the statement holds for smaller $|I|$. We consider the case of $|I|\ge 1$. Let $s$ be an element in $I$. It satisfies $1\le s\le m-1$. We consider the following three objects
\begin{equation*}
\begin{split}
P_1&=\prob_Y \left( 
\left(\bigcap_{j\in J} \left\{ x_{k_j} (t_j) \ge a_j 
\right\}
\right)\bigcap
\left(\bigcap_{i\in I\setminus\{s\}} \left\{ x_{k_i} (t_i) < a_i 
\right\}\right)
\right),\\
P_2&=\prob_Y \left( 
\left(\bigcap_{j\in J\cup\{s\}} \left\{ x_{k_j} (t_j) \ge a_j 
\right\}
\right)\bigcap
\left(\bigcap_{\substack{i\in I\setminus\{s\}}} \left\{ x_{k_i} (t_i) < a_i 
\right\}\right)
\right), \\
P_3&=\prob_Y \left( 
\left(\bigcap_{j\in J } \left\{ x_{k_j} (t_j) \ge a_j 
\right\}
\right)\bigcap
\left(\bigcap_{\substack{i\in I}} \left\{ x_{k_i} (t_i) < a_i 
\right\}\right)
\right).
\end{split}
\end{equation*}
Note the event considered in $P_1$ is a union of the two disjoint events considered in $P_2$ and $P_3$. Therefore we have $P_1=P_2+P_3$.

On the other hand, since $|I\setminus\{s\}|<|I|$ we could apply the induction hypothesis to $I$ and $II$. We have
\begin{equation*}
\begin{split}
P_1&=(-1)^{|I|-1}\oint  \ddbar{z_1}\cdots\oint\ddbar{z_{s-1}}\oint\ddbar{z_{s+1}}\cdots\oint\ddbar{z_{m-1}}\\
&\quad\left[\prod_{\substack{1\le \ell\le m-1\\ \ell\ne s}}\frac{1}{1-z_\ell}\right] \mathcal{D}_Y(z_1,\cdots,z_{s-1},z_{s+1},\cdots,z_{m-1})
\end{split}
\end{equation*} 
and
\begin{equation*}
\begin{split}
P_2&=(-1)^{|I|-1}\oint  \ddbar{z_1}\cdots\oint\ddbar{z_{s-1}}\oint\ddbar{z_{s}}\oint\ddbar{z_{s+1}}\cdots\oint\ddbar{z_{m-1}}\\
&\quad\left[\prod_{\substack{1\le \ell\le m-1}}\frac{1}{1-z_\ell}\right] \mathcal{D}_Y(z_1,\cdots,z_{s-1},z_s,z_{s+1},\cdots,z_{m-1}),
\end{split}
\end{equation*} 
where the contours of integration are circles centered at the origin. The radius of $z_i$ contour is larger than $1$ in both $P_1$ and $P_2$ if $i\in I\setminus\{s\}$, otherwise it is smaller than $1$. We remind that in the term $\mathcal{D}_Y(z_1,\cdots,z_{s-1},z_{s+1},\cdots,z_{m-1})$ of $P_1$, the parameters are $a_\ell,k_\ell,t_\ell$ for $1\le \ell\le m$ but $\ell\ne s$. This is also consistent with Proposition~\ref{lm:lemma4}. 

Now we apply Proposition~\ref{lm:lemma4} and obtain
\begin{equation*}
\begin{split}
P_1-P_2=(-1)^{|I|}\oint  \ddbar{z_1}\cdots\oint\ddbar{z_{m-1}}
\left[\prod_{\substack{1\le \ell\le m-1}}\frac{1}{1-z_\ell}\right] \mathcal{D}_Y(z_1,\cdots,z_{m-1}),
\end{split}
\end{equation*} 
where the contours are circles centered at the origin. The radius of $z_i$ is larger than $1$ if $i\in I$, otherwise it is smaller than $1$. This equals to $P_3$ by our argument at the beginning of the proof. This finishes the induction.

\subsection{Proof of Proposition~\ref{prop:invariance_spaces}}
\label{sec:proof_invariance_spaces}

We will prove the proposition by using the following lemma.
\begin{lm}
\label{lm:switching_contours}
Suppose $\Sigma^\out,\Sigma,\Sigma^\inn$ are three nested simple closed contours in $\complexC$. Let $\Omega$ be an open region containing  these three contours and all the points between them. Assume $U^{(1)}=(u_1^{(1)},\cdots,u_{n_1}^{(1)})$ and $U^{(2)}=(u^{(2)}_1,\cdots,u_{n_2}^{(2)})$ are two vectors of variables. Here $n_1,n_2\ge 0$. We also assume that $F(U^{(1)},U^{(2)})$ is an analytic function on $\Omega^{n_1+n_2}$. Then for each $z\ne 1$, we have
\begin{equation}
\label{eq:aux18}
\begin{split}
&\prod_{i_1=1}^{n_1}\left[\frac{1}{1-z }\int_{\Sigma^\out} \ddbarr{u_{i_1}^{(1)}}
-\frac{z}{1-z }\int_{\Sigma ^\inn} \ddbarr{u_{i_1}^{(1)}} \right]
\cdot \prod_{{i_2}=1}^{n_{2}} \int_{\Sigma } \ddbarr{u_{i_2}^{(2)}} \Ch(U^{(1)};U^{(2)})F(U^{(1)},U^{(2)})\\
&=\prod_{i_2=1}^{n_2}\left[\frac{1}{1-z }\int_{ \Sigma^\inn} \ddbarr{u_{i_2}^{(2)}}
-\frac{z}{1-z }\int_{ \Sigma ^\out} \ddbarr{u_{i_2}^{(2)}} \right]
\cdot \prod_{{i_1}=1}^{n_{1}} \int_{ \Sigma } \ddbarr{u_{i_1}^{(1)}} \Ch(U^{(1)};U^{(2)})F(U^{(1)},U^{(2)}),
\end{split}
\end{equation}
where $\Ch(W;W')$ is the Cauchy-type factor defined in~\eqref{eq:def_cauchy}.
\end{lm}

We will first use Lemma~\ref{lm:switching_contours} to prove Proposition~\ref{prop:invariance_spaces}, then prove Lemma~\ref{lm:switching_contours}.

Consider Proposition~\ref{prop:invariance_spaces}. By using the series expansion formula, it is sufficient to show that for any $\boldsymbol{n}=(n_1,\cdots,n_m)\in(\intZ_{\ge 0})^{m}$, we have
\begin{equation}
\label{eq:aux05}
\begin{split}
&\prod_{\ell=2}^{m}  \prod_{{i_\ell}=1}^{n_{\ell}} \left[\frac{1}{1-z_{\ell-1}}\int_{\Sigma_{\ell,\LL}^\inn} \ddbarr{u_{i_\ell}^{(\ell)}}
-\frac{z_{\ell-1}}{1-z_{\ell-1}}\int_{\Sigma_{\ell,\LL}^\out} \ddbarr{u_{i_\ell}^{(\ell)}} \right]
\cdot \prod_{{i_1}=1}^{n_{1}} \int_{\Sigma_{1,\LL}} \ddbarr{u_{i_1}^{(1)}}\\
&\qquad
\left[\prod_{\ell=1}^{m-1}\Ch(U^{(\ell)};U^{(\ell+1)})\right]\cdot F(U^{(1)},\cdots,U^{(m)})\\
&=\prod_{\ell=1}^{m-1}  \prod_{{i_\ell}=1}^{n_{\ell}} \left[\frac{1}{1-z_{\ell}}\int_{\tilde\Sigma_{\ell,\LL}^\out} \ddbarr{u_{i_\ell}^{(\ell)}}
-\frac{z_{\ell}}{1-z_{\ell}}\int_{\tilde\Sigma_{\ell,\LL}^\inn} \ddbarr{u_{i_\ell}^{(\ell)}} \right]
\cdot \prod_{{i_m}=1}^{n_{m}} \int_{\tilde\Sigma_{m,\LL}} \ddbarr{u_{i_m}^{(m)}}\\
&\qquad
\left[\prod_{\ell=1}^{m-1}\Ch(U^{(\ell)};U^{(\ell+1)})\right]\cdot F(U^{(1)},\cdots,U^{(m)}),
\end{split}
\end{equation}
where  $F$ is any function analytic for each variable $u_{i_\ell}^{(\ell)}$ in $\Omega_\LL\setminus\{-1\}$. The vector $U^{(\ell)}=(u_1^{(\ell)},\cdots,u_{n_\ell}^{(\ell)})$ for $\ell=1,\cdots,m$. Recall that $\Sigma_{m,\LL}^\out$,$\cdots$,$\Sigma_{2,\LL}^\out$,$\Sigma_{1,\LL}$,$\Sigma_{2,\inn}$,$\cdots$,$\Sigma_{m,\LL}^\inn$ are nested, and $\tilde \Sigma_{1,\LL}^\out,\cdots,\tilde \Sigma_{m-1,\LL}^\out,\tilde\Sigma_{m,\LL},\tilde\Sigma_{m-1,\LL}^\inn,\cdots,\tilde\Sigma_{1,\LL}^\inn$ are also nested. 

We prove~\eqref{eq:aux05} by induction.

If $m=2$, we need to show that
\begin{equation*}
\begin{split}
&\prod_{i_1=1}^{n_1}\left[\frac{1}{1-z_{1}}\int_{\tilde\Sigma_{1,\LL}^\out} \ddbarr{u_{i_1}^{(1)}}
-\frac{z_{1}}{1-z_{1}}\int_{\tilde\Sigma_{1,\LL}^\inn} \ddbarr{u_{i_1}^{(1)}} \right]
\cdot \prod_{{i_2}=1}^{n_{2}} \int_{\tilde\Sigma_{2,\LL}} \ddbarr{u_{i_2}^{(2)}} \Ch(U^{(1)};U^{(2)})F(U^{(1)},U^{(2)})\\
&=\prod_{i_2=1}^{n_2}\left[\frac{-z_1}{1-z_{1}}\int_{\Sigma_{2,\LL}^\out} \ddbarr{u_{i_2}^{(2)}}
+\frac{1}{1-z_{1}}\int_{\Sigma_{2,\LL}^\inn} \ddbarr{u_{i_2}^{(2)}} \right]
\cdot \prod_{{i_1}=1}^{n_{1}} \int_{\Sigma_{1,\LL}} \ddbarr{u_{i_1}^{(1)}} \Ch(U^{(1)};U^{(2)})F(U^{(1)},U^{(2)}).
\end{split}
\end{equation*}
This follows from Lemma~\ref{lm:switching_contours} by deforming the contours appropriately.

Suppose~\eqref{eq:aux05} holds for $m-1$ with $m\ge 3$. We want to show that it holds for $m$. Without loss of generality, we assume that $\Sigma_{m,\LL}^\out$ is outside of all the contours $\tilde\Sigma_{\ell,\LL}^\out$, and $\Sigma_{m,\LL}^\inn$ is inside all the contours $\tilde\Sigma_{\ell,\LL}^\inn$. We first fix all other contours but just apply Lemma~\ref{lm:switching_contours} case for the variables $U^{(m)}$, $U^{(m-1)}$ and the contours $\tilde\Sigma_{m-1,\LL}^\out,\tilde\Sigma_{m,\LL},\tilde\Sigma_{m-1,\LL}^\inn$. This gives
\begin{equation}
\label{eq:aux12}
\begin{split}
&\prod_{\ell=1}^{m-1} \prod_{{i_\ell}=1}^{n_{\ell}} \left[\frac{1}{1-z_{\ell}}\int_{\tilde\Sigma_{\ell,\LL}^\out} \ddbarr{u_{i_\ell}^{(\ell)}}
-\frac{z_{\ell}}{1-z_{\ell}}\int_{\tilde\Sigma_{\ell,\LL}^\inn} \ddbarr{u_{i_\ell}^{(\ell)}} \right]
\cdot \prod_{{i_m}=1}^{n_{m}} \int_{\tilde\Sigma_{m,\LL}} \ddbarr{u_{i_m}^{(m)}}\\
&\qquad
\left[\prod_{\ell=1}^{m-1}\Ch(U^{(\ell)};U^{(\ell+1)})\right]\cdot F(U^{(1)},\cdots,U^{(m)})\\
&=\prod_{\ell=1}^{m-2}  \prod_{{i_\ell}=1}^{n_{\ell}} \left[\frac{1}{1-z_{\ell}}\int_{\tilde\Sigma_{\ell,\LL}^\out} \ddbarr{u_{i_\ell}^{(\ell)}}
-\frac{z_{\ell}}{1-z_{\ell}}\int_{\tilde\Sigma_{\ell,\LL}^\inn} \ddbarr{u_{i_\ell}^{(\ell)}} \right]
\cdot \prod_{{i_{m-1}}=1}^{n_{m-1}} \int_{\tilde\Sigma_{m,\LL}} \ddbarr{u_{i_{m-1}}^{(m-1)}}\\
&\qquad \prod_{i_m=1}^{n_m}\left[\frac{1}{1-z_{m-1}}\int_{\tilde\Sigma_{m-1,\LL}^\inn}\ddbarr{u_{i_m}^{(m)}}-\frac{z_{m-1}}{1-z_{m-1}}\int_{\tilde\Sigma_{m-1,\LL}^\out}\ddbarr{u_{i_m}^{(m)}}\right]
\left[\prod_{\ell=1}^{m-1}\Ch(U^{(\ell)};U^{(\ell+1)})\right]\cdot F(U^{(1)},\cdots,U^{(m)}).
\end{split}
\end{equation}
Then we deform the $\tilde\Sigma_{m-1,\LL}^\out$ to   $\Sigma_{m,\LL}^\out$ and $\tilde\Sigma_{m-1,\LL}^\inn$ to $\Sigma_{m,\LL}^\inn$. The integrand does not encounter any poles during the deformation since the variables of $U^{(m-1)}$ is on $\tilde\Sigma_{m,\LL}$ which lies between $\tilde\Sigma_{m-1,\LL}^\out$ and $\tilde\Sigma_{m-1,\LL}^\inn$. Then we apply the induction hypothesis for all other variables in $U^{(\ell)}$ for $\ell\ne m$ and all other contours $\tilde\Sigma_{1,\LL}^\out$,$\cdots$,$\tilde\Sigma_{m-2,\LL}^\out$,$\tilde\Sigma_{m,\LL}$,$\tilde\Sigma_{m-2,\LL}^\inn$,$\cdots$, $\tilde\Sigma_{1,\LL}^\inn$, and obtain
\begin{equation*}
\begin{split}
&\prod_{\ell=1}^{m-2}  \prod_{{i_\ell}=1}^{n_{\ell}} \left[\frac{1}{1-z_{\ell}}\int_{\tilde\Sigma_{\ell,\LL}^\out} \ddbarr{u_{i_\ell}^{(\ell)}}
-\frac{z_{\ell}}{1-z_{\ell}}\int_{\tilde\Sigma_{\ell,\LL}^\inn} \ddbarr{u_{i_\ell}^{(\ell)}} \right]
\cdot \prod_{{i_{m-1}}=1}^{n_{m-1}} \int_{\tilde\Sigma_{m,\LL}} \ddbarr{u_{i_{m-1}}^{(m-1)}}\\
&\qquad\left[\prod_{\ell=1}^{m-1}\Ch(U^{(\ell)};U^{(\ell+1)})\right]\cdot F(U^{(1)},\cdots,U^{(m)})\\
&=\prod_{\ell=2}^{m-1}  \prod_{{i_\ell}=1}^{n_{\ell}} \left[\frac{1}{1-z_{\ell-1}}\int_{\Sigma_{\ell,\LL}^\inn} \ddbarr{u_{i_\ell}^{(\ell)}}
-\frac{z_{\ell-1}}{1-z_{\ell-1}}\int_{\Sigma_{\ell,\LL}^\out} \ddbarr{u_{i_\ell}^{(\ell)}} \right]
\cdot \prod_{{i_1}=1}^{n_{1}} \int_{\Sigma_{1,\LL}} \ddbarr{u_{i_1}^{(1)}}\\
&\qquad\left[\prod_{\ell=1}^{m-1}\Ch(U^{(\ell)};U^{(\ell+1)})\right]\cdot F(U^{(1)},\cdots,U^{(m)})
\end{split}
\end{equation*}
for any fixed $U^{(m)}$ on $(\Sigma_{m,\LL}^\out\cup\Sigma_{m,\LL}^\inn)^{n_m}$. Together with~\eqref{eq:aux12} and the discussions above, we immediately obtain~\eqref{eq:aux05}. This finishes the induction.

\vspace{0.4cm}

Below we prove Lemma~\ref{lm:switching_contours}. We use induction on $n_1$. If $n_1=0$, the equation becomes trivial: It follows by writing (here we omit the integrand)
$$
\int_\Sigma\ddbarr{u_{i_2}^{(2)}} =\frac{1}{1-z}\int_{\Sigma^\inn}\ddbarr{u_{i_2}^{(2)}} -\frac{z}{1-z}\int_{\Sigma^\out}\ddbarr{u_{i_2}^{(2)}}
$$
since the integrand is an analytic function of $u_{i_2}^{(2)}\in\Omega$. We remark that we used such a decomposition before in Section~\ref{sec:basic_step}. See the last two equations in page 40 and the following discussions.

 Suppose the lemma holds for $n_1-1$ for some $n_1\ge 1$, we want to prove the case for $n_1$.

Consider the integral over $u_{n_1}^{(1)}$. We write
\begin{equation*}
\frac{1}{1-z }\int_{\Sigma^\out} \ddbarr{u_{n_1}^{(1)}}
-\frac{z}{1-z }\int_{\Sigma ^\inn} \ddbarr{u_{n_1}^{(1)}}=\int_{\Sigma^\out} \ddbarr{u_{n_1}^{(1)}}+ \frac{z}{1-z }\left[\int_{\Sigma ^\out} \ddbarr{u_{n_1}^{(1)}}-\int_{\Sigma ^\inn} \ddbarr{u_{n_1}^{(1)}}
\right].
\end{equation*}
Then
\begin{equation*}
\begin{split}
&\left[\frac{1}{1-z }\int_{\Sigma^\out} \ddbarr{u_{n_1}^{(1)}}
-\frac{z}{1-z }\int_{\Sigma ^\inn} \ddbarr{u_{n_1}^{(1)}}\right]\Ch(U^{(1)};U^{(2)})F(U^{(1)},U^{(2)})\\
&=\int_{\Sigma^\out} \ddbarr{u_{n_1}^{(1)}}\Ch(U^{(1)};U^{(2)})F(U^{(1)},U^{(2)})
+
\frac{z}{1-z}\sum_{j=1}^{n_2}\mathrm{Res}\left(\Ch(U^{(1)};U^{(2)})F(U^{(1)},U^{(2)}),{u_{n_1}^{(1)}=u_{j}^{(2)}}\right).
%\\
%&\qquad+\sum_{j=1}^{n_2}\left.\Ch(U^{(1)};U^{(2)})F(U^{(1)},U^{(2)})\right|_{u_{n_1}^{(1)}=u_{j}^{(2)}}.+(-1)^{n_1+n_2-1}\frac{z}{1-z}\sum_{j=1}^{n_2}(-1)^j\Ch(\hat U^{(1)};\hat U^{(2)}_{j^c})\left.F(U^{(1)},U^{(2)})\right|_{u_{n_1}^{(1)}=u_{j}^{(2)}}.
\end{split}
\end{equation*}

By plugging the above equation into the left hand side of~\eqref{eq:aux18}, we obtain
\begin{equation}
\label{eq:aux19}
\text{LHS of~\eqref{eq:aux18}}=S_1+\frac{z}{1-z}\sum_{j=1}^{n_2}S_{2,j},
\end{equation}
where
\begin{equation*}
S_1=\prod_{i_1=1}^{n_1-1}\left[\frac{1}{1-z }\int_{\Sigma^\out} \ddbarr{u_{i_1}^{(1)}}
-\frac{z}{1-z }\int_{\Sigma ^\inn} \ddbarr{u_{i_1}^{(1)}} \right]
\cdot \prod_{{i_2}=1}^{n_{2}} \int_{\Sigma } \ddbarr{u_{i_2}^{(2)}} \cdot \int_{\Sigma^\out} \ddbarr{u_{n_1}^{(1)}}\Ch(U^{(1)};U^{(2)})F(U^{(1)},U^{(2)})
\end{equation*}
and
\begin{equation*}
S_{2,j}=\prod_{i_1=1}^{n_1-1}\left[\frac{1}{1-z }\int_{\Sigma^\out} \ddbarr{u_{i_1}^{(1)}}
-\frac{z}{1-z }\int_{\Sigma ^\inn} \ddbarr{u_{i_1}^{(1)}} \right]
\cdot \prod_{{i_2}=1}^{n_{2}} \int_{\Sigma } \ddbarr{u_{i_2}^{(2)}}\mathrm{Res}\left(\Ch(U^{(1)};U^{(2)})F(U^{(1)},U^{(2)}),{u_{n_1}^{(1)}=u_{j}^{(2)}}\right).
\end{equation*}
Below we consider $S_1$ and $S_{2,j}$ separately.

For $S_1$, we first deform the contour of $u_{n_1}^{(1)}$ to some larger contour $\Sigma^{\out}_+$ in $\Omega$ which encloses $\Sigma^\out$. Then we apply induction hypothesis for other contours and obtain
\begin{equation*}
\begin{split}
S_1=\prod_{i_2=1}^{n_2}\left[\frac{1}{1-z }\int_{\Sigma^\inn} \ddbarr{u_{i_2}^{(2)}}
-\frac{z}{1-z }\int_{\Sigma ^\out} \ddbarr{u_{i_2}^{(2)}} \right]
\cdot \prod_{{i_1}=1}^{n_{1}-1} \int_{\Sigma } \ddbarr{u_{i_1}^{(1)}} \cdot \int_{\Sigma^\out_+} \ddbarr{u_{n_1}^{(1)}}\Ch(U^{(1)};U^{(2)})F(U^{(1)},U^{(2)}).
\end{split}
\end{equation*}
By deforming the contour of $u_{n_1}^{(1)}$ to $\Sigma$, we have
\begin{equation}
\label{eq:aux20}
S_1=\text{ RHS of~\eqref{eq:aux18}}-\frac{z}{1-z}\sum_{j=1}^{n_2}T_{j}
\end{equation}
with
\begin{equation*}
\begin{split}
T_j&=\prod_{\substack{1\le i_2\le n_2\\ i_2\ne j}}\left[\frac{1}{1-z }\int_{\Sigma^\inn} \ddbarr{u_{i_2}^{(2)}}
-\frac{z}{1-z }\int_{\Sigma ^\out} \ddbarr{u_{i_2}^{(2)}} \right]
\cdot \int_{\Sigma^\out}\ddbarr{u_{j}^{(2)}}\\
&\qquad\cdot \prod_{{i_1}=1}^{n_{1}-1} \int_{\Sigma } \ddbarr{u_{i_1}^{(1)}} \mathrm{Res}\left(\Ch(U^{(1)};U^{(2)})F(U^{(1)},U^{(2)}),{u_{n_1}^{(1)}=u_{j}^{(2)}}\right).
\end{split}
\end{equation*}

For $S_{2,j}$, it is easy to verify that the function
\begin{equation*}
\mathrm{Res}\left(\Ch(U^{(1)};U^{(2)})F(U^{(1)},U^{(2)}),{u_{n_1}^{(1)}=u_{j}^{(2)}}\right)=(-1)^{n_1+n_2+j-1}\Ch(\hat U^{(1)};\hat U^{(2)}_{j^c})\left.F(U^{(1)},U^{(2)})\right|_{u_{n_1}^{(1)}=u_{j}^{(2)}}.
\end{equation*}
Here the notation $\hat U^{(1)}:=(u_1^{(1)},\cdots,u_{n_1-1}^{(1)})$ is obtained by dropping the variable $u_{n_1}^{(1)}$ from the vector $U^{(1)}$, and $\hat U^{(2)}_{j^c}=(u_1^{(2)},\cdots,u_{j-1}^{(2)},u_{j+1}^{(2)},\cdots,u_{n_2}^{(2)})$ is obtained by dropping the variable $u_j^{(2)}$ from $U^{(2)}$. The above expression implies that we could deform contour of $u_j^{(2)}$ to  $\Sigma^\out$, and apply the induction hypothesis for other contours in $S_{2,j}$. This gives $S_{2,j}=T_j$. Together with~\eqref{eq:aux19} and~\eqref{eq:aux20}, we obtain~\eqref{eq:aux18}. This finishes the induction. We finish the proof of the lemma.

\subsection{Proof of Proposition~\ref{prop:invariance_distribution}}
\label{sec:proof_invariance_distribution}

We only prove the proposition with condition (1). The case for the other condition is similar.

	It is sufficient to show $\mathcal{D}_{\boldsymbol{n},Y}(z_1,\cdots,z_{m-1})$ does not change if we replace $\Kess_Y(v,u)$ by $\Kess_Y(v,u) + \Knull(v,u)$ in~\eqref{eq:D_nY}. This further reduces to prove
	\begin{equation}
	\label{eq:aux13}
	\begin{split}
	0=&\left[\prod_{\ell=2}^{m}  \prod_{{i_\ell}=1}^{n_{\ell}} \int_{\Sigma_{\ell,\LL}^\star} \ddbarr{u_{i_\ell}^{(\ell)}}\right]
	\prod_{{i_1}=1}^{n_{1}} \int_{\Sigma_{1,\LL}} \ddbarr{u_{i_1}^{(1)}}\left[\prod_{\ell=2}^{m}  \prod_{{i_\ell}=1}^{n_{\ell}}  \int_{\Sigma_{\ell,\RR}^\star} \ddbarr{v_{i_\ell}^{(\ell)}}\right]
	\cdot \prod_{{i_1}=1}^{n_{1}} \int_{\Sigma_{1}} \ddbarr{v_{i_1}^{(1)}}\\
	&\quad \Knull(v_i^{(1)},u_j^{(1)})\left[\prod_{\ell=1}^{m-1}\Ch(U^{(\ell)};U^{(\ell+1)})\Ch(V^{(\ell)};V^{(\ell+1)})\right]  F(U^{(1)},\cdots,V^{(m)})
	\end{split}
	\end{equation}
	for any $1\le i,j\le n_1$. Here the function $\Ch(W;W')$ represents the Cauchy-type factor defined in~\eqref{eq:def_cauchy}. The function $F(U^{(1)},\cdots,V^{(m)})=\tilde F(U^{(1)},\cdots,V^{(m)})\cdot \prod_{\ell=1}^m f_\ell(V^{(\ell)})$ for some function $\tilde F$ which is analytic for each  $v_{i_\ell}^{(\ell)}\in\Omega_\RR$ when $\ell\ge 1$. The symbol $\star$ represents any choice of ``$\out$'' or ``$\inn$'' in the contours of integration $\Sigma_{\ell,\LL}^\star$ and $\Sigma_{\ell,\RR}^\star$.
	
	The proof of~\eqref{eq:aux13} is a slight modification of that for Lemma~\ref{lm:lemma3}. We provide the details below for the completeness.
	
	We consider the double integral with respect to $v_i^{(1)}$ and $u_j^{(1)}$. Recall the formulas of $f_i$ defined in~\eqref{eq:fi}. We have  $f_1(v^{(1)}_i)=(v^{(1)}_i)^{-k_1}(v^{(1)}_i+1)^{a_1+k_1}e^{-t_1v^{(1)}_i}$. By applying the condition (1) of $\Knull_Y$, we know that the double integral with respect to $v_i^{(1)}$ and $u_j^{(1)}$ equals to zero if the contour of $v_{i}^{(1)}$ could be deformed sufficiently small to $0$. Thus the original double integral with respect to $v_i^{(1)}$ and $u_j^{(1)}$ only gives the possible residues at $v_i^{(1)}=v_{i'}^{(2)}$. By evaluating this residue, we obtain a new integrand $\Knull(v_{i'}^{(2)},u_{j}^{(1)})  \left[f_1(v_{i'}^{(2)})f_2(v_{i'}^{(2)})\right]$ multiplied by some other factors. Note that $f_1(v_{i'}^{(2)})f_2(v_{i'}^{(2)})= (v^{(2)}_{i'})^{-k_2}(v^{(2)}_{i'}+1)^{a_2+k_2}e^{-t_2v^{(2)}_{i'}}$. Thus the double integral with respect to $v_{i'}^{(2)}$ and $u_j^{(1)}$ equals to zero if the contour  for $v_{i'}^{(2)}$ could be deformed to sufficiently close to  $0$. We only need to evaluate the possible residues for $v_{i'}^{(2)} = v_{i''}^{(3)}$. After finitely many steps, there are no poles of this type within the contours and the last double integral becomes $0$.

\subsection{Proof of Proposition~\ref{prop:orthogonality}}
\label{sec:proof_orthogonality}

	By inserting the definition of $\Kess_Y(v,u)$, it is equivalent to prove
	\begin{equation*}
	\oint_0 v^{-i} (v+1)^{\lambda_i} \cdot \frac{1}{v-u}\cdot \ich_{\boldsymbol{\lambda}}(v,u) \ddbarr{v} =-u^{-i}(u+1)^{\lambda_i},
	\end{equation*}
	where $\lambda_1\ge\cdots\ge\lambda_N=\lambda_{N+1}=\cdots=0$. Now we fix $1\le i\le N$ and assume the integral contour is small enough such that $|v|<|u|$. It is sufficient to show, by approximating the integral using summation,
	\begin{equation}
	\label{eq:proof_Kess_int00}
	\lim_{M\to\infty}\frac{1}{M}\sum_{j=1}^{M} (v\xi ^j)^{-i+1} (v\xi^j+1)^{\lambda_i}\cdot \frac{1}{v\xi^j-u} \ich_{\boldsymbol{\lambda}} (v\xi^j,u) =-u^{-i}(u+1)^{\lambda_i},
	\end{equation}
	where $\xi=e^{\frac{2\pi\ii}{M}}$.

	We first reformulate the factor $1/(v\xi^j-u)$. By applying the Vandermonde determinant formula, we obtain
	\begin{equation*}
	\frac{\det\left[ (v\xi^{\alpha})^{-\beta}1_{\alpha\ne j}+ u^{-\beta}1_{\alpha=j}\right]_{\alpha,\beta=1}^M}{\det\left[ (v\xi^{\alpha})^{-\beta}\right]_{\alpha,\beta=1}^M}
	=\frac{v^M}{u^M}\cdot \frac{\prod_{\alpha\ne j}(u-v\xi^\alpha)}{\prod_{\alpha\ne j}(v\xi^j-v\xi^\alpha)}.
	\end{equation*}
	Moreover, by using the property of $\xi$, it is easy to see
	\begin{equation*}
	\prod_{\alpha\ne j}(u-v\xi^\alpha)=\frac{u^M-v^M}{u-v\xi^j},\quad 
	\prod_{\alpha\ne j}(v\xi^j-v\xi^\alpha) = M (v\xi^j)^{M-1}=\frac{Mv^M}{v\xi^j}.
	\end{equation*}
	Thus we have
	\begin{equation}
	\label{eq:proof_Kess_int01}
	\frac{1}{v\xi^j - u} =-\frac{M}{v\xi^j} \cdot \frac{u^M}{u^M-v^M}\cdot \frac{\det\left[ (v\xi^{\alpha})^{-\beta}1_{\alpha\ne j}+ u^{-\beta}1_{\alpha=j}\right]_{\alpha,\beta=1}^M}{\det\left[ (v\xi^{\alpha})^{-\beta}\right]_{\alpha,\beta=1}^M}.
	\end{equation}

	On the other hand, by applying the Cramer's rule, we have
	\begin{equation*}
	\sum_{j=1}^M (v\xi^j)^{-i}(v\xi^j+1)^{\lambda_i}\frac{\det\left[ (v\xi^{\alpha})^{-\beta}(v\xi^{\alpha}+1)^{\lambda_\beta}1_{\alpha\ne j}+ u^{-\beta}(u+1)^{\lambda_\beta}1_{\alpha=j}\right]_{\alpha,\beta=1}^M}{\det\left[ (v\xi^{\alpha})^{-\beta}(v\xi^\alpha+1)^{\lambda_\beta}\right]_{\alpha,\beta=1}^M}=u^{-i}(u+1)^{\lambda_i}.
	\end{equation*}
	Thus if $M\ge |\boldsymbol{\lambda}|$, by using the formula of $\ich_{\boldsymbol{\lambda}}(v\xi^j,u)$ in~\eqref{eq:alter_ich} we obtain
	\begin{equation}
	\label{eq:proof_Kess_int02}
	\sum_{j=1}^M (v\xi^j)^{-i}(v\xi^j+1)^{\lambda_i}\cdot \ich_{\boldsymbol{\lambda}}(v\xi^j,u) \cdot \frac{\det\left[ (v\xi^{\alpha})^{-\beta}1_{\alpha\ne j}+ u^{-\beta}1_{\alpha=j}\right]_{\alpha,\beta=1}^M}{\det\left[ (v\xi^{\alpha})^{-\beta}(v\xi^\alpha+1)^{\lambda_\beta}\right]_{\alpha,\beta=1}^M}=u^{-i}(u+1)^{\lambda_i}.
	\end{equation}
	
	Now we combine~\eqref{eq:proof_Kess_int01} and~\eqref{eq:proof_Kess_int02} and get
	\begin{equation*}
	\begin{split}
	\sum_{j=1}^M(v\xi ^j)^{-i+1} (v\xi^j+1)^{\lambda_i}\cdot \frac{1}{v\xi^j-u} \ich_{\boldsymbol{\lambda}} (v\xi^j,u)&=-u^{-i}(u+1)^{\lambda_i}\cdot \frac{Mu^M}{u^M-v^M}\cdot \frac{\det\left[ (v\xi^{\alpha})^{-\beta}(v\xi^\alpha+1)^{\lambda_\beta}\right]_{\alpha,\beta=1}^M}{\det\left[ (v\xi^{\alpha})^{-\beta}\right]_{\alpha,\beta=1}^M}\\
	&=-u^{-i}(u+1)^{\lambda_i}\cdot \frac{Mu^M}{u^M-v^M}
	\end{split}
	\end{equation*}
	for $M\ge |\boldsymbol{\lambda}|$. Here we used the fact that $\frac{\det\left[ (v\xi^{\alpha})^{-\beta}(v\xi^\alpha+1)^{\lambda_\beta}\right]_{\alpha,\beta=1}^M}{\det\left[ (v\xi^{\alpha})^{-\beta}\right]_{\alpha,\beta=1}^M}=\ich_{\boldsymbol{\lambda}}(v,v)=1$. Together with the fact that $v^M/u^M\to0$, we obtain~\eqref{eq:proof_Kess_int00} immediately.	

%Hence, $\Kess_Y(v,u)$ equals to a special solution plus any function orthogonal to $v^{-i}(v+1)^{y_i+i}$, $1\le i\le N$. 

\subsection{Proof of Proposition~\ref{prop:Kess_flat}}
\label{sec:proof_Kess_flat}

In order to evaluate $\Kess_{Y_{\mathrm{pf}}}$, we need to consider the function $\ich_{\boldsymbol{\lambda}(Y_{\mathrm{pf}})}(v,u)$. Recall the formula~\eqref{eq:alter_ich_ext}, we write
\begin{equation*}
\ich_{\boldsymbol{\lambda}(Y_{\mathrm{pf}})}(v,u)=\mathcal{G}_{\boldsymbol{\lambda}(Y_{\mathrm{pf}})}(u,v\xi,\cdots,v\xi^{N-1}) + v^{N}\cdot r_1(v,u),
\end{equation*}
where $\xi=e^{2\pi\ii/N}$, and $r_1(v,u)$ is some polynomial. It turn out that
\begin{equation}
\label{eq:aux24}
\mathcal{G}_{\boldsymbol{\lambda}(Y_{\mathrm{pf}})}(u,v\xi,\cdots,v\xi^{N-1})=\frac{2v+1}{u+v+1}\cdot\left(\frac{u+1}{v+1}\right)^N+ v^N\cdot r_2(v,u)
\end{equation} 
for some function $r_2(v,u)$ which is analytic for $(v,u)$ satisfying $|v|<\{1/2,|u+1|\}$. By combing the above two equations and using the definition of $\Kess_{Y_{\mathrm{pf}}}$ we prove the proposition immediately.

\vspace{0.3cm}
It remains to prove~\eqref{eq:aux24}. We show it below.

Note that the function $\boldsymbol{\lambda}(Y_{\mathrm{pf}})=(\lambda_1,\cdots,\lambda_N)$ with $\lambda_i=(y_i+i)-(y_N+N)=N-i$. Thus
\begin{equation*}
\mathcal{G}_{\boldsymbol{\lambda}(Y_{\mathrm{pf}})}(w_1,\cdots,w_N)=\frac{\det\left[w_i^{-j}(1+w_i)^{N-j}\right]_{i,j=1}^N}{\det\left[w_i^{-j}\right]_{i,j=1}^N}.
\end{equation*} 
By applying the Vandermonde determinant formula, we have
\begin{equation*}
\mathcal{G}_{\boldsymbol{\lambda}(Y_{\mathrm{pf}})}(w_1,\cdots,w_N)=\prod_{i<j}\frac{w_j(w_j+1)-w_i(w_i+1)}{w_j-w_i}=\prod_{1\le i<j\le N}(w_i+w_j+1).
\end{equation*}
As a result, $\mathcal{G}_{\boldsymbol{\lambda}(Y)}(u,v\xi,\cdots,v\xi^{N-1})$ can be expressed as
\begin{equation}
\label{eq:aux21}
\mathcal{G}_{\boldsymbol{\lambda}(Y)}(u,v\xi,\cdots,v\xi^{N-1})=\prod_{1\le j\le N-1}\frac{u+1+v\xi^j}{v+1+v\xi^j} \cdot \prod_{0\le i<j\le N-1}(v\xi^i+v\xi^j+1).
\end{equation}
We remark that our assumption of $|v|<1/2$ guarantees $v+1+v\xi^j\ne 0$ for each $j$. Note that the last product is invariant under $v\to v\xi^j$ for any $j$, therefore
\begin{equation}
\label{eq:aux22}
\prod_{0\le i<j\le N-1}(v\xi^i+v\xi^j+1)=1+v^N\cdot  r_3(v^N) 
\end{equation}
for some polynomial $r_3$. Moreover, the following two identities hold since $\xi$ is the root of unity
\begin{equation*}
\prod_{j=0}^{N-1}(u+1+v\xi^j) = (u+1)^N -(-v)^N,\quad \prod_{j=0}^{N-1}(v+1+v\xi^j) =(v+1)^N -(-v)^N.
\end{equation*}
Thus 
\begin{equation}
\label{eq:aux08}
\prod_{1\le j\le N-1}\frac{u+1+v\xi^j}{v+1+v\xi^j}=\frac{2v+1}{u+v+1}\cdot \frac{ (u+1)^N -(-v)^N}{(v+1)^N -(-v)^N} =\frac{2v+1}{u+v+1}\cdot\left(\frac{u+1}{v+1}\right)^N\cdot (1+v^Nr_4(v,u)),
\end{equation}
where
\begin{equation*}
r_4(v,u)=\frac{(-1)^N}{(u+1)^N}\cdot\frac{(u+1)^N-(v+1)^N}{(v+1)^N-(-v)^N}
\end{equation*}
is also analytic in $|v|<\min\{1/2,|u+1|\}$.~\eqref{eq:aux24} follows from combing~\eqref{eq:aux21},~\eqref{eq:aux22} and~\eqref{eq:aux08}.

\subsection{Proof of Proposition~\ref{prop:D_flat}}
\label{sec:proof_D_flat}

In this section we prove Proposition~\ref{prop:D_flat}. The proof is based on a Cauchy chain argument similar to that of Lemma~\ref{lm:lemma3} and a deformation of contour.

As we discussed before the proposition, we could combine Propositions~\ref{prop:Kess_flat} and~\ref{prop:invariance_distribution} and replace the original kernel $\Kess_{Y_{\mathrm{pf}}}(v,u)$ by the following kernel
\begin{equation*}
\mathcal{K}_{Y_{\mathrm{pf}}}^{(\mathrm{ess},1)}(v,u)=\frac{2v+1}{(v-u)(u+v+1)}
\end{equation*}
if we choose the contours described below. The contours $\hat\Sigma_{m,\RR}^\out,\cdots,\hat\Sigma_{2,\RR}^\out,\hat\Sigma_{1,\RR},\hat\Sigma_{2,\RR}^\inn,\cdots,\hat\Sigma_{m,\RR}^\inn$ are nested contours within the region $\disk(1/2)=\{v: |v|<1/2\}$, and the contours $\hat\Sigma_{m,\RR}^\out,\cdots,\hat\Sigma_{2,\RR}^\out,\hat\Sigma_{1,\RR},\hat\Sigma_{2,\RR}^\inn,\cdots,\hat\Sigma_{m,\RR}^\inn$ are nested contours around $-1$ satisfying $\hat\Sigma_{1,\LL}$ is outside of $-1-\hat\Sigma_{1,\RR}=\{-1-v: v\in\hat\Sigma_{1,\RR}\}$ and $\hat\Sigma_{2,\LL}^\inn$ is inside $-1-\hat\Sigma_{1,\RR}$.

Now we evaluate $\mathcal{D}_{\boldsymbol{n},Y_{\mathrm{pf}}}$ below
\begin{equation*}
\begin{split}
\mathcal{D}_{\boldsymbol{n},Y_{\mathrm{pf}}}(z_1,\cdots,z_{m-1})&=\prod_{\ell=1}^m\prod_{i_\ell=1}^{n_\ell} \int_{\Sigma_{\ell,\LL}}\mathrm{d}\mu_{\boldsymbol{z}}(u_{i_\ell}^{(\ell)})\int_{\Sigma_{\ell,\RR}}\mathrm{d}\mu_{\boldsymbol{z}}(v_{i_\ell}^{(\ell)})\\
&\quad \left[(-1)^{n_1(n_1+1)/2}\frac{\Delta(U^{(1)};V^{(1)})}{\Delta(U^{(1)})\Delta(V^{(1)})} \det\left[\mathcal{K}_{Y_{\mathrm{pf}}}^{(\mathrm{ess},1)}(v_i^{(1)},u_j^{(1)})\right]_{i,j=1}^{n_1}\right]\\
&\quad\cdot \left[\prod_{\ell=1}^m\frac{(\Delta(U^{(\ell)}))^2(\Delta(V^{(\ell)}))^2}{(\Delta(U^{(\ell)};V^{(\ell)}))^2}f_\ell(U^{(\ell)}) f_\ell(V^{(\ell)})\right]\\
&\quad\cdot \left[\prod_{\ell=1}^{m-1} \frac{\Delta(U^{(\ell)};V^{(\ell+1)})\Delta(V^{(\ell)};U^{(\ell+1)})}{\Delta(U^{(\ell)};U^{(\ell+1)})\Delta(V^{(\ell)};V^{(\ell+1)})}\left(1-z_\ell\right)^{n_\ell}\left(1-\frac{1}{z_\ell}\right)^{n_{\ell+1}}\right].
\end{split}
\end{equation*}
Recall that we assume $a_1+k_1\le 0$ in this proposition. Therefore the function $f_1(u )=u^{k_1}(u +1)^{-a_1-k_1}e^{t_1u}$ is analytic at $u =-1$. After we integrate  $u_{i_1}^{(1)}$ along $\hat\Sigma_{1,\LL}$, only two types of residues survive: $u_{i_1}^{(1)}=-v_{j_1}^{(1)}-1$ for some $1\le j_1\le n_1$, or $u_{i_1}^{(1)}=u_{i_2}^{(2)}\in \hat\Sigma_{2,\LL}^\inn$ for some $1\le i_2\le n_2$. These two types of residues come from the term $\mathcal{K}_{Y_{\mathrm{pf}}}^{(\mathrm{ess},1)}(v_{j_1}^{(1)},u_{i_1}^{(1)})$ and $\frac{1}{u_{i_1}^{(1)}-u_{i_2}^{(2)}}$ respectively. We claim that the second type of residues contribute a zero. In fact, after evaluating the residue at $u_{i_1}^{(1)}=u_{i_2}^{(2)}$, the integrand has the form $\mathcal{K}_{Y_{\mathrm{pf}}}^{(\mathrm{ess},1)}(v_{j_1}^{(1)},u_{i_2}^{(2)})\cdot f_1(u_{i_2}^{(2)})f_2(u_{i_2}^{(2)})\cdot \frac{1}{\Delta(U^{(2)};U^{(3)})}$ times some function analytic for $u_{i_2}^{(2)}$ inside the region bounded by the contour $\hat\Sigma_{2,\LL}^\inn$. This integrand again is analytic at $u_{i_2}^{(2)}=-1$ since $f_1(u)f_2(u)=u^{k_2}(u +1)^{-a_2-k_2}e^{t_2u}$ and $a_2+k_2\le 0$ by our assumption. Hence we only need to evaluate the residues of $u_{i_2}^{(2)}$. Now due to the assumption that $\Sigma_{2,\LL}^\inn$ is inside $-1-\Sigma_{1,\RR}$, there is only one type of residues $u_{i_2}^{(2)}=u_{i_3}^{(3)}$ for some $i_3$. We repeat this procedure and finally will stop at some step when no residues are inside the contour. This procedure ends with no nonzero contribution. Thus the claim is true.

Now the above argument implies that the integral with respect to $u_{i_1}^{(1)}$ only gives the residues at $u_{i_1}^{(1)}=-v_{j_1}^{(1)}-1$ from $\mathcal{K}_{Y_{\mathrm{pf}}}^{(\mathrm{ess},1)}(v_{j_1}^{(1)},u_{i_1}^{(1)})$. Therefore this integral is the same as an integral along $-1-\Sigma_{1,\RR}$ with the kernel $\mathcal{K}_{Y_{\mathrm{pf}}}^{(\mathrm{ess},1)}(v_{j_1}^{(1)},u_{i_1}^{(1)})$ replaced by $\delta(-v_{j_1}^{(1)}-1,u_{i_1}^{(1)})$. Therefore $\mathcal{D}_{\boldsymbol{n},Y_{\mathrm{pf}}}$ does not change if we replace the contour $\hat\Sigma_{1,\LL}$ by $-1-\hat\Sigma_{1,\RR}$ and the kernel $\mathcal{K}_{Y_{\mathrm{pf}}}^{(\mathrm{ess},1)}(v,u)$ by $\delta(-v-1,u)$. These replacements also do not change $\mathcal{D}_{Y_{\mathrm{pf}}}$. With this new kernel, we are free to deform the contours $\hat \Sigma_{1,\RR},\hat{\Sigma}_{\ell,\RR}^\inn$, $\hat\Sigma_{\ell,\RR}^\out$ and $\hat{\Sigma}_{\ell,\LL}^\inn$, $\hat\Sigma_{\ell,\LL}^\out$, $2\le\ell\le m$, to  $ \Sigma_{1,\RR},{\Sigma}_{\ell,\RR}^\inn$, $\Sigma_{\ell,\RR}^\out$ and ${\Sigma}_{\ell,\LL}^\inn$, $\Sigma_{\ell,\LL}^\out$, $2\le\ell\le m$, respectively. This finishes the proof.

\subsection{Proof of Proposition~\ref{lm:lemma4}}
\label{sec:proof_lemma4}

	By using the series expansion of $\mathcal{D}_Y$ in Definition~\ref{def:D_Y_series}, we only need to show
	\begin{equation}
	\label{eq:consistency}
	\begin{split}
	&\sum_{n_{s+1}\ge 0}\frac{1}{((n_{s+1})!)^2}\left[\oint_{|z_s|<1}\ddbar{z_s} -\oint_{|z_s|>1}\ddbar{z_s} \right]\frac{1}{1-z_s}\mathcal{D}_{\boldsymbol{n},Y}(z_1,\cdots,z_{m-1})\\
	&=\mathcal{D}_{\hat{\boldsymbol{n}},Y}(z_1,\cdots,z_{s-1},z_{s+1},\cdots,z_{m-1};(a_1,k_1,t_1),\cdots,(a_{s-1},k_{s-1},t_{s-1}),(a_{s+1},k_{s+1},t_{s+1}),\cdots,(a_m,k_m,t_m)).
	\end{split}
	\end{equation}
	Here the vector $\hat{\boldsymbol{n}}:=(n_1,\cdots,n_{s},n_{s+2},\cdots,n_m)$ is the vector obtained by removing $n_{s+1}$ from $\boldsymbol{n}$. We also list the parameters $(a_\ell,k_\ell,t_\ell)$ ($\ell\ne s$) to avoid possible confusion.
	
	By dropping common factors in the series expansion formula of both sides of~\eqref{eq:consistency}, it is sufficient to show
	\begin{equation}
	\label{eq:aux03}
	\begin{split}
	&\sum_{n_{s+1}\ge 0}\frac{1}{((n_{s+1})!)^2}\left[\oint_{|z_s|<1}\ddbar{z_s} -\oint_{|z_s|>1}\ddbar{z_s} \right]	\\
	&
	\prod_{{i}=1}^{n_{s+1}} \left[\int_{\Sigma_{s+1,\LL}^\inn} \ddbarr{u_{i}^{(s+1)}}
	-z_s\int_{\Sigma_{s+1,\LL}^\out} \ddbarr{u_{i}^{(s+1)}} \right]
	\prod_{{i}=1}^{n_{s+1}} \left[\int_{\Sigma_{s+1,\RR}^\inn} \ddbarr{v_{i}^{(s+1)}}
	-z_s\int_{\Sigma_{s+1,\RR}^\out} \ddbarr{v_{i}^{(s+1)}} \right]\\
	&{(1-z_s)^{n_{s}-n_{s+1}-1}}{z_s^{-n_{s+1}}} (1-z_{s+1})^{n_{s+1}}\Ch(U^{(s)};U^{(s+1)}) \Ch(V^{(s)};V^{(s+1)})B(U^{(s)},U^{(s+1)};V^{(s)},V^{(s+1)}) \\
	&= (1-z_{s+1})^{n_s}B(U^{(s)},U^{(s)};V^{(s)},V^{(s)})
	\end{split}
	\end{equation}
	for any function $B(U^{(s)},U^{(s+1)};V^{(s)},V^{(s+1)})$ which satisfies (a) it is analytic for $u_i^{(s+1)}$ between the contours $\Sigma_{s+1,\LL}^\out$ and $\Sigma_{s+1,\LL}^\inn$, and $v_i^{(s+1)}$ between the contours $\Sigma_{s+1,\RR}^\out$ and $\Sigma_{s+1,\RR}^\inn$, $1\le i\le n_{s+1}$, and (b) it is anti-symmetric for $u_1^{(s+1)},\cdots,u_{n_{s+1}}^{(s+1)}$, and anti-symmetric for $v_1^{(s+1)},\cdots,v_{n_{s+1}}^{(s+1)}$. In other words, exchanging two variables $u_i^{(s+1)}$ and $u_j^{(s+1)}$ in $B$ only gives a sign change, and so is the exchanging of $v_i^{(s+1)}$ and $v_j^{(s+1)}$.	
	The function $\Ch(W;W')$ is the Cauchy-type factor defined in~\eqref{eq:def_cauchy}.
	%\begin{equation*}
	%\Ch(W;W')=\frac{\Delta(W)\Delta(W')}{\Delta(W;W')}.
	%\end{equation*}
	
	We write the summand on the left hand side of~\eqref{eq:aux03} as $(1-z_{s+1})^{n_{s+1}}\cdot B_{n_{s+1}}$. The equation~\eqref{eq:aux03} follows from the following identity
	\begin{equation}
	\label{eq:aux04}
	B_{n_{s+1}}=\begin{dcases}
	B(U^{(s)},U^{(s)};V^{(s)},V^{(s)}),& n_{s+1}=n_{s},\\
	0,&\text{otherwise}.
	\end{dcases}
	\end{equation}
	It remains to prove~\eqref{eq:aux04}. We prove it by considering all the three cases below.
	
	\vspace{0.3cm}
	\textbf{Case (1)}. $n_{s+1} < n_{s}$.
	
	This case is trivial. The $z_s$ integral is zero since the  integrand is analytic at $z_s=1$: there is no pole between the contours $|z_s|<1$ and $|z_s|>1$.
	
	\textbf{Case (2)}. $n_{s+1} > n_{s}$.
	
	$z_s=1$ is a pole of order $n_{s+1}-n_{s}+1$. Thus the integral of $z_s$ gives 
	\begin{equation*}
	\begin{split}
	&B_{n_{s+1}}	\\
	&=c\cdot
	\left.\frac{\dd^{n_{s+1}-n_{s}}}{\dd z_s^{n_{s+1}-n_{s}}}\right|_{z_s=1}\left({z_s^{-n_{s+1}}}\prod_{{i}=1}^{n_{s+1}} \left[\int_{\Sigma_{s+1,\LL}^\inn} 
	-z_s\int_{\Sigma_{s+1,\LL}^\out}  \right]
	\prod_{{i}=1}^{n_{s+1}} \left[\int_{\Sigma_{s+1,\RR}^\inn} 
	-z_s\int_{\Sigma_{s+1,\RR}^\out} \right]\right)\\
	&\quad (-1)^{n_{s}-n_{s+1}} \Ch(U^{(s)};U^{(s+1)}) \Ch(V^{(s)};V^{(s+1)})B(U^{(s)},U^{(s+1)};V^{(s)},V^{(s+1)})
	\end{split}
	\end{equation*}
	for some constant $c=\frac{1}{((n_{s+1})!)^2 (n_{s+1}-n_s)!}$. Here for the sake of saving space, we omit the integral symbols $\ddbarr{u_{i}^{(s+1)}}$ in the integrals $\int_{\Sigma_{s+1,\LL}^\inn} $ and $\int_{\Sigma_{s+1,\LL}^\out}$, and $\ddbarr{v_{i}^{(s+1)}}$ in $\int_{\Sigma_{s+1,\RR}^\inn} $ and $\int_{\Sigma_{s+1,\RR}^\out}$. Note that there are $2n_{s+1}$ integrals of the form $\int_{\Sigma_{s+1,\Delta}^\out} -z_s\int_{\Sigma_{s+1,\Delta}^\inn}$ for $\Delta\in\{\LL,\RR\}$. After the $n_{s+1}-n_{s}$ times of differentiation with respect to $z_s$, there are still at least $2n_{s+1}-(n_{s+1}-n_{s})=n_s+n_{s+1}$ integrals of the form $\int_{\Sigma_{s+1,\Delta}^\out} -\int_{\Sigma_{s+1,\Delta}^\inn}$ survive (with $z_s=1$). On the other hand, each integral $\int_{\Sigma_{s,\Delta}^\out} -\int_{\Sigma_{s,\Delta}^\inn}$ is either zero or equals to some residue at $u_i^{(s+1)}=u_{i'}^{(s)}$ or $v_i^{(s+1)}=v_{i'}^{(s)}$. It is easy to count the maximal possible numbers of these residues from $\Ch(U^{(s)};U^{(s+1)})$ and $\Ch(V^{(s)};V^{(s+1)})$ are both $\min\{n_s,n_{s+1}\}$. With our assumption, $n_s+n_{s+1}>2\min\{n_s,n_{s+1}\}$. Thus there exists at least one integral $\int_{\Sigma_{s+1,\Delta}^\out} -\int_{\Sigma_{s+1,\Delta}^\inn}$, which survives from the $z_s$ differentiation, does not contribute any residue from the Cauchy-type factors. This integral is zero. Thus $B_{n_{s+1}}=0$.
	
	\textbf{Case (3)}. $n_{s+1} = n_{s}$. 
	Similar to the Case (2), we have
	\begin{equation*}
	\begin{split}
	B_{n_{s+1}}	&=\frac{1}{((n_{s+1})!)^2}\prod_{{i}=1}^{n_{s+1}} \left[\int_{\Sigma_{s+1,\LL}^\inn} \ddbarr{u_{i}^{(s+1)}}
	-\int_{\Sigma_{s+1,\LL}^\out} \ddbarr{u_{i}^{(s+1)}} \right]
	\prod_{{i}=1}^{n_{s+1}} \left[\int_{\Sigma_{s+1,\RR}^\inn} \ddbarr{v_{i}^{(s+1)}}
	-\int_{\Sigma_{s+1,\RR}^\out} \ddbarr{v_{i}^{(s+1)}} \right]\\
	&\quad \Ch(U^{(s)};U^{(s+1)}) \Ch(V^{(s)};V^{(s+1)})B(U^{(s)},U^{(s+1)};V^{(s)},V^{(s+1)}). 
	\end{split}
	\end{equation*}
	The nonzero contributions come from the residues of
	\begin{equation}
	\label{eq:residue}
	\mathrm{Res}\left(\Ch(U^{(s)};U^{(s+1)}) \Ch(V^{(s)};V^{(s+1)})B(U^{(s)},U^{(s+1)};V^{(s)},V^{(s+1)}), U^{(s+1)}=\sigma(U^{(s)}),V^{(s+1)}=\sigma'(V^{(s)})\right)
	\end{equation}
	for some permutations $\sigma,\sigma'\in S_{n_{s+1}}$, where $\sigma(W)$ denotes the permuted vector $W$ by $\sigma$. More precisely, if $W=(w_1,\cdots,w_n)$ and $\sigma\in S_n$, then $\sigma(W):=(w_{\sigma(1)},\cdots,w_{\sigma(n)})$. Moreover, we used a more general notation of the residue. It could be understood as a composition of taking residues one by one. For example, $\mathrm{Res}(f(w_1,w_2),w_1=c_1,w_2=c_2)$ means $\mathrm{Res}(\mathrm{Res}(f(w_1,w_2),w_1=c_1),w_2=c_2)$.
	
	Since $B(U^{(s)},U^{(s+1)};V^{(s)},V^{(s+1)})$ is anti-symmetric on the coordinates of $U^{(s+1)}$, and on the coordinates of $V^{(s+1)}$, it is a direct to verify that the residue~\eqref{eq:residue} is independent of the choices of $\sigma$ and $\sigma'$. There are $((n_{s+1})!)^2$ choices of $\sigma$ and $\sigma'$. Thus 
	\begin{equation*}
	\begin{split}
	&B_{n_{s+1}}\\
	&=(-1)^{2n_{s+1}}\mathrm{Res}\left(\Ch(U^{(s)};U^{(s+1)}) \Ch(V^{(s)};V^{(s+1)})B(U^{(s)},U^{(s+1)};V^{(s)},V^{(s+1)}), U^{(s+1)}=U^{(s)},V^{(s+1)}=V^{(s)}\right)\\
	&=B(U^{(s)},U^{(s)};V^{(s)},V^{(s)}).
	\end{split}
	\end{equation*}
	This finishes the proof.

\subsection{Proof of Proposition~\ref{cor:consistence}}
\label{sec:proof_consistence_cor}

	When $s=m$, note that 
	\begin{equation*}
	(1-z_{m-1})^{n_{m-1}}\mathcal{D}_{\tilde{\boldsymbol{n}},Y}(z_1,\cdots,z_{m-2})=\mathcal{D}_{\boldsymbol{n},Y}(z_1,\cdots,z_{m-1})
	\end{equation*}
	with $\boldsymbol{n}=(n_1,\cdots,n_{m-1},0)$ and $\tilde{\boldsymbol{n}}=(n_1,\cdots,n_{m-1})$. Thus we just need to prove that if $n_m\ge 1$
	\begin{equation}
	\label{eq:aux14}
\begin{split}
0=&\left[\prod_{\ell=2}^{m}  \prod_{{i_\ell}=1}^{n_{\ell}} \int_{\Sigma_{\ell,\LL}^\star} \ddbarr{u_{i_\ell}^{(\ell)}}\right]
\prod_{{i_1}=1}^{n_{1}} \int_{\Sigma_{1,\LL}} \ddbarr{u_{i_1}^{(1)}}\left[\prod_{\ell=2}^{m}  \prod_{{i_\ell}=1}^{n_{\ell}}  \int_{\Sigma_{\ell,\RR}^\star} \ddbarr{v_{i_\ell}^{(\ell)}}\right]
\cdot \prod_{{i_1}=1}^{n_{1}} \int_{\Sigma_{1,\RR}} \ddbarr{v_{i_1}^{(1)}}\\
&\quad \left[\prod_{\ell=1}^{m-1}\Ch(U^{(\ell)};U^{(\ell+1)})\Ch(V^{(\ell)};V^{(\ell+1)})\right]  F(U^{(1)},\cdots,V^{(m)}).
\end{split}
\end{equation}
Here the function $\Ch(W;W')$ represents the Cauchy-type factor defined in~\eqref{eq:def_cauchy}. The function 
$$F(U^{(1)},\cdots,V^{(m)})=\tilde F(U^{(1)},\cdots,V^{(m)})\cdot \left(\prod_{i_1=1}^{n_1}\frac{u_{i_1}^{(1)}+1}{v_{i_1}^{(1)}+1}\right)^{y_N+N}\cdot \prod_{\ell=1}^m f_\ell(U^{(\ell)})f_\ell(V^{(\ell)})
$$ for some function $\tilde F$ which is analytic for each $u_{i_\ell}^{(\ell)}\in\Omega_\LL$  and each $v_{i_\ell}^{(\ell)}\in\Omega_\RR$, $\ell=1,\cdots,m$. The symbol $\star$ represents any choice of ``$\out$'' or ``$\inn$'' in each integral contour $\Sigma_{\ell,\LL}^\star$ or $\Sigma_{\ell,\RR}^\star$. By the definition of $f_\ell$ and the assumption that $a_m+k_m=\min\{a_\ell+k_\ell: 1\le\ell\le m\}<y_N+N$, we know that $F$ is analytic at $-1$ along any chain of variable $u_{j_s}^{(s)},u_{j_{s+1}}^{(s+1)},\cdots,u_{j_m}^{(m)}$ with $j_m=1$ and any $j_\ell$ satisfying $1\le j_\ell\le n_\ell$ for $s\le \ell<m$. More explicitly,
\begin{equation*}
\left.F(U^{(1)},\cdots,V^{(m)})\right|_{u_{j_s}^{(s)}=u_{j_{s+1}}^{(s+1)}=\cdots=u_{j_m}^{(m)}=u}
\end{equation*} 
is analytic at $u=-1$ when all other variables are fixed. Thus we could apply Lemma~\ref{lm:lemma3}.~\eqref{eq:aux14} follows.

	When $s<m$, after applying Proposition~\ref{lm:lemma4}, we only need to prove
	\begin{equation*}
	\oint\cdots\oint \left[\prod_{\ell=1}^{m-1}\frac{1}{1-z_\ell}\right] \mathcal{D}_Y(z_1,\cdots,z_{m-1}) \ddbar{z_1}\cdots\ddbar{z_{m-1}}=0
	\end{equation*}
	if the radius of $z_s$ contour is greater than $1$.
	
	By using the series expansion formula of $\mathcal{D}_Y$, it is sufficient to prove
	\begin{equation}
	\label{eq:aux23}
	\oint_{|z_s|>1}\frac{1}{1-z_s}\mathcal{D}_{\boldsymbol{n},Y}(z_1,\cdots,z_{m-1})\ddbar{z_s}=0
	\end{equation}
	for any $\boldsymbol{n}=(n_1,\cdots,n_m)\in(\intZ_{\ge 0})^m$. By using the formula~\eqref{eq:D_nY}, we write
\begin{equation*}
\begin{split}
&\mathcal{D}_{\boldsymbol{n},Y}(z_1,\cdots,z_{m-1})\\
&=\prod_{\ell=2}^{m}  \prod_{{i_\ell}=1}^{n_{\ell}} \left[\frac{1}{1-z_{\ell-1}}\int_{\Sigma_{\ell,\LL}^\inn} \ddbarr{u_{i_\ell}^{(\ell)}}
-\frac{z_{\ell-1}}{1-z_{\ell-1}}\int_{\Sigma_{\ell,\LL}^\out} \ddbarr{u_{i_\ell}^{(\ell)}} \right]
\cdot \prod_{{i_1}=1}^{n_{1}} \int_{\Sigma_{1,\LL}} \ddbarr{u_{i_1}^{(1)}}\\
&\quad \prod_{\ell=2}^{m}  \prod_{{i_\ell}=1}^{n_{\ell}} \left[\frac{1}{1-z_{\ell-1}}\int_{\Sigma_{\ell,\RR}^\inn} \ddbarr{v_{i_\ell}^{(\ell)}}
-\frac{z_{\ell-1}}{1-z_{\ell-1}}\int_{\Sigma_{\ell,\RR}^\out} \ddbarr{v_{i_\ell}^{(\ell)}} \right]
\cdot \prod_{{i_1}=1}^{n_{1}} \int_{\Sigma_{1,\RR}} \ddbarr{v_{i_1}^{(1)}}\\
&\quad\cdot \left[\prod_{i_1=1}^{n_1}(u_{i_1}^{(1)}+1)^{y_N+N}\right]\left[\prod_{\ell=1}^mf_\ell(U^{(\ell)}) \right]\cdot \left[\prod_{\ell=1}^{m-1} \Ch(U^{(\ell)};U^{(\ell+1)})\right]\\
&\quad \cdot \left(1-z_s\right)^{n_s}\left(1-\frac{1}{z_s}\right)^{n_{s+1}}\cdot F(U^{(1)},\cdots,V^{(m)},z_1,\cdots,z_{s-1},z_{s+1},\cdots,z_{m-1}),
\end{split}
\end{equation*}
where the function $F$ is analytic for each $u_{i_\ell}^{(\ell)}\in\Omega_\LL$.

Below we will use an argument similar to Lemma~\ref{lm:lemma3}. We evaluate the integral with respect to each $u_{i_s}^{(s)}$, $1\le i_s\le n_{s}$. Note that the function $f_s(u_{i_s}^{(s)})$ is analytic at $u_{i_s}^{(s)}=-1$ by the assumption that $a_s+k_s=\min\{a_\ell+k_\ell:1\le \ell\le m\}$. Therefore only the residues at $u_{i_s}^{(s)}=u_{i_{s+1}}^{(s+1)}$ for some $u_{i_{s+1}}^{(s+1)}\in\Sigma_{s+1,\LL}^\inn$, and, if $u_{i_s}^{(s)}\in\Sigma_{s,\LL}^\out$, the residues at  $u_{i_s}^{(s)}=u_{i_{s-1}}^{(s-1)}$ for some $u_{i_{s-1}}^{(s-1)}\in \Sigma_{s-1,\LL}^\inn\cup \Sigma_{s-1,\LL}^\out$ survive. Here we used the nesting order of the contours. We claim that the second type of residues does not contribute after we integrate over $u_{i_{s-1}}^{(s-1)}$. In fact, considering the fact that $f_s(u_{i_{s-1}}^{(s-1)})f_{s-1}(u_{i_{s-1}}^{(s-1)})$ is still analytic at $u_{i_{s-1}}^{(s-1)}=-1$ by our assumption that $a_{s-2}+k_{s-2}\ge a_s+k_s$, the integral with respect to $u_{i_{s-1}}^{(s-1)}$ only leaves a further level of residues $u_{i_{s-1}}^{(s-1)}=u_{i_{s-2}}^{(s-2)}$. This procedure will end at $u_{i_{s}}^{(s)}=u_{i_{s-1}}^{(s-1)}=u_{i_{s-2}}^{(s-2)}=\cdots=u_{i_1}^{(1)}$. At the last step, the integral is $0$ since $(u_{i_1}^{(1)}+1)^{y_N+N}\prod_{\ell=1}^{s}f_\ell(u_{i_1}^{(1)})$ is analytic at $u_{i_1}^{(1)}=-1$ due to the assumption that $y_N+N\ge a_s+k_s$. This proves the claim. Therefore, only the first type of residues survive for each $u_{i_s}^{(s)}$ integral. Note that there are $n_s$ such integrals, therefore  $\mathcal{D}_{\boldsymbol{n},Y}=0$ if $n_s>n_{s+1}$. When $n_{s+1}\ge n_s$, we only need to consider the case when there are at least $n_s$ variables  $u_{i_{s+1}}^{(s+1)}$  chosen from $\Sigma_{s+1,\LL}^{\inn}$. 

Note that every time we have a variable $u_{i_{s+1}}^{(s+1)}\in\Sigma_{s+1,\LL}^\inn$ in the expansion of the integrals, we get a factor $\frac{1}{1-z_s}$. We also have a factor $(1-z_s)^{n_s}$ in $\mathcal{D}_{\boldsymbol{n},Y}$. Thus the surviving terms in $\mathcal{D}_{\boldsymbol{n},Y}$ are of order $O(z_s^{-n_s+n_s})=O(1)$ when $z_s$ is large. We immediately obtain~\eqref{eq:aux23} by deforming the $z_s$ contour to infinity.

%\bibliographystyle{alpha}
%\bibliography{bibliography}
\def\cydot{\leavevmode\raise.4ex\hbox{.}}

\end{document}